\documentclass[11pt]{article}
\usepackage{graphicx, times}
\usepackage{amsfonts}
\usepackage{amsmath}
\usepackage{amsthm}
\usepackage{amssymb}
\usepackage{fancyhdr}
\usepackage{indentfirst}
\usepackage{titlesec}
\usepackage{hyperref}
\usepackage{abstract}
\usepackage{stmaryrd}
\usepackage{mathcomp}
\usepackage{epsfig}
\usepackage{pst-grad} 
\usepackage{pst-plot}

\usepackage[normalem]{ulem}
\usepackage{soul}

\newcommand{\al}{\alpha}
\newcommand{\be}{\beta}
\newcommand{\V}{\mathcal{V}}
\newcommand{\R}{\mathbb{R}}
\newcommand{\N}{\mathbb{N}}
\newcommand{\mH}{\mathcal{H}}
\newcommand{\F}{\mathcal{F}}

\newcommand{\B}{\mathcal{B}}
\newcommand{\C}{\mathcal{C}}

\newcommand{\lan}{\langle}
\newcommand{\ran}{\rangle}
\newcommand{\la}{\lambda}
\newcommand{\lc}{\llcorner}
\newcommand{\ttimes}{\scalebox{0.7}{$\mathbb{X}$}}
\newcommand{\La}{\Lambda}
\newcommand{\si}{\sigma}
\newcommand{\Si}{\Sigma}
\newcommand{\de}{\delta}

\newcommand{\ep}{\epsilon}
\newcommand{\pr}{\prime}
\newcommand{\Om}{\Omega}
\newcommand{\om}{\omega}

\newcommand{\ga}{\gamma}
\newcommand{\ka}{\kappa}
\newcommand{\lap}{\triangle}
\newcommand{\ti}{\tilde}

\newcommand{\mZ}{\mathbb{Z}}
\newcommand{\Z}{\mathcal{Z}}
\newcommand{\mS}{\mathcal{S}}
\newcommand{\f}{\mathbf{f}}

\newcommand{\M}{\mathbf{M}}
\newcommand{\bL}{\mathbf{L}}
\newcommand{\bI}{\mathbf{I}}
\newcommand{\mf}{\mathbf{f}}

\newcommand{\n}{\mathbf{n}}
\newcommand{\mr}{\mathbf{r}}
\newcommand{\md}{\mathbf{d}}

\newcommand{\rom}[1]{\expandafter\romannumeral #1}
\newcommand{\Rom}[1]{\uppercase\expandafter{\romannumeral #1}}

\newcommand{\Ind}{\textrm{Ind}}

\begin{document}

\newtheorem{theorem}{Theorem}[section]
\newtheorem{proposition}[theorem]{Proposition}
\newtheorem{corollary}[theorem]{Corollary}

\newtheorem{claim}{Claim}

\newtheorem{fact}{Fact}

\theoremstyle{remark}
\newtheorem{remark}[theorem]{Remark}

\theoremstyle{definition}
\newtheorem{definition}[theorem]{Definition}

\theoremstyle{plain}
\newtheorem{lemma}[theorem]{Lemma}

\numberwithin{equation}{section}

\titleformat{\section}
   {\normalfont\bfseries\large\filcenter}
   {\arabic{section}}
   {12pt}{}
%
 \titleformat{\subsection}[runin]
   {\normalfont\bfseries}
   {\arabic{section}.\arabic{subsection}}
   {11pt}{}
%
%
%

\pagestyle{headings}
\renewcommand{\headrulewidth}{0.4pt}

\title{\textbf{Min-max hypersurface in manifold of positive Ricci curvature}}
\author{Xin Zhou\footnote{The author is partially supported by NSF grant DMS-1406337.}}
\date{\today}
\maketitle

\pdfbookmark[0]{}{beg}

\renewcommand{\abstractname}{}    
\renewcommand{\absnamepos}{empty} 
\begin{abstract}
\textbf{Abstract:} 
In this paper, we study the shape of the min-max minimal hypersurface produced by Almgren-Pitts-Schoen-Simon \cite{AF62, AF65, P81, SS81} in a Riemannian manifold $(M^{n+1}, g)$ of positive Ricci curvature for all dimensions. The min-max hypersurface has a singular set of Hausdorff codimension $7$. We characterize the Morse index, area and multiplicity of this singular min-max hypersurface. In particular, we show that the min-max hypersurface is either orientable and has Morse index one, or is a double cover of a non-orientable stable minimal hypersurface.

As an essential technical tool, we prove a stronger version of the discretization theorem. The discretization theorem, first developed by Marques-Neves in their proof of the Willmore conjecture \cite{MN12}, is a bridge to connect sweepouts appearing naturally in geometry to sweepouts used in the min-max theory. Our result removes a critical assumption of \cite{MN12}, called the no mass concentration condition, and hence confirms a conjecture by Marques-Neves in \cite{MN12}.
\end{abstract}

\section{Introduction}

Given an $(n+1)$ dimensional closed Riemannian manifold $M^{n+1}$, minimal hypersurfaces are critical points of the area functional. When $M$ has certain topology, a natural way to produce minimal hypersurface is to minimize area among its homology class. This idea leads to the famous existence and regularity theory for area minimizing hypersurfaces by De Giorgi, Federer, Fleming, Almgren and Simons etc. (c.f. \cite{FH, Gi, Si83}). In general cases, when every hypersurface is homologically trivial, e.g. if the Ricci curvature of the ambient manifold is positive, the minimization method fails. This motivates F. Almgren \cite{AF62, AF65}, followed up by J. Pitts \cite{P81}, to develop a Morse theoretical method for the area functional in the space of hypersurfaces, namely the min-max theory. The heuristic idea of developing a Morse theory is to associate a nontrivial 1-cycle in the space of hypersurfaces with a critical point of the area functional, i.e. a minimal hypersurface. In particular, denote $\Z_n(M)$ by the space of all closed hypersurfaces with a natural topology in geometric measure theory, called the flat topology. Now consider a one-parameter family $\Phi: [0, 1]\rightarrow \Z_n(M)$. Let $[\Phi]$ be the set of all maps $\Psi: [0, 1]\rightarrow \Z_n(M)$ which are homotopic to $\Phi$ in $\Z_n(M)$. The min-max value can be associated with $[\Phi]$ as
\begin{equation}\label{min-max value: simple form}
\bL([\Phi])=\inf\big\{\max_{x\in[0, 1]}Area(\Psi(x)): \Psi\in[\Phi]\big\}.
\end{equation}
Almgren \cite{AF62} showed that there is a nontrivial $\Phi$ with $\bL([\Phi])>0$ in any closed manifold $M$; together with Pitts \cite{AF65, P81}, they showed that when $2\leq n\leq 5$, there is a disjoint collection of closed, smooth, embedded, minimal hypersurfaces $\{\Si_i\}_{i=1}^l$ with integer multiplicity $k_i\in\N$ such that  $\sum_{i=1}^l k_i Area(\Si_i)=\bL([\Phi])$. Schoen and Simon \cite{SS81} extended the regularity results to $n\geq 6$. Note that for $n\geq 7$, the min-max hypersurface $\Si_i$ has a singular set of codimension $7$. Later on, there are other variations of the Almgren-Pitts min-max theory, c.f. \cite{Sm82, CD03, DT09}.

However, besides the existence and regularity, much is unknown about these min-max hypersurfaces. For instance, a natural question is how large can the area and multiplicity be? Moreover, in this Morse theoretical approach, one key open problem, raised by Almgren \cite{AF65} and emphasized by F. Marques \cite[\S 4.1]{M14} and A. Neves \cite[\S 8]{N14}, is to bound the Morse index of the min-max minimal hypersurface by the number of parameters. It is conjectured that generically the Morse index is equal to the number of parameters, and the multiplicity is one. The importance of this problem lies in several aspects. First, finding minimal hypersurfaces with bounded (or prescribed) Morse index is a central motivation for Almgren \cite{AF65} to develop the min-max theory. Also the bound of Morse index plays an important role in application to geometric problems. In his famous open problems section \cite[Problem 29 and 30]{Y}, S. T. Yau stressed the importance of the estimates of Morse index in several conjectures. In the recent celebrated proof of the Willmore conjecture by Marques and Neves \cite{MN12}, a key part is to prove that the Morse index of certain min-max minimal surface in the standard three-sphere is bounded by $5$. The major challenge of bounding the Morse index comes from the fact that the min-max hypersurface is constructed as a very weak limit (i.e. varifold limit), therefore classical methods in nonlinear analysis (c.f. \cite{St00}) do not extend to this situation. Here one difficulty of understanding the weak limit is due to the existence of multiplicity (see \cite{I95} for similar issue in studying the singularity of mean curvature flow).

The current progress of understanding the min-max hypersurfaces mainly focused on the case of one-parameter families. Marques and Neves \cite{MN11} have confirmed the Morse index conjecture in three dimension when the Ricci curvature of the ambient manifold is positive, where they proved the existence of minimal Heegaard surface of Morse index 1 in certain 3-manifolds. This was extended to manifold $M^{n+1}$ with positive Ricci curvature in dimensions $2\leq n\leq 6$, when the min-max hypersurfaces are smooth, by the author \cite{Z12}. In \cite{Z12}, we also gave a general characterization of the multiplicity, area and Morse index of the min-max hypersurface. In particular, the min-max hypesurface is either orientable of Morse index 1, or is a double cover of a non-orientable least area minimal hypersurface. Recently, the methods in \cite{MN11, Z12} were used by Mazet and Rosenberg \cite{MR15} to study the minimal hypersurfaces of least area in an arbitrary closed Riemannian manifold $M^{n+1}$ with $2\leq n\leq 6$. They gave several characterizations of the least area minimal hypersurfaces similar to \cite{Z12}.  The work in this paper will generalize the characterization of the min-max hypersurface to all dimensions, even allowing singularities. Several new ingredients are developed to deal with the presence of singularities.

\vspace{1em}
Let $(M^{n+1}, g)$ be an $(n+1)$-dimensional, connected, closed, orientable Riemannian manifold. We consider singular hypersurfaces which share the same regularity properties as the min-max hypersurfaces. To be precise, we set up some terminology. By a {\em singular hypersurface} with a singular set of Hausdorff co-dimension no less than $k$ ($k\in\N, k<n$), we mean a closed subset $\overline{\Si}$ of $M$ with finite $n$-dimensional Hausdorff measure $\mH^n(\overline{\Si})<\infty$, where the {\em regular part} of $\overline{\Si}$ is defined as:
$$reg(\Si)=\{x\in\overline{\Si}:\ \textrm{$\overline{\Si}$ is a smooth, embedded, hypersurface near $x$}\};$$
and the {\em singular part} of $\Si$ is $sing(\Si)=\overline{\Si}\backslash reg(\Si)$ (see \cite{SS81, I96}), with the $(n-k+\ep)$-dimensional Hausdorff measure $\mH^{n-k+\ep}\big(sing(\Si)\big)=0$ for all $\ep>0$. Clearly the regular part $reg(\Si)$ is an open subset of $\overline{\Si}$. Later on, we will denote $\Si=reg(\Si)$ and also call $\Si$ a singular hypersurface.
Given such a singular hypersurface $\Si$, it represents an integral varifold, denoted by $[\Si]$ (c.f. \cite[\S 15]{Si83}). We say $\Si$ is {\em minimal} if $[\Si]$ is stationary (c.f. \cite[16.4]{Si83}). In fact, this is equivalent to the fact that the mean curvature of $reg(\Si)$ is zero and the density of $[\Si]$ is finite everywhere (c.f. \cite[(3)(4)]{I96}). To simplify the presentation, in the following we simply assume that the tangent cones (c.f. \cite[\S 42]{Si83}) of $[\Si]$ have multiplicity one everywhere (which is satisfied by min-max hypersurfaces by Lemma \ref{tangent cone multiplicity one}). 
We use $\Ind(\Si)$ to denote the Morse index of $\Si$ (see \S \ref{Orientation, second variation and Morse index}). Denote
\begin{equation}\label{space of minimal surfaces}
\begin{split}
\mS= \{\Si^n:\ &\textrm{$\overline{\Si}$ is a connected, closed, minimal, hypersurface with a singular set}\\
                       &\textrm{$sing(\Si)$ of Hausdorff co-dimension no less than $7$}\}.
\end{split}
\end{equation}
Let
\begin{equation}\label{definition of A_M}
\left. A_M= \inf_{\Si\in\mS}\Big\{ \begin{array}{ll}
\mH^n(\Si), \quad \textrm{   if $\Si$ is orientable}\\
2\mH^n(\Si), \quad \textrm{ if $\Si$ is non-orientable}
\end{array} \Big\}.\right. 
\end{equation}

If the Ricci curvature of $M$ is positive, then the min-max hypersurface has only one connected component  (Theorem \ref{generalized Frankel theorem}), and we denote it by $\Si$. Our main result is as follows.
\begin{theorem}\label{main theorem1}
Assume that the Ricci curvature of $M$ is positive; then the min-max hypersurface $\Si$
\begin{itemize}
\vspace{-5pt}
\addtolength{\itemsep}{-0.7em}
\setlength{\itemindent}{0em}
\item[$(\rom{1})$] \underline{either} is orientable of multiplicity one, which has Morse index $\Ind(\Si)=1$, and $\mH^n(\Si)=A_M$;
\item[$(\rom{2})$] \underline{or} is non-orientable with multiplicity two, which is stable, i.e. $\Ind(\Si)=0$, and $2\mH^n(\Si)=A_M$.
\end{itemize}
\end{theorem}
\begin{remark}
The fact that $\mH^n(\Si)=A_M$ or $2\mH^n(\Si)=A_M$ says that the min-max hypersurface has least area among all singular minimal hypersurfaces (if counting non-orientable minimal hypersurface with multiplicity two).
\end{remark}

The main idea contains two parts. First, given a minimal hypersurface $\Si$, we will embed $\Si$ into a one parameter family $\{\Si_t\}_{t\in[-1, 1]}$ with $\Si_0=\Si$, such that the area of $\Si$ achieves a strict maximum, i.e. $Area(\Si_t)<Area(\Si)$ if $t\neq 0$. Second, we will show that all of such one parameter families obtained in this way (from a minimal hypersurface) belong to the same homotopy class. Then from the definition of the min-max value (\ref{min-max value: simple form}), the family $\{\Si_t\}$ corresponding to the min-max hypersurface $\Si$ must be optimal, i.e. $\max_t Area(\Si_t)\leq \max_t Area(\Si^{\pr}_t)$, where $\{\Si^{\pr}_t\}$ is generated by any other minimal hypersurface $\Si^{\pr}$ in the first step. The characterization of Morse index, multiplicity and area of $\Si$ will then follow from this optimality condition. Specifically, in the first part, we will choose the one parameter family as the level sets of the distance function to $\Si$. 
Note that the minimal hypersurface $\Si$ has a singular set of Hausdorff codimension $7$. To deal with the presence of singularities, we will use an idea explored by Gromov \cite{Gr} in his study of isoperimetric inequalities. To show the homotopic equivalence of these one parameter families, we need to use an isomorphism constructed by Almgren in \cite{AF62}, under which the homotopy groups of the space of hypersurfaces in $M$ are mapped isometrically to the homology groups of $M$.

One main difficulty is caused by the fact that two different topology are used on the space of hypersurfaces $\Z_n(M)$. The geometric method in the first part produces families of hypersurfaces which are continuous under the flat topology. However, the Almgren-Pitts min-max theory works under another topology, called the mass norm topology, which is much stronger than the flat topology. A bridge is desired to connect the two topology. In fact, this is a very common problem in the study of min-max theory (c.f. \cite{MN12, Z12, MN13, Mo14}). Pitts already developed some tools in his book \cite{P81}. Marques-Neves, in their proof of the Willmore conjecture \cite{MN12}, first gave a complete theory to connect families continuous under flat topology to families satisfying the requirement of the Almgren-Pitts setting (see also \cite{Z12, MN13, Mo14}). Marques-Neves need a critical technical assumption for the starting family, called no mass concentration condition, which means that there is no point mass in the measure-theoretical closure of the family. However, in our situation the one parameter family does not necessarily satisfy the no mass concentration condition due to the presence of singular set. In fact, in the same paper \cite[\S 13.2]{MN12}, Marques-Neves conjectured that this assumption might not be necessary. Here we verify this conjecture under a very general condition. As this improvement will be useful in other situation, we present it here (in a simplified form).

\begin{theorem}\label{discretization and identification: simple form}
(see Theorem \ref{discretization and identification} for a detailed version) Given a continuous (under the flat topology) one parameter family of hypersurfaces $\Phi: [0, 1]\rightarrow \Z_n(M)$, such that for each $x\in[0, 1]$, $\Phi(x)$ is represented by the boundary of some set $\Om_x\subset M$ of finite perimeter, and such that $\max_{x}Area\big(\Phi(x)\big)<\infty$, then there exists a $(1, \M)$-homotopy sequence $\{\phi_i\}$ (one parameter family in the sense of Almgren-Pitts, c.f. \S \ref{homotopy sequences}), satisfying 
$$\max_{x}Area\big(\Phi(x)\big)=\limsup_{i\rightarrow\infty}\max_x Area(\phi_i(x)).$$
\end{theorem}
\begin{remark}
The key step is to develop a new discretization procedure to connect the given family to a new family which satisfy the no mass concentration condition (see Lemma \ref{case 2} and the discussions there). Under the same condition that the hypersurfaces are represented by boundary of sets of finite perimeter, the above result is also true for multi-parameter families.
\end{remark}

The paper is organized as follows. In Section \ref{preliminary results}, we give several preliminary results concerning the topology, second variation and Morse index for singular hypersurfaces in a manifold of positive Ricci curvature. In Section \ref{min-max family}, we show that the level sets of distance function to a singular minimal hypersurface is a good one parameter family. In Section \ref{Almgren-Pitts min-max theory}, we introduce the Almgren-Pitts theory. In Section \ref{Discretization and construction of sweepouts}, we prove Theorem \ref{discretization and identification: simple form}. Finally, we prove Theorem \ref{main theorem1} in Section \ref{proof of the main theorem}.

\vspace{1em}
\noindent\textbf{Acknowledgement:} I would like to thank Prof. Bill Minicozzi to bring \cite{Gr} into my attention. I also wish to thank Prof. Richard Schoen for many helpful discussions, and thank Prof. Tobias Colding and Prof. Shing Tung Yau for their interests. Finally, I am indebted to the referee for helpful comments to clarify several presentations.


\section{Preliminary results}\label{preliminary results}

In this section, we give several preliminary results about minimal hypersurfaces with a singular set of Hausdorff dimension less to or equal than $n-7$.

\subsection{Notions of geometric measure theory.}\label{notions of geometric measure theory}
For notions in geometric measure theory, we refer to \cite{Si83} and \cite[\S 2.1]{P81}.

Fix a connected, closed, oriented Riemannian manifold $(M^{n+1}, g)$ of dimension $n+1$. Assume that $(M^{n+1}, g)$ is embedded in some $\R^{N}$ for $N$ large. We denote by
\begin{itemize}
\vspace{-5pt}
\setlength{\itemindent}{0.5em}
\addtolength{\itemsep}{-0.7em}
\item $\bI_{k}(M)$ the space of $k$-dimensional integral currents in $\R^{N}$ with support in $M$;
\item $\Z_k(A, B)$ the space of integral currents $T\in\bI_k(M)$, with $spt(T)\subset A$\footnote{$spt(T)$ denotes the support of $T$ \cite[26.11]{Si83}.} and $spt(\partial T)\subset B$\footnote{$\partial T\in \bI_{n-1}(M)$ denotes the boundary of $T$ \cite[26.3]{Si83}.}, where $A, B$ are compact subset of $M$, and $B\subset A$;
\item $\Z_{k}(M)$ the space of integral currents $T\in\bI_{k}(M)$ with $\partial T=0$;
\item $\V_{k}(M)$ the closure, in the weak topology, of the space of $k$-dimensional rectifiable varifolds in $\R^{N}$ with support in $M$;
\item $\F$ and $\M$ respectively the flat norm \cite[\S 31]{Si83} and mass norm \cite[26.4]{Si83} on $\bI_{k}(M)$;
\item $\C(M)$ the space of sets $\Om\subset M$ with finite perimeter \cite[\S 14]{Si83}\cite[\S 1.6]{Gi}.
\end{itemize}
Given $T\in\bI_{k}(M)$, $|T|$ and $\|T\|$ denote respectively the integral varifold and Radon measure in $M$ associated with $T$. $\bI_{k}(M)$ and $\Z_{k}(M)$ are in general assumed to have the flat norm topology. $\bI_{k}(M, \M)$ and $\Z_{k}(M, \M)$ are the same space endowed with the mass norm topology. Given $T\in\Z_k(M)$, $\B^{\F}_s(T)$ and $\B^{\M}_s(T)$ denote respectively balls in $\Z_k(M)$ centered at $T$, of radius $s$, under the flat norm $\F$ and the mass norm $\M$. Given a closed, orientable hypersurface $\Si$ in $M$ with a singular set of Hausdorff dimension no larger than $(n-7)$, or a set $\Om\in\C(M)$ with finite perimeter, we use $[[\Si]]$, $[[\Om]]$ to denote the corresponding integral currents with the natural orientation, and $[\Si]$, $[\Om]$ to denote the corresponding integer-multiplicity varifolds.


\subsection{Nearest point projection to $\overline{\Si}$.} Here we recall the fact that the nearest point projection of any point in $M$ to $\overline{\Si}$ (away from the singular set of $\Si$) is a regular point of $\overline{\Si}$ when $\Si$ is minimal. Similar result for isoperimetric hypersurfaces appeared in \cite{Gr}. 
\begin{lemma}\label{nearest point projection}
Let $\Si\in\mS$ be a singular minimal hypersurface in $M$. Take a point $p\in M\backslash \overline{\Si}$, and a minimizing geodesic $\ga$ connecting $p$ to $\overline{\Si}$ in $M$, i.e. $\ga(0)=p$, $\ga(1)=q\in\overline{\Si}$, and $length(\ga)=dist(p, \overline{\Si})$. Then $q$ is a regular point of $\Si$.
\end{lemma}
\begin{proof}
Take the geodesic sphere of $M$ center at $\ga(\frac{1}{2})$ with radius $\frac{1}{2}dist(p, \overline{\Si})$. The sphere is a smooth hypersurface near $q$, and $\overline{\Si}$ lies in one side of the sphere. 
So the tangent cone of $\overline{\Si}$ (viewed as a rectifiable varifold with multiplicity $1$ by assumption) at $q$ is contained in a half-space of $\R^{n+1}$ (separated by the tangent plane of the sphere). As $\overline{\Si}$ is stationary, by \cite[36.5, 36.6]{Si83}, the tangent cone of $\overline{\Si}$ at $q$ is equal to the tangent plane of the sphere (with multiplicity $1$), 
and hence $\overline{\Si}$ is smooth at $q$ by the Allard Regularity Theorem (c.f. \cite{Al72}\cite[24.2]{Si83}).
\end{proof}


\subsection{Connectedness.} For stationary hypersurface with a small singular set, the connectedness of the closure is the same as the connectedness of the regular part. In fact, this follows from the strong maximum principle for stationary singular hypersurfaces.
\begin{theorem}\label{strong maximum principle}
\cite[Theorem A]{I96}
\begin{enumerate}
\vspace{-5pt}
\addtolength{\itemsep}{-0.7em}
\setlength{\itemindent}{0em}
\item If $V_1$ and $V_2$ are stationary integer rectifiable $n$-varifods in an open subset $\Om\subset M^{n+1}$, satisfying
$$\mH^{n-2}(spt(V_1)\cap spt(V_2)\cap\Om)=0,$$
then $spt(V_1)\cap spt(V_2)\cap\Om=\emptyset$.
\item Assume that $\Si$ is a stationary hypersurface in $\Om$ with a singular set of Hausdorff dimension less than $n-2$. If $\overline{\Si}\cap\Om$ is connected, then $reg(\Si)\cap\Om$ is connected.
\end{enumerate}
\end{theorem}
\begin{remark}
By part 2, the closure of a singular hypersurface in our setting is connected if and only if the regular part is.
\end{remark}

\begin{definition}\label{connectedness}
A singular minimal hypersurface $\Si$ (with $dim\big(sing(\Si)\big)\leq n-7$) is {\em connected} if its regular part is connected.
\end{definition}


\subsection{Orientation, second variation and Morse index.}\label{Orientation, second variation and Morse index}

\begin{definition}\label{orientation}
A singular hypersurface $\Si$ is {\em orientable} (or {\em non-orientable}) if the regular part is orientable (or non-orientable).  
\end{definition}

A singular hypersurface $\Si$ is said to be {\em two-sided} if the normal bundle $\nu(\Si)$ of the regular part $\Si$ inside $M$ is trivial.
\begin{lemma}\label{orientable and two-sided}
Let $M^{n+1}$ be an $(n+1)$-dimensional, connected, closed, orientable manifold, and $\Si\subset M$ a connected, singular hypersurface with $dim\big(sing(\Si)\big)\leq n-2$, and with compact closure $\overline{\Si}$. Then $\Si$ is orientable if and only if $\Si$ is two-sided.
\end{lemma}
\begin{proof}
The tangent bundle of $M$, when restricted to $\Si$, has a splitting into the tangent bundle $T\Si$ and normal bundle $\nu(\Si)$ of $\Si$, i.e. $TM|_{\Si}=T\Si\oplus\nu(\Si)$. By \cite[Lemma 4.1]{H}, $T\Si$ is orientable if and only if $\nu(\Si)$ is orientable. By \cite[Theorem 4.3]{H}\footnote{It is not hard to see that $\Si$ is paracompact, so \cite[Theorem 4.3]{H} is applicable.}, $\nu(\Si)$ is orientable if and only if $\nu(\Si)$ is trivial.
\end{proof}

When $\Si$ is two-sided, there exists a unit normal vector field $\nu$. The {\em Jacobi operator} is
\begin{equation}\label{Jacobi operator}
L_{\Si}\phi=\lap_{\Si}\phi+\big(Ric(\nu, \nu)+|A|^2\big)\phi,
\end{equation}
where $\phi\in C^1_{c}(\Si)$, $\lap_{\Si}$ is the Laplacian operator of the induced metric on $\Si$, and $A$ is the second fundamental form of $\Si$ along $\nu$. Given an open subset $\Om$ of $\Si$ with smooth boundary $\partial \Om$, we say that $\la\in\R$ is an {\em Dirichlet eigenvalue} of $L_{\Si}$ on $\Om$ if there exists a non-zero function $\phi\in C^{\infty}_0(\Om)$ vanishing on $\partial\Om$, i.e. $\phi|_{\partial \Om}\equiv 0$, such that $L_{\Si}\phi=-\la\phi$. The {\em (Dirichlet) Morse index} of $\Om$, denoted by $\Ind_D(\Om)$, is the number of negative Dirichlet eigenvalues of $L_{\Si}$ on $\Om$ counted with multiplicity. 

When $\Si$ is non-orientable, we need to pass to the orientable double cover $\ti{\Si}$ of $\Si$.  Then there exists a unit normal vector field $\ti{\nu}$ along $\ti{\Si}$, satisfying $\ti{\nu}\circ \tau=-\ti{\nu}$, where $\tau: \ti{\Si}\rightarrow\ti{\Si}$ is the orientation-reversing involution, such that $\Si=\ti{\Si}/\{id, \tau\}$. The {\em Jacobi operator} $L_{\ti{\Si}}$ is well-defined using $\ti{\nu}$. Given an open subset $\Om\subset \Si$, and its lift-up $\ti{\Om}$ to $\ti{\Si}$, we can define the {\em Dirichlet eigenvalue} and {\em (Dirichlet) Morse index} by restricting the Jacobi operator $L_{\ti{\Si}}$ to functions $\ti{\phi}\in C^1_0(\ti{\Om})$ which are anti-symmetric under $\tau$, i.e. $\ti{\phi}\circ\tau=-\ti{\phi}$. (In this case, $\ti{\phi}\ti{\nu}$ descends to a vector field on $\Si$). We refer to \cite{Ro} for more discussions on Morse index in the non-orientable case.

\begin{definition}\label{Morse index for hypersurfaces with singularities}
The {\em Morse index} of $\Si$ is defined as,
$$\Ind(\Si)=\sup\{\Ind_D(\Om):\ \textrm{$\Om$ is any open subset of $\Si$ with smooth boundary.}\}$$
$\Si$ is called {\em stable} if $\Ind{\Si}\geq 0$, or equivalently, $\Si$ is stable in the classical sense on any compactly supported open subsets.
\end{definition}


\subsection{Positive Ricci curvature.} We need two properties for singular minimal hypersurfaces in manifolds of positive Ricci curvature. The first one says that there is no stable, two-sided, singular hypersurface with a small singular set. This generalizes an easy classical result for smooth hypersurfaces \cite[Chap 1.8]{CM11}. 
When $\Si$ is two-sided, the fact that $\Si$ is stable is equivalent to the following stability inequality,
\begin{equation}\label{stability inequality}
\int_{\Si}\big(Ric_g(\nu, \nu)+|A_{\Si}|^2\big)\varphi^2 d\mH^n \leq \int_{\Si}|\nabla \varphi|^2d\mH^n,
\end{equation}
for any $\varphi\in C^{\infty}_c(\Si)$.

\begin{lemma}\label{no stable two-sided minimal hypersurface}
\cite{S10} Assume that $(M^{n+1}, g)$ has positive Ricci curvature, i.e. $Ric_g>0$, and $\Si$ is a singular minimal hypersurface, with $\mH^{n-2}(sing(\Si))=0$. If $\Si$ is two-sided, then $\Si$ is not stable.
\end{lemma}
\begin{proof}
Suppose that $\Si$ is stable. Since $\mH^{n-2}(sing(\Si))=0$, for any $\ep>0$, we can take a countable covering $\cup_i B_{r_i}(p_i)$ of $sing(\Si)$ using geodesics balls $\{B_{r_i}(p_i)\}_{i\in\N}$ of $M$, such that
$$\sum_{i\in\N} r_i^{n-2}<\ep.$$
For each $i$, we can choose a smooth cutoff function $f_i$, such that $f_i=1$ outside $B_{2r_i}(p_i)$, $f_i=0$ inside $B_{r_i}(p_i)$, and $|\nabla f_i|\leq \frac{2}{r_i}$ inside the annulus $B_{2r_i}(p_i)\backslash B_{r_i}(p_i)$. Let $f_{\ep}$ be the minumum of all $f_i$'s (which is Lipschitz), and plug it into the stability inequality (\ref{stability inequality}),
\begin{displaymath}
\begin{split}
\int_{\Si}\big(Ric(\nu, \nu)+|A_{\Si}|^2\big)f_{\ep}^2 d\mH^n & \leq \int_{\Si}|\nabla f_{\ep}|^2 d\mH^n\\
                                                                                                          & \leq 4 \sum_{i\in\N}\int_{\Si\cap B_{2r_i}(p_i)}\frac{1}{r_i^2}d\mH^n \leq 4 \sum_{i\in\N} \frac{1}{r_i^2}\cdot C r_i^n\leq 4 C\ep.
\end{split}
\end{displaymath}
Here we used the monotonicity formula \cite[17.6]{Si83} to get the volume bound $\mH^n(\Si\cap B_{2r_i}(p_i))\leq Cr_i^n$ in the third $``\leq"$. Now let $\ep$ tend to zero, we get a contradiction to the fact that $Ric(\nu, \nu)>0$.
\end{proof}
\begin{remark}
If we only require $Ric_g\geq 0$, the above proof will show that the stable hypersurface must be smooth and totally geodesic, and the restriction of $Ric_g$ to $\Si$ is zero.
\end{remark}

\vspace{5pt}
The second property says that any two such singular minimal hypersurfaces in manifold with positive Ricci curvature must intersect, which generalizes the classical Frankel's theorem \cite{Fr66} for smooth minimal hypersurfaces.

\begin{theorem}\label{generalized Frankel theorem}
(Generalized Frankel Theorem) Assume that $(M^{n+1}, g)$ has positive Ricci curvature. Given any two connected, singular, minimal hypersurfaces $\Si$ and $\Si^{\pr}$ with singular sets of Hausdorff co-dimension no less than $2$, then $\overline{\Si}$ and $\overline{\Si^{\pr}}$ must intersect on a set of Hausdorff dimension no less than $n-2$. Theorefore, $\Si\cap\Si^{\pr}\neq \emptyset$.

\end{theorem}

\begin{proof}
First if $\overline{\Si}\cap\overline{\Si^{\pr}}=\emptyset$, then we can find two points $p\in\overline{\Si}$, $p^{\pr}\in\overline{\Si^{\pr}}$, such that $d(p, p^{\pr})=dist(\overline{\Si}, \overline{\Si^{\pr}})$. By the argument as in Lemma \ref{nearest point projection}, both $p, p^{\pr}$ are regular points of $\Si, \Si^{\pr}$. Then as in \cite[\S 2]{Fr66}, we can get a contradiction by looking at the second variational formula of the length functional along the minimizing geodesic connecting $p$ to $p^{\pr}$ when $(M, g)$ has positive Ricci curvature.

Then $\overline{\Si}\cap \overline{\Si^{\pr}}\neq \emptyset$, so Theorem \ref{strong maximum principle} implies that $\overline{\Si}\cap \overline{\Si^{\pr}}$ must have Hausdorff dimension no less than $n-2$.
\end{proof}


\subsection{Orientation and singular hypersurfaces.} 
Now we list a few properties related to the orientation of singular hypersurfaces. Similar properties for smooth hypersurfaces were discussed in \cite[\S 3]{Z12}.
\begin{proposition}\label{property of orientation}
Given a connected, minimal, singular hypersurface $\Si^{n}$ with a singular set of Hausdorff dimension less than $n-2$, then
\begin{enumerate}
\vspace{-5pt}
\addtolength{\itemsep}{-0.7em}
\setlength{\itemindent}{0em}
\item $\Si$ is orientable if and only if $\overline{\Si}$ represents an integral $n$-cycle.

\item If $\overline{\Si}$ separates $M$, i.e. $M\backslash\overline{\Si}$ contains two connected components, then $\Si$ is orientable.

\item When $M$ has positive Ricci curvature, if $\Si$ is orientable, then $\overline{\Si}$ separates M.
\end{enumerate}
\end{proposition}
\begin{proof}
Part 1. $\overline{\Si}$ is a rectifiable set, and when $\Si$ is orientable, it can represent an integer-multiplicity rectifiable current $[\overline{\Si}]$ as follows:
$$[\overline{\Si}](\om)=\int_{\Si}\lan\xi(x), \om(x)\ran d\mH^n=\int_{\Si} \om,$$
where $\xi(x)$ is the orientation form of $\Si$, and $\om$ is any smooth $n$-form on $M$. Now we will show that $[\overline{\Si}]$ is a cycle, i.e. $\partial [\overline{\Si}]=0$. Given any smooth $(n-1)$-form $\om$ on $M$, take the sequence of cutoff functions $f_{\ep}$, $\ep\rightarrow 0$, as in the proof of Lemma \ref{no stable two-sided minimal hypersurface},
\begin{displaymath}
\begin{split}
\partial [\overline{\Si}](\om) & =[\overline{\Si}](d\om)=\int_{\Si}d\om=\lim_{\ep\rightarrow 0}\int_{\Si}f_{\ep} d\om\\
                                                 & =\lim_{\ep\rightarrow 0}\int_{\Si}d(f_{\ep}\om)-df_{\ep}\wedge\om.
\end{split}
\end{displaymath}
The first term is zero by the Stokes Theorem, and the second term can be estimated as:
$$|\int_{\Si}df_{\ep}\wedge\om|\leq \int_{\Si}|df_{\ep}\wedge\om|d\mH^n\leq C\sum_{i\in\N}\int_{\Si\cap B_{r_i}(p_i)}\frac{1}{r_i}d\mH^n\leq C\sum_{i\in\N} r_i^{n-1}\rightarrow 0.$$

Now assume that $\overline{\Si}$ represents an integral cycle, and we will show that $\Si$ is orientable. In fact, assume that $[\overline{\Si}]=\lan\Si, \xi(x), \theta(x)=1\ran$ is an integral cycle, where $\xi(x)$ is locally an orientation form. Given any open subset $U\subset M\backslash sing(\Si)$, then $\partial \big([\overline{\Si}]\lc U\big)=0$ in $U$ by definition. By the same argument in \cite[Proposition 6, Claim 4]{Z12}, $[\overline{\Si}]\lc U$ represents an integral $n$-cycle in $\Si\cap U$, hence by the Constancy Theorem \cite[26.27]{Si83}, $[\overline{\Si}]\lc U=[\Si\cap U]$. Let $U$ exhausts the whole regular part $\Si$, then $[\overline{\Si}]\lc (M\backslash sing(\Si))=[\Si]$; hence the orientation of $[\overline{\Si}]$ gives a global orientation of $\Si$.

\vspace{5pt}
Part 2. The case for smooth $\overline{\Si}$ is given in \cite[\S 4 Theorem 4.5]{H}. Now we modify the proof to our case. Take a connected component $U$ of $M\backslash\overline{\Si}$, the (topological) boundary $\partial U$ of $U$ is then a closed subset of $\overline{\Si}$. By using local coordinate charts of $(M, \overline{\Si})$ around any smooth point of $\Si$, it is easy to see that $\partial U\cap \Si$ is a open subset of $\Si$. Hence as a subset of $\Si$, $\partial U\cap \Si$ is both open and closed, so $\partial U\cap\Si=\Si$ since $\Si$ is connected, and then $\partial U=\overline{\Si}$. Using the same argument as in \cite[Theorem 4.2]{H}, the orientation of $U$ induces an orientation for the normal bundle $N$ of the regular part of $\partial U$, i.e. $\Si$. Note the splitting of the tangent bundle $TM$ restricted on $\Si$: $TM|_{\Si}=T\Si\bigoplus N$; hence $T\Si$ is orientable by \cite[Lemma 4.1]{H}.

\vspace{5pt}
Part 3. By Part 1, $\overline{\Si}$ represents an integral cycle $[\overline{\Si}]$, hence it represents an integral homology class $\big[ [\overline{\Si}] \big]$ in $H_n(M, \mZ)$ \cite[4.4.1]{FH}. If $\overline{\Si}$ does not separate $M$, i.e. $M\backslash\overline{\Si}$ is connected, we claim that $\big[ [\overline{\Si}] \big]$ is non-trivial in $H_n(M, \mZ)$. In fact, if $\big[ [\overline{\Si}] \big]=0$, then there exists an integral $(n+1)$-current $C\in\bI_{n+1}(M, \mZ)$, such that $\partial C=[\overline{\Si}]$. Given any connected open subset $U\subset M\backslash\overline{\Si}$, then $\partial (C\lc U)=0$ in $U$ by definition. The Constancy Theorem \cite[26.27]{Si83} implies that $C\lc U=m[U]$, for some $m\in\mZ$, where $[U]$ denotes the integral $(n+1)$-current represented by $U$. As $M\backslash\overline{\Si}$ is connected ($\overline{\Si}$ does not separate $M$), we can take $U=M\backslash\overline{\Si}$, and hence $C\lc (M\backslash\overline{\Si})=m[M\backslash\overline{\Si}]$. As $\overline{\Si}$ has zero $(n+1)$-dimensional Hausdorff measure, then $C=m[M]$, hence $\partial C=m\partial[M]=0$, which is a contradiction to the fact that $\partial C=[\overline{\Si}]$.

Now we can take the mass minimizer $T_0\in\big[ [\overline{\Si}] \big]$ inside the homology class \cite[4.4.4]{FH}\cite[34.3]{Si83}. The codimension one regularity theory (\cite[Theorem 37.7]{Si83}) says that $T_{0}$ is represented by a minimal hypersurface $\overline{\Si}_{0}$ (possibly with multiplicity) with a singular set of Hausdorff dimension no larger than $n-7$, i.e. $T_{0}=m[\overline{\Si}_{0}]$, where $m\in\mZ$, $m\neq 0$. Since $m[\overline{\Si}_{0}]$ represents a nontrivial integral homology class, $\Si_{0}$ is orientable by Part 1. 
Hence $\Si_{0}$ is two-sided by Lemma \ref{orientable and two-sided}. By the nature of mass minimizing property of $T$, $\Si_{0}$ must be locally volume minimizing, and hence $\Si_{0}$ is stable, contradicting the positive Ricci curvature condition via Lemma \ref{no stable two-sided minimal hypersurface}.
\end{proof}


\section{Min-max family}\label{min-max family}

In this section, by using the volume comparison result in \cite{HK}, we show that every singular minimal hypersurface in a manifold with positive Ricci curvature lies in a nice ``mountain-pass" type family. In particular, the family sweeps out the whole manifold, and the area of the minimal hypersurface (when it is orientable), or the area of its double cover (when the hypersurface is non-orientable) achieves a strict maximum among the family. Actually, in manifold with positive Ricci curvature, the level sets of distance function towards the singular minimal hypersurface will play the role.


\subsection{A volume comparison result in \cite{HK}.}

Let $(M^{n+1}, g)$ be a closed, oriented manifold. Given a singular minimal hypersurface $\Si\in\mS$, denote $\nu(\Si)$ by the normal bundle of the regular part $\Si$ in $M$. Let $exp_{\nu}: \nu(\Si)\rightarrow M$ be the normal exponential map. 
Given $\xi\in\nu(\Si)$, the {\em focal distance} in the direction of $\xi$ means the first time $t>0$ such that the derivative of the normal exponential map at $t\xi$, i.e. $dexp_{\nu}(t\xi)$, becomes degenerate. Denote $\Om$ by the sets of all vectors $\xi$ in $\nu(\Si)$, which is no longer than the diameter of $M$ or the focal distance in the direction of $\xi$,
\begin{lemma}
$exp_{\nu}:\Om\rightarrow M\backslash sing(\Si)$ is surjective.
\end{lemma}
\begin{proof}
Any point $x\in M\backslash\overline{\Si}$ can be connected to $\overline{\Si}$ by a minimizing geodesic. Also by Lemma \ref{nearest point projection}, the nearest point of $x$ in $\overline{\Si}$ is a regular point of $\Si$; then the minimizing geodesic meets $\Si$ orthogonally, and hence $exp_{\nu}$ is surjective to $M\backslash sing(\Si)$. Moreover, if $\xi$ is the tangent vector of the minimizing geodesic (parametrized on $[0, 1]$) connecting $x$ to $\Si$, then the length of $\xi$ is no more than the focal distance in the direction of $\xi$.
\end{proof}
Now we will introduce a Riemannian metric on $\nu(\Si)$ (see also \cite[\S 3]{HK}), such that $\nu(\Si)$ is locally isomorphic to the product of $\Si$ with the fiber. Let $\pi: \nu(\Si)\rightarrow \Si$ be the projection map. Denote $D$ by the Riemannian connection of $M$, and $D^{\perp}$ the normal connection of $\nu(\Si)$. 
The tangent bundle of $\nu(\Si)$ can be split as a sum of ``vertical" and ``horizontal" sub-bundles $T \nu(\Si)=V+H$ as follows. Given $\xi\in\nu(\Si)$, the vertical tangent space $V_{\xi}$ contains tangent vectors of $\nu(\Si)$ which are tangent to the fibers and hence killed by $\pi_{*}$, so $V_{\xi}$ is canonically isometric to the fiber space $\nu_{\pi(\xi)}(\Si)$. 
The horizontal tangent space $H_{\xi}$ contains tangent vectors of $\nu(\Si)$ which are tangent to $D^{\perp}$-parallel curves---viewed as vector fields along their base curves (projected to $\Si$ by $\pi$), so $H_{\xi}$ is canonically isometric to $T_{\pi(\xi)}\Si$ under $\pi_{*}$. The metric on $\nu(\Si)$ can be defined as:
$$\|v\|^2=\|\pi_{*}v\|^2+\|v_{ver}\|^2,\quad v\in T_{\xi}\nu(\Si),$$
where $v_{ver}$ denotes the vertical component of $v$. It is easily seen that under this metric, $\nu(\Si)$ is locally isometric to the product of $\Si$ with the fibers. 

We need the following estimate of the volume form along normal geodesics by \cite[\S 3]{HK}. Fix $p\in\Si$ and a normal vector $\xi\in\nu_p(\Si)$. Given an orthonormal basis $e_1, \cdots, e_n$ of $T_p\Si$, they can be lifted up to $T\nu(\Si)$ as horizontal vector fields $u_1(s), \cdots, u_n(s)$ along the normal vectors $s\xi$. By our construction above, $u_1(s), \cdots, u_n(s)$ form an orthonormal basis of $T_{s\xi}\nu(\Si)$, as $\pi_*(u_i(s))=e_i$. The distortion of the $n$-dimensional volume element under the normal exponential map $exp_{\nu}: T\nu(\Si)\rightarrow M$ is given by $\|dexp_{\nu}u_1(s)\wedge\cdots\wedge dexp_{\nu}u_n(s)\|$. Assume that the Ricci curvature of $(M, g)$ satisfies $Ric_g\geq n\La$ for some $\La>0$. 
Consider an $(n+1)$-dimensional manifold $\ti{M}$ of constant curvature $\La$, and a totally geodesic hypersurface $\ti{\Si}$. Fix an arbitrary point $\ti{p}\in\ti{\Si}$, with a unit normal $\nu(\ti{p})$. Choose an orthonormal basis $\ti{e}_1, \cdots, \ti{e}_n$ of $T_{\ti{p}}(\ti{\Si})$, and a frame $\ti{u}_1(s), \cdots, \ti{u}_n(s)$ along $s\nu(\ti{p})$ constructed as above. We have the following comparison estimates:
\begin{lemma}\label{normal comparison lemma}
\cite[\S 3.2.1, Case (d)]{HK}. Let $s_0$ be no larger than the first focal distance of $\Si$ in the direction of $\xi$, then for $0\leq s\leq s_0$,
$$\|dexp_{\nu}u_1(s)\wedge\cdots\wedge dexp_{\nu}u_n(s)\|\leq \|dexp_{\nu}\ti{u}_1(s)\wedge\cdots\wedge dexp_{\nu}\ti{u}_n(s)\|.$$
\end{lemma}

It is easy to calculate that the $n$-dimensional volume distortion of the constant curvature manifold $\ti{M}$ is given by $\|dexp_{\nu}\ti{u}_1(s)\wedge\cdots\wedge dexp_{\nu}\ti{u}_n(s)\|=\cos^n(\sqrt{\La}s)\|dexp_{\nu}\ti{u}_1(0)\wedge\cdots\wedge dexp_{\nu}\ti{u}_n(0)\|=\cos^n(\sqrt{\La}s)$.
\begin{corollary}\label{distortion of normal exponential map}
Under the above setting,
$$\|dexp_{\nu}u_1(s)\wedge\cdots\wedge dexp_{\nu}u_n(s)\|\leq \cos^n(\sqrt{\La}s).$$
\end{corollary}


\subsection{Orientable case.}
Let $\Si\in\mS$ be orientable, then $\Si$ is two-sided. Denote $\nu$ by the unit normal vector field along $\Si$. When $Ric_g>0$, $\overline{\Si}$ separates $M$ by Proposition \ref{property of orientation}, i.e. $M\backslash\overline{\Si}=M_1\cup M_2$. Now the signed distance function $d^{\Si}_{\pm}$ is well-defined by
\begin{equation}\label{signed distance}
\left. d^{\Si}_{\pm}(x)= \Bigg\{ \begin{array}{ll}
dist(x, \overline{\Si}), \quad \ \textrm{ if $x\in M_1$}\\
-dist(x, \overline{\Si}), \quad \textrm{ if $x\in M_2$}\\
0,\quad\quad\quad \textrm{ if $x\in\overline{\Si}$}
\end{array} .\right. 
\end{equation}
Consider the levels sets of the signed distance function: $\Si_t=\{x\in M:\ d^{\Si}_{\pm}(x)=t\}$ for $-d(M)\leq t\leq d(M)$. Denote
\begin{equation}\label{S+}
\mS_{+}=\{\Si^{n}\in\mS:\ \Si^{n} \textrm{ is orientable}\}.
\end{equation}
We collect several properties of the distance family as follows:
\begin{proposition}\label{property of level surface 1}
Assume that $Ric_g>0$. For any $\Si\in\mS_{+}$, the distance family $\{\Si_{t}\}_{t\in[-d(M), d(M)]}$ satisfy that:
\begin{itemize}
\vspace{-5pt}
\setlength{\itemindent}{1em}
\addtolength{\itemsep}{-0.7em}
\item[$(a)$] $\Si_{0}=\Si$;
\item[$(b)$] $\mH^{n}(\Si_{t})\leq \mH^n(\Si)$, with equality only if $t=0$;
\item[$(c)$] For any open set $U\subset M\backslash sing(\Si)$ with compact closure $\overline{U}$, 
$\{\Si_{t}\lc U\}_{t\in[-\ep, \ep]}$ forms a smooth foliation of a neighborhood of $\Si$ in $U$, i.e. 
$$\Si_{t}\lc U=\{exp_{\nu}\big(t\nu(x)\big):\ x\in\Si\cap U\},\quad t\in [-\ep, \ep].$$
\vspace{-2em}
\end{itemize}
\end{proposition}

\begin{proof}
(a) is trivial by construction.

To prove (b), consider the height-$t$ section $S_t(\Si)=\{\xi\in\nu(\Si): \xi=t\nu\}$ of $\nu(\Si)$ for $-d(M)\leq t\leq d(M)$.
\begin{lemma}\label{height-t section isomorphic to Si}
Under the canonical metric of $\nu(\Si)$, $S_t(\Si)$ is isometric to $\Si$.
\end{lemma}
\begin{proof}
First, it is easy to see that the projection map $\pi: \nu(\Si)\rightarrow\Si$ restricts to be a one to one map $\pi:S_t(\Si)\rightarrow\Si$. Also the tangent plane $T_{\xi}S_t(\Si)$ of $S_t(\Si)$ at $\xi=t\nu$ consists all horizontal vectors of $T_{\xi}\nu(\Si)$. 
Then $\pi_{*}: T_{\xi}S_t(\Si)\rightarrow T_{\pi(\xi)}\Si$ gives the isometry by the construction of the metric on $\nu(\Si)$.
\end{proof}
\noindent Recall that $exp_{\nu}: \Om\subset\nu(\Si)\rightarrow M\backslash sing(\Si)$ is surjective, 
so the pre-image $exp_{\nu}^{-1}(\Si_t)$ is totally contained in $S_t(\Si)\cap\Om$, and hence by Corollary \ref{distortion of normal exponential map},
\begin{equation}\label{volume estimate orientable case}
\begin{split}
\mH^n(\Si_t) & \leq \int_{S_t(\Si)\cap\Om} \|(dexp_{\nu})_* dvol_{S_t(\Si)}\|=\int_{S_t(\Si)\cap\Om}\|dexp_{\nu}u_1(s)\wedge\cdots\wedge dexp_{\nu}u_n(s)\|\\
                     & \leq \int_{\Si}\cos^n(\sqrt{\La}t)d\mH^n\leq \cos^n(\sqrt{\La}t)\mH^n(\Si).
\end{split}
\end{equation}

To prove (c), we first realize that $\nu(\Si)$ is globally isometric to $\Si\times\R$ when $\Si$ is orientable, so that $\nu(\Si)$ has a global smooth foliation structure. When restricted to the zero section, the normal exponential map $exp_{\nu}: \nu(\Si)\rightarrow M$ is the identity map, and has non-degenerate tangent map. As the closure $\overline{U}$ is a compact subset of $M\backslash sing(\Si)$, we can use the Inverse Function Theorem to infer that $exp_{\nu}$ is a diffeomorphism in a small neighborhood of $exp_{\nu}^{-1}(\Si\cap U)$. Hence (c) follows.
\end{proof}

\subsection{Non-orientable case.}
Given $\Si\in\mS$ non-orientable, $\overline{\Si}$ does not separate $M$ by Proposition \ref{property of orientation}. Denote $d^{\Si}(x)=dist(x, \overline{\Si})$ by the distance function (without sign). Consider the level sets of $d^{\Si}$: $\Si_t=\{x\in M:\ d^{\Si}(x)=t\}$ for $0\leq t\leq d(M)$. Denote
\begin{equation}\label{S-}
\mS_{-}=\{\Si^{n}\in\mS:\ \Si^{n} \textrm{ is non-orientable}\}.
\end{equation}
We have:
\begin{proposition}\label{property of level surface 2}
Assume that $Ric_g>0$. For any $\Si\in\mS_{-}$, the distance family $\{\Si_t\}_{0\leq 0\leq d(M)}$ satisfy that:
\begin{itemize}
\vspace{-5pt}
\setlength{\itemindent}{1em}
\addtolength{\itemsep}{-0.7em}
\item[$(a)$] $\Si_{0}=\Si$;
\item[$(b)$] $\mH^{n}(\Si_{t})<2\mH^n(\Si)$, for all $0\leq t\leq d(M)$;
\item[$(c)$] When $t\rightarrow 0$, $\mH^n(\Si_t)\rightarrow 2\mH^n(\Si)$, and $\Si_t$ converge smoothly to a double cover of $\Si$ in any open set $U\subset M\backslash sing(\Si)$ with compact closure $\overline{U}$.
\vspace{-5pt}
\end{itemize}
\end{proposition}
\begin{proof}
(a) is by construction.

For (b), let the height-$t$ section of $\nu(\Si)$ be $\ti{S}_t(\Si)=\{\xi\in\nu(\Si): |\xi|=t\}$ for $0\leq t\leq d(M)$. Similar as the proof of Lemma \ref{height-t section isomorphic to Si}, the projection map $\pi: \ti{S}_t(\Si)\rightarrow\Si$ is locally isometric. Also as the fiber of $\nu(\Si)$ is one dimensional, $\pi$ is a 2-to-1 map. Hence $\pi: \ti{S}_t(\Si)\rightarrow\Si$ is an isometric double cover. The pre-image of the exponential map $exp_{\nu}^{-1}(\Si_t)$ is then contained in $\ti{S}_t\cap \Om$, with $\Om$ as above. By the volume comparison estimates in (\ref{volume estimate orientable case}),
\begin{equation}\label{volume estimate non-orientable case}
\mH^n(\Si_t)\leq \int_{\ti{S}_t(\Si)\cap\Om} \|(dexp_{\nu})_* dvol_{\ti{S}_t(\Si)}\|\leq 2\int_{\Si}\cos^n(\sqrt{\La}t)d\mH^n\leq 2\cos^n(\sqrt{\La}t)\mH^n(\Si).
\end{equation}

For (c), to prove that $\mH^n(\Si_t)\rightarrow 2\mH^n(\Si)$, as $t\rightarrow 0$, by (\ref{volume estimate non-orientable case}), we only need to prove that $\lim_{t\rightarrow 0}\mH^n(\Si_t)\geq 2\mH^n(\Si)$, and this follows from the smooth convergence $\Si_t\rightarrow 2\Si$ on any open set $U\subset\subset M\backslash sing(\Si)$. By similar argument as Proposition \ref{property of level surface 1}(c), when restricted to a small neighborhood of $exp_{\nu}^{-1}(\Si\cap U)$, $exp_{\nu}: \nu(\Si)\rightarrow M$ is a diffeomorphism. Therefore, the convergence $\Si_t\rightarrow 2\Si$ on $U$ 
follows from the fact that $\ti{S}_t(\Si)$ converge smoothly to a double cover of the zero section, as $t\rightarrow 0$. 
\end{proof}


\section{Almgren-Pitts min-max theory}\label{Almgren-Pitts min-max theory}

In this section, we will introduce the min-max theory developed by Almgren and Pitts \cite{AF62, AF65, P81}. We will mainly follow \cite[\S 4]{Z12} \cite[4.1]{P81} and \cite[\S 7 and \S 8]{MN12}. We refer to \S \ref{notions of geometric measure theory} for the notions of Geometric Measure Theory. At the end of this section, we will recall the characterization of the orientation structure of the min-max hypersurfaces proved by the author in \cite{Z12}.


\subsection{Homotopy sequences.}\label{homotopy sequences}

\begin{definition}\label{cell complex}
(cell complex.)
\begin{enumerate}
\vspace{-5pt}
\addtolength{\itemsep}{-0.7em}
\item For $m\in\N$, $I^m=[0, 1]^m$, $I^m_0=\partial I^m=I^n\backslash (0, 1)^m$;
\item For $j\in\N$, $I(1, j)$ is the cell complex of $I$, whose $1$-cells are all intervals of form $[\frac{i}{3^{j}}, \frac{i+1}{3^{j}}]$, and $0$-cells are all points $[\frac{i}{3^{j}}]$; 
$I(m, j)=I(1, j)\otimes \cdots \otimes I(1, j)$ ($m$ times) is a cell complex on $I^m$;
\item For $p\in\N$, $p\leq m$, $\al=\al_1\otimes\cdots\otimes\al_m$ is a $p$-cell if for each $i$, $\al_i$ is a cell of $I(1, j)$, and $\sum_{i=1}^mdim(\al_i)=p$. $0$-cell is called a vertex;
\item $I(m, j)_p$ denotes the set of all $p$-cells in $I(m, j)$, and $I_0(m, j)_p$ denotes the set of $p$-cells of $I(m, j)$ supported on $I^m_0$;
\item Given a $p$-cell $\al\in I(m, j)_p$, and $k\in\N$, $\al(k)$ denotes the $p$-dimensional sub-complex of $I(m, j+k)$ formed by all cells contained in $\al$. For $q\in\N$, $q\leq p$, $\al(k)_q$ and $\al_0(k)_q$ denote respectively the set of all $q$-cells of $I(m, j+k)$ contained in $\al$, or in the boundary of $\al$;
\item $T(m, j)=I(m-1, j)\otimes\{[1]\}$, $B(m, j)=I(m-1, j)\otimes\{[0]\}$ and $S(m, j)=I_0(m-1, j)\otimes I(1, j)$ denote the top, bottom and side sub-complexes of $I(m, j)$ respectively;
\item The boundary homeomorphism $\partial: I(m, j)\rightarrow I(m, j)$ is given by
$$\partial(\al_1\otimes\cdots\otimes\al_m)=\sum_{i=1}^m(-1)^{\si(i)}\al_1\otimes\cdots\otimes \partial \al_i\otimes\cdots\otimes\al_m,$$
where $\si(i)=\sum_{l<i}dim(\al_l)$, $\partial[a, b]=[b]-[a]$ if $[a, b]\in I(1, j)_1$, and $\partial[a]=0$ if $[a]\in I(1, j)_0$;
\item The distance function $\md: I(m, j)_0\times I(m, j)_0\rightarrow\N$ is defined as $\md(x, y)=3^{j}\sum_{i=1}^m|x_i-y_i|$;
\item The map $\n(i, j): I(m, i)_{0}\rightarrow I(m, j)_{0}$ is defined as: $\n(i, j)(x)\in I(m, j)_{0}$ is the unique element of $I(m, j)_0$, such that $\md\big(x, \n(i, j)(x)\big)=\inf\big\{\md(x, y): y\in I(m, j)_{0}\big\}$.
\end{enumerate}
\end{definition}

As we are mainly interested in applying the Almgren-Pitts theory to the 1-parameter families, in the following of this section, our notions will be restricted to the case $m=1$.

Consider a map to the space of integral cycles: $\phi: I(1, j)_{0}\rightarrow\Z_{n}(M^{n+1})$. The \emph{fineness of $\phi$} is defined as:
\begin{equation}\label{fineness}
\mf(\phi)=\sup\Big\{\frac{\M\big(\phi(x)-\phi(y)\big)}{\md(x, y)}:\ x, y\in I(1, j)_{0}, x\neq y\Big\}.
\end{equation}
$\phi: I(1, j)_{0}\rightarrow\big(\Z_{n}(M^{n+1}), \{0\}\big)$ denotes a map such that $\phi\big(I(1, j)_{0}\big)\subset\Z_{n}(M^{n+1})$ and $\phi|_{I_{0}(1, j)_{0}}=0$, i.e. $\phi([0])=\phi([1])=0$.

\begin{definition}\label{homotopy for maps}
Given $\de>0$ and $\phi_{i}: I(1, k_{i})_{0}\rightarrow\big(\Z_{n}(M^{n+1}), \{0\}\big)$, $i=1, 2$, we say \emph{$\phi_{1}$ is $1$-homotopic to $\phi_{2}$ in $\big(\Z_{n}(M^{n+1}), \{0\}\big)$ with fineness $\de$}, if $\exists\ k_{3}\in\N$, $k_{3}\geq\max\{k_{1}, k_{2}\}$, and
$$\psi: I(1, k_{3})_{0}\times I(1, k_{3})_{0}\rightarrow \Z_{n}(M^{n+1}),$$
such that
\begin{itemize}
\vspace{-5pt}
\setlength{\itemindent}{1em}
\addtolength{\itemsep}{-0.7em}
\item $\mf(\psi)\leq \de$;
\item $\psi([i-1], x)=\phi_{i}\big(\n(k_{3}, k_{i})(x)\big)$, $i=1, 2$;
\item $\psi\big(I(1, k_{3})_{0}\times I_{0}(1, k_{3})_{0}\big)=0$.
\end{itemize}
\end{definition}

\begin{definition}\label{(1, M) homotopy sequence}
A \emph{$(1, \M)$-homotopy sequence of mappings into $\big(\Z_{n}(M^{n+1}), \{0\}\big)$} is a sequence of mappings $\{\phi_{i}\}_{i\in\N}$,
$$\phi_{i}: I(1, k_{i})_{0}\rightarrow\big(\Z_{n}(M^{n+1}), \{0\}\big),$$
such that $\phi_{i}$ is $1$-homotopic to $\phi_{i+1}$ in $\big(\Z_{n}(M^{n+1}), \{0\}\big)$ with fineness $\de_{i}$, and
\begin{itemize}
\vspace{-5pt}
\setlength{\itemindent}{1em}
\addtolength{\itemsep}{-0.7em}
\item $\lim_{i\rightarrow\infty}\de_{i}=0$;
\item $\sup_{i}\big\{\M(\phi_{i}(x)):\ x\in I(1, k_{i})_{0}\big\}<+\infty$.
\end{itemize}
\end{definition}

\begin{definition}\label{homotopy for sequences}
Given two $(1, \M)$-homotopy sequences of mappings $S_{1}=\{\phi^{1}_{i}\}_{i\in\N}$ and $S_{2}=\{\phi^{2}_{i}\}_{i\in\N}$ into $\big(\Z_{n}(M^{n+1}), \{0\}\big)$, \emph{$S_{1}$ is homotopic with $S_{2}$} if $\exists\ \{\de_{i}\}_{i\in\N}$, such that
\begin{itemize}
\vspace{-5pt}
\setlength{\itemindent}{1em}
\addtolength{\itemsep}{-0.7em}
\item $\phi^{1}_{i}$ is $1$-homotopic to $\phi^{2}_{i}$ in $\big(\Z_{n}(M^{n+1}), \{0\}\big)$ with fineness $\de_{i}$;
\item $\lim_{i\rightarrow \infty}\de_{i}=0$.
\end{itemize}
\end{definition}

The relation ``is homotopic with" is an equivalent relation on the space of $(1, \M)$-homotopy sequences of mapping into $\big(\Z_{n}(M^{n+1}), \{0\}\big)$ (see \cite[\S 4.1.2]{P81}). An equivalent class is a \emph{$(1, \M)$ homotopy class of mappings into $\big(\Z_{n}(M^{n+1}), \{0\}\big)$}. Denote the set of all equivalent classes by $\pi^{\#}_{1}\big(\Z_{n}(M^{n+1}, \M), \{0\}\big)$. Similarly we can define the $(1, \F)$-homotopy class (using another fineness associated with the $\F$-norm in place of the $\M$-norm in (\ref{fineness})), and denote the set of all equivalent classes by $\pi^{\#}_{1}\big(\Z_{n}(M^{n+1}, \F), \{0\}\big)$.


\subsection{Almgren's isomorphism.}\label{Almgren isomorphism}

Almgren \cite{AF62} showed that the homotopy groups of $\Z_n(M)$ (under $\M$ and $\F$ topology) are all isomorphic to the top homology group of $M$ by constructing an isomorphism as follows.

By \cite[Corollary 1.14]{AF62}, there exists a small number $\nu_M>0$ (depending only on $M$), such that for any two $n$-cycles $T_1, T_2\in \Z_n(M^{n+1})$, if $\F(T_2-T_1)\leq \nu_M$, then there exists an $(n+1)$-dimensional integral current $Q\in \bI_{n+1}(M)$, with $\partial Q=T_2-T_1$, and $\M(Q)=\F(T_2-T_1)$. $Q$ is called the {\em isoperimetric choice} for $T_2-T_1$.

Given $\phi: I(1, k)_0\rightarrow \Z_n(M^{n+1})$, with $\f(\phi)\leq \de\leq \nu_M$, then for any $1$-cell $\al\in I(1, k)_1$, with $\al=[t^1_{\al}, t^2_{\al}]$, $\F\big(\phi(t^1_{\al})-\phi(t^2_{\al})\big)\leq \M\big(\phi(t^1_{\al})-\phi(t^2_{\al})\big)\leq \f(\phi) \leq \nu_M$. So there exists an isoperimetric choice $Q_{\al}\in \bI_{n+1}(M^{n+1})$, with 
$$\M(Q_{\al})=\F\big(\phi(t^1_{\al})-\phi(t^2_{\al})\big),\ \textrm{and}\ \partial Q_{\al}=\phi(t^2_{\al})-\phi(t^1_{\al}).$$
Now the sum of the isoperimetric choices for all $1$-cells is an $(n+1)$-dimensional integral current, i.e. $\sum_{\al\in I(1, k)_1}Q_{\al}\in \bI_{n+1}(M^{n+1})$.
We call the map:
\begin{equation}\label{Almgren isomorphism1}
F_A: \phi\rightarrow \sum_{\al\in I(1, k)_1}Q_{\al}
\end{equation}
{\em Almgren's isomorphism} (the name comes from Theorem \ref{isomorphism}).

Given a $(1, \M)$-homotopy sequence of mappings $S=\{\phi_i\}_{i\in\N}$ into $\big(\Z_{n}(M^{n+1}), \{0\}\big)$, take $i$ large enough, and $\phi_i: I(1, k_i)_0\rightarrow \big(\Z_{n}(M^{n+1}), \{0\}\big)$, such that $\f(\phi_i)\leq \de_i\leq \nu_M$. Then
$$F_A(\phi_i)=\sum_{\al\in I(1, k_i)_1}Q_{\al}$$
is an $(n+1)$-dimensional integral cycle as $\phi_i([0])=\phi_i([1])=0$, and represents an $(n+1)$-dimensional integral homology class
$$\big[\sum_{\al\in I(1, k_i)_1}Q_{\al}\big]\in H_{n+1}(M^{n+1}).$$
Moreover, Almgren \cite[\S 3.2]{AF62} showed that this homology class depends only on the homotopy class of $\{\phi_i\}$. Hence it reduces to a map
$$F_A: \pi^{\#}_{1}\big(\Z_{n}(M^{n+1}, \M), \{0\}\big)\rightarrow H_{n+1}(M^{n+1}),$$
defined in \cite[\S 3.2]{AF62} as:
\begin{equation}\label{Almgren isomorphism2}
F_A: [\{\phi_i\}_{i\in\N}] \rightarrow \big[\sum_{\al\in I(1, k_i)_1}Q_{\al}\big].
\end{equation}
Almgren also proved that this mapping is an isomorphism. 

\begin{theorem}\label{isomorphism}
\emph{(\cite[Theorem 13.4]{AF62} and \cite[Theorem 4.6]{P81})} The followings are all isomorphic under $F_A$:
$$H_{n+1}(M^{n+1}),\ \pi^{\#}_{1}\big(\Z_{n}(M^{n+1}, \M), \{0\}\big),\ \pi^{\#}_{1}\big(\Z_{n}(M^{n+1}, \F), \{0\}\big).$$
\end{theorem}
\noindent We also call this map {\em Almgren's isomorphism}.


\subsection{Existence of min-max hypersurface.}

\begin{definition}
(Min-max definition) Given $\Pi\in\pi^{\#}_{1}\big(\Z_{n}(M^{n+1}, \M), \{0\}\big)$, define:
$$\bL: \Pi\rightarrow\R^{+}$$
as a function given by:
$$\bL(S)=\bL(\{\phi_{i}\}_{i\in\N})=\limsup_{i\rightarrow\infty}\max\big\{\M\big(\phi_{i}(x)\big):\ x \textrm{ lies in the domain of $\phi_{i}$}\big\}.$$
The \emph{width of $\Pi$} is defined as
\begin{equation}\label{width}
\bL(\Pi)=\inf\{\bL(S):\ S\in\Pi\}.
\end{equation}
$S\in\Pi$ is call a \emph{critical sequence}, if $\bL(S)=\bL(\Pi)$. Let $K: \Pi\rightarrow\{\textrm{compact subsets of $\V_{n}(M^{n+1})$}\}$ be defined by
$$K(S)=\{V:\ V=\lim_{j\rightarrow\infty}|\phi_{i_{j}}(x_{j})|:\ \textrm{$x_{j}$ lies in the domain of $\phi_{i_{j}}$}\}.$$
A \emph{critical set} of $S$ is $C(S)=K(S)\cap\{V:\ \M(V)=\bL(S)\}$.
\end{definition}

The celebrated min-max theorem of Almgren-Pitts (Theorem 4.3, 4.10, 7.12, Corollary 4.7 in \cite{P81}) and Schoen-Simon (for $n\geq 6$ \cite[Theorem 4]{SS81}) is as follows.
\begin{theorem}\label{AP min-max theorem}
Given a nontrivial $\Pi\in\pi^{\#}_{1}\big(\Z_{n}(M^{n+1}, \M), \{0\}\big)$, then $\bL(\Pi)>0$, and there exists a stationary integral varifold $V$, whose support is a disjoint collection of connected, closed, singular, minimal hypersurfaces $\{\Si_i\}_{i=1}^l$, with singular sets of Hausdorff dimension no larger than $n-7$, (which may have multiplicity, say $m_i$), such that $V=\sum_{i=1}^lm_i[\Si_i]$, and
$$\sum_{i=1}^lm_i\mH^n(\Si_i)=\bL(\Pi).$$
In particular, $V$ lies in the critical set $C(S)$ of some critical sequence $S\in\Pi$.
\end{theorem}


\subsection{Orientation and multiplicity.}
As $V$ lies in the critical set $C(S)$, $V$ is a varifold limit of a sequence of integral cycles $\{\phi_{i_j}(x_j)\}_{j\in\N}$. It has been conjectured that $V$ should inherit some orientation structures from $\{\phi_{i_j}(x_j)\}_{j\in\N}$. In fact, we verified this conjecture and gave a characterization of the orientation structure of $V$ in low dimensions (where the support of $V$ is the smooth) in \cite[Proposition 6.1]{Z12}. Some straightforward modifications of the proof will give similar characterization for singular min-max hypersurfaces (in all dimensions) as follows.

\begin{proposition}\label{orientation and multiplicity result}
Let $V$ be the stationary varifold in Theorem \ref{AP min-max theorem}, with $V=\sum_{i=1}^{l}m_i [\Si_{i}]$. 
If $\Si_i$ is non-orientable, then the multiplicity $m_i$ must be an even number.
\end{proposition}
\begin{remark}
When a connected component $\Si_i$ is orientable, it represents an integral cycle by Proposition \ref{property of orientation}. While a connected component $\Si_i$ is non-orientable, an even multiple of it also represents an integral cycle---a zero cycle. This result will play a key role in the characterization of the multiplicity in Theorem \ref{main theorem1}. (This result was also used in \cite{Z12, MR15} to characterize the multiplicity of min-max hypersurfaces).
\end{remark}


\section{Discretization and construction of sweepouts}\label{Discretization and construction of sweepouts}

The purpose of this section is to adapt the families of currents constructed by geometric method (in \S \ref{min-max family}) to the Almgren-Pitts setting (in \S \ref{Almgren-Pitts min-max theory}). Usually families constructed by geometric method are continuous under the flat norm topology, but the Almgren-Pitts theory applies only to discrete family continuous under the mass norm topology. Therefore we need to discretize our families and to make them continuous under the mass norm topology. Similar issue was also an essential technical difficulty in the celebrated proof of the Willmore conjecture \cite{MN12}, and also in a previous paper by the author \cite{Z12} which deals with the same problem in low dimensions. A key technical condition in these discretization type theorems in \cite{MN12, MN13, Z12} is the no local mass concentration assumption. Roughly speaking, it means that the weak measure-theoretical closure of the family of currents does not contain any point mass. However, the families used here do not necessarily satisfy this technical assumption, so we will build up a stronger version of the discretization theorem without assuming the no mass concentration condition. Actually, this issue was originally considered by Pitts \cite[\S 3.5, \S3.7]{P81} in another setting. Our strategy is motivated by Pitts's method, and is simpler than Pitts's discretization procedure. In this paper, we only deal with families of currents which are boundaries of sets of finite perimeter. This is already enough for the purpose of many geometric applications, as all the known interesting geometric families (c.f. \cite{MN12, MN13, Z12}) belong to this class. In fact, it is conjectured by Marque and Neves \cite[\S 13.2]{MN12} that the no mass concentration assumption is not necessary, and our result confirms this conjecture in the co-dimension one case. For the purpose of simplicity, we only present the discretization theorem for one-parameter families. The case for multi-parameter families is still true by similar arguments as in \cite[Theorem 13.1]{MN12} using our technical results Proposition \ref{technical 1} and Proposition \ref{technical 2} in place of \cite[Proposition 13.3, 13.5]{MN12}, and will be addressed elsewhere.

Another key ingredient which utilizes the big machinery by Almgren-Pitts is an identification type result. We will show that all the discretized families corresponding to those families constructed in \S \ref{min-max family} belong to the same homotopy class in the sense of Almgren-Pitts. This type of result was proved in \cite{Z12} under the no mass concentration assumption, and we will extend this identification type result to the case without no mass concentration assumption. We prove this by showing that the image of the discretized families under the Almgren's isomorphism represent the top homology class of $M$. Then these families must be homotopic to each other by Theorem \ref{isomorphism}.

The main result can be summarized as the following theorem. Recall that $\C(M)$ is consisted by all subsets of $M$ of finite perimeter.

\begin{theorem}\label{discretization and identification}
Given a continuous mapping
$$\Phi: [0, 1]\rightarrow\big(\Z_{n}(M^{n+1}, \F), \{0\}\big),$$
satisfying
\begin{itemize}
\setlength{\itemindent}{1em}
\addtolength{\itemsep}{-0.7em}

\item[$(a)$] $\Phi(x)=\partial [[\Om_x]]$, $\Om_x\in\C(M)$, for all $x\in [0, 1]$; 
\item[$(b)$] $\sup_{x\in[0, 1]}\M(\Phi(x))<\infty$;
\end{itemize}
then there exists a $(1, \M)$-homotopy sequence
$$\phi_{i}: I(1, k_{i})_{0}\rightarrow \big(\Z_{n}(M^{n+1}, \M), \{0\}\big),$$
and a sequence of homotopy maps
$$\psi_{i}: I(1, k_{i})_{0}\times I(1, k_{i})_{0}\rightarrow\Z_{n}(M^{n+1}, \M),$$
with $k_{i}<k_{i+1}$, and $\{\de_{i}\}_{i\in\N}$ with $\de_{i}>0$, $\de_{i}\rightarrow 0$, and $\{l_{i}\}_{i\in\N}$, $l_{i}\in\N$ with $l_{i}\rightarrow\infty$, such that $\psi_{i}([0], \cdot)=\phi_{i}$, $\psi_{i}([1], \cdot)=\phi_{i+1}|_{I(1, k_{i})_{0}}$, and
\begin{itemize}
\setlength{\itemindent}{1em}
\addtolength{\itemsep}{-0.7em}
\item[$(\rom{1})$] $$\M\big(\phi_{i}(x)\big)\leq \sup\big\{\M\big(\Phi(y)\big):\ x, y\in\al, \textrm{ for some 1-cell }\al\in I(1, l_{i})\big\}+\de_{i},$$
and hence
\begin{equation}\label{max mass control of discrete family}
\bL(\{\phi_{i}\}_{i\in\N})\leq\sup_{x\in[0, 1]}\M\big(\Phi(x)\big);
\end{equation}
\item[$(\rom{2})$] $\f(\psi_{i})<\de_{i}$;
\item[$(\rom{3})$] $\sup\big\{\F\big(\phi_{i}(x)-\Phi(x)\big):\ x\in I(1, k_{i})_{0}\big\}<\de_{i}$,
\item[$(\rom{4})$] If $\Om_0=\emptyset$, $\Om_1=M$, then
$$F_A(\{\phi_i\})=[[M]],$$
where $F_A$ is the Almgren's isomorphism, and $[[M]]$ is the fundamental class of $M$.
\end{itemize}
\end{theorem}
\begin{remark}
The proof of properties $(\rom{1})(\rom{2})(\rom{3})$ is based on the proof of \cite[Theorem 13.1]{MN12} and \cite[\S 3.5, 3.7]{P81}. The idea to deal with the existence of mass concentration is motivated by \cite[\S 3.5, 3.7]{P81}. We actually simplify the discretization procedure in \cite[\S 3.5]{P81} for currents which can be represented by boundary of sets of finite perimeter using some new observations (c.f. Lemma \ref{case 2}). The proof of property $(\rom{4})$ is based on the ideas in \cite[Theorem 5.8]{Z12}.

Upon first perusal of this section, the reader might skip the following technical proof and move to \S \ref{proof of the main theorem}. 
\end{remark}

\subsection{Technical preliminaries.}

The following two technical results are parallel to \cite[Proposition 13.3, 13.5]{MN12}, while without assuming the no mass concentration condition.

The first result is parallel to \cite[13.3]{MN12}, and it says that given $T\in\Z_n(M^{n+1})$, and $l, m\in N$, there exists $k\in\N$, $k>l$, such that any $\phi$ which maps $I_0(m, l)_0$ into a small neighborhood of $T$ (with respect to the flat topology) can be extended to a map $\ti{\phi}$ which maps $I(m, k)_0$ into a slightly larger neighborhood of $T$ (with respect to the flat topology), such that the fineness and maximal mass of $\ti{\phi}$ are not much bigger than those of $\phi$. Compared to \cite[13.3]{MN12}, we do not require the no mass concentration condition, but we need to assume that the image of $\phi$ are represented by boundary of sets of finite perimeter. Also, the extension $\ti{\phi}$ will be mapped to a slightly large neighborhood. The idea to deal with the mass concentration traces back to \cite[3.5]{P81}. We will first deform $\phi$ to certain local cones around the mass concentration points (c.f. Lemma \ref{case 2}), and then apply similar extension process as \cite[13.3]{MN12}.

Fix an integer $n_0\in\N$.

\begin{proposition}\label{technical 1}
Given $\de, L>0$, $l, m\in\N$, $m\leq n_0+1$, and
$$T\in\Z_n(M)\cap \{S: \M(S)\leq 2L\}, \textrm{with}\ T=\partial[[\Om_T]],$$
$\Om_T\in\C(M)$, 
then there exist $0<\ep=\ep(l, m, T, \de, L)<\de$, and $k=k(l, m, T, \de, L)\in\N$, $k>l$, and a function $\rho=\rho_{(l, m, T, \de, L)}: \R^1_+\rightarrow \R^1_+$, with $\rho(s)\rightarrow 0$, as $s\rightarrow 0$, such that: for any $0<s<\ep$, and
\begin{equation}\label{phi to be extended}
\phi: I_0(m, l)_0\rightarrow \B^{\F}_s(T)\cap\{S: \M(S)\leq 2L\},\ \textrm{with } \phi(x)=\partial[[\Om_x]],
\end{equation}
$\Om_x\in\C(M)$, 
$x\in I_0(m, l)_0$, there exists
$$\ti{\phi}: I(m, k)_0\rightarrow \B^{\F}_{\rho(s)}(T),\ \textrm{with }\ti{\phi}(y)=\partial[[\Om_y]],$$
$\Om_y\in\C(M)$, 
$y\in I(m, k)_0$, and satisfying
\begin{itemize}
\setlength{\itemindent}{1em}
\addtolength{\itemsep}{-0.7em}
\item[$(\rom{1})$] $\f(\ti{\phi})\leq \de$ if $m=1$, and $\f(\ti{\phi})\leq m(\f(\phi)+\de)$ if $m>1$;
\item[$(\rom{2})$] $\ti{\phi}=\phi\circ \n(k, l)$ on $I_0(m, k)_0$;
\item[$(\rom{3})$] $$\sup_{x\in I(m, k)_0}\M\big(\ti{\phi}(x)\big)\leq \sup_{x\in I_0(m, l)_0}\M\big(\phi(x)\big)+\frac{\de}{n_0+1};$$
\item[$(\rom{4})$] If $m=1$, $\de<\nu_M$\footnote{$\nu_M$ is defined in Section \ref{Almgren isomorphism}.}, $\phi([0])=\partial[[\Om_0]]$, $\phi([1])=\partial[[\Om_1]]$, then
$$F_A(\ti{\phi})=[[\Om_1-\Om_0]]$$
where $F_A$ is the Almgren's isomorphism (\ref{Almgren isomorphism1}).
\end{itemize}
\end{proposition}
\begin{remark}
(\rom{1}) controls the fineness of the extension $\ti{\phi}$; (\rom{2}) says that on the boundary vertices $I_0(m, k)_0$ of the cell complex $I(m, k)$, the extension $\ti{\phi}$ directly inherits from $\phi$; (\rom{3}) controls the increase of the mass; (\rom{4}) calculates the image of $\ti{\phi}$ under the Almgren's isomorphism when $m=1$.
\end{remark}

\begin{proof}
We use the contradiction argument. If the statement is not true, by Section \ref{reverse of technical 1}, there exists $k_0\in\N$ large enough, $\rho_0>0$, and a sequence of $\ep_k<1/k$, and
$$\phi_k: I_0(m, l)_0\rightarrow \B^{\F}_{\ep_k}(T)\cap \{S: \M(S)\leq 2L\},$$
$\phi_k(x)=\partial[[\Om^k_x]]$, $\Om^k_x\in\C(M)$, such that there is no extension $\ti{\phi}_k$ of $\phi_k$ from $I(m, k_0)$ to $\B^{\F}_{\rho_0}(T)$, i.e. $\ti{\phi}_k: I(m, k_0)_0\rightarrow \B^{\F}_{\rho_0}(T)$, satisfying all the above properties (\rom{1})(\rom{2})(\rom{3})(\rom{4}).

The next lemma is an analog to \cite[Lemma 13.4]{MN12} without assuming the no mass concentration condition, and uses some new ideas motivated from \cite[\S 3.5]{P81}. Proposition \ref{technical 1} will be proved using Lemma \ref{lemma 1 for technical 1}.
\begin{lemma}\label{lemma 1 for technical 1}
With $\phi_k, \ep_k$ as above, there exist $N=N(l, m, T, \de, L)\in\N$, $N>l$, and a subsequence $\{\phi_j\}$, and a sequence of positive numbers $\rho_j\rightarrow 0$, as $j\rightarrow \infty$, such that we can construct
$$\psi_j: I(1, N)_0\times I_0(m, l)_0\rightarrow \B^{\F}_{\rho_j}(T),$$
satisfying
\begin{itemize}
\setlength{\itemindent}{1em}
\addtolength{\itemsep}{-0.7em}
\item[$(0)$] $\psi_j(y, x)=\partial[[\Om^j_{y, x}]]$, $\Om^j_{y, x}\in\C(M)$, $(y, x)\in I(1, N)_0\times I_0(m, l)_0$;
\item[$(\rom{1})$] $\f(\psi_j)\leq \de$ if $m=1$, and $\f(\psi_j)\leq \f(\phi_j)+\de$ if $m>1$;
\item[$(\rom{2})$] $\psi_j([0], \cdot)=\phi_j$, $\psi_j([1], \cdot)=T$;
\item[$(\rom{3})$] $$\sup\{\M\big(\psi_j(y, x)\big), (y, x)\in I(1, N)_0\times I_0(m, l)_0\}\leq \sup_{x\in I_0(m, l)_0}\M\big(\phi_j(x)\big)+\frac{\de}{n_0+1};$$
\item[$(\rom{4})$] If $m=1$, $\de<\nu_M$, $\phi_j([0])=\partial[[\Om_{j, 0}]]$, $\phi_j([1])=\partial[[\Om_{j, 1}]]$, then
$$F_A(\psi_j|_{I(1, N)_0\times\{[0]\}})=[[\Om_T-\Om_{j, 0}]],\quad F_A(\psi_j|_{I(1, N)_0\times \{[1]\}})=[[\Om_T-\Om_{j, 1}]]\footnote{Here we identify $I(1, N)_0\times \{[0]\}$ and $(I(1, N))_0\times \{[1]\}$ with $I(1, N)_0$},$$
where $F_A$ is the Almgren's isomorphism (\ref{Almgren isomorphism1}).
\end{itemize}
\end{lemma}
\begin{proof}
As a subset in $\V_n(M)$ with uniformly bounded mass is weakly compact, we can find a subsequence $\{\phi_j\}$ of $\{\phi_k\}$, and a map
$$V: I_0(m, l)_0\rightarrow \V_n(M),$$
such that $\lim_{j\rightarrow\infty}|\phi_j(x)|=V(x)$ as varifolds, $\|V(x)\|(M)\leq 2L$, for all $x\in I_0(m, l)_0$. Also as $\ep_j\rightarrow 0$, $\lim_{j\rightarrow\infty}\phi_j(x)=T$ as currents. 

Now we need to separate our discussion into two cases:
\begin{itemize}
\setlength{\itemindent}{3em}
\addtolength{\itemsep}{-0.7em}
\item[{\bf Case 1:}] $\|V(x)\|(p)\leq \de/5$, for all $p\in M$, $x\in I_0(m, l)_0$;
\item[{\bf Case 2:}] The set $S_{con}\footnote{The notion $S_{con}$ means ``set of mass concentration points".}=\{q\in M:\ \|V(x)\|(q)> \de/5 \textrm{ for some }x\in I_0(m, l)_0\}\neq \emptyset$.
\end{itemize}

\begin{lemma}\label{case 1}
In {\bf Case 1}, there exist $N_1=N_1(l, m, T, \de, L)\in\N$, and
$$\psi_j: I(1, N_1)_0\times I_0(m, l)_0\rightarrow \B^{\F}_{\ep_j}(T),$$
satisfying properties $(0)(\rom{1})(\rom{2})(\rom{3})(\rom{4})$ in Lemma \ref{lemma 1 for technical 1}.
\end{lemma}
\begin{remark}
The proof is a straightforward adaption of \cite[3.7]{P81}\cite[Lemma 13.4]{MN12}\cite[Theorem 5.8]{Z12}, so we omit some identical details.
\end{remark}
\begin{proof}
By the lower semi-continuity of weak convergence $\lim_{j\rightarrow\infty}\phi_j(x)\rightarrow T$,
$$\|T\|\big(B_r(p)\big)\leq \|V(x)\|\big(B_r(p)\big),\quad \forall p\in M, r>0.$$
As $\|V(x)\|(\{p\})\leq \de/5$ for all $x\in I_0(m, l)_0$, $p\in M$, we can find a finite collection of pairwise disjoint open balls $\{B_{r_i}(p_i): 1\leq i\leq v\}$, $p_i\in M$, $r_i>0$, $v\in\N$, such that for all $x\in I_0(m, l)_0$,

\begin{fact}\label{fact 1 of case 1}
\begin{enumerate}
\setlength{\itemindent}{1em}
\addtolength{\itemsep}{-0.7em}
\item $\|T\|\big(B_{r_i}(p_i)\big)\leq \|V(x)\|\big(B_{r_i}(p_i)\big)< \de/3$;

\item  $\|T\|\big(M\backslash\cup_{i=1}^v B_{r_i}(p_i)\big)\leq \|V(x)\|\big(M\backslash\cup_{i=1}^v B_{r_i}(p_i)\big)< \de/3$;

\item $\|T\|\big(\partial B_{r_i}(p_i)\big)=\|V(x)\|\big(\partial B_{r_i}(p_i)\big)=0$;

\item $v$ depends only on $l, m, T, \de, L$ by compactness of varifolds with bounded mass.
\end{enumerate}
\end{fact}

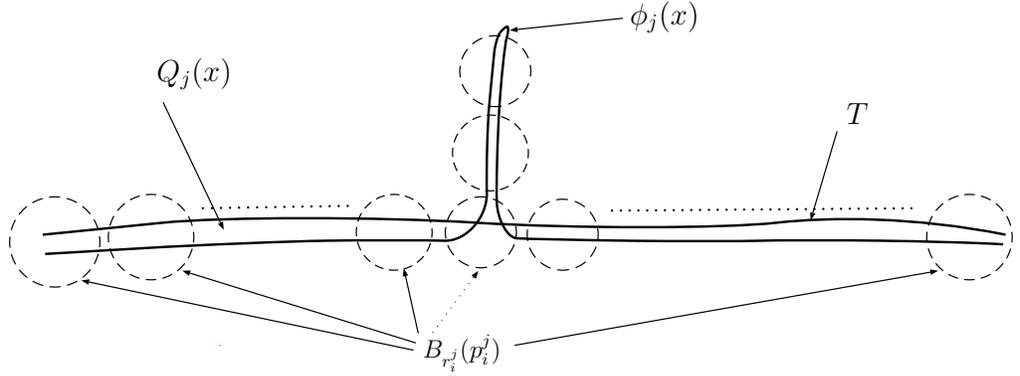
\begin{figure}[t]
\begin{center}
\psscalebox{0.8 0.8} 
{
\begin{pspicture}(0,-4.7081466)(19.72,4.7081466)
\definecolor{colour0}{rgb}{1.0,0.2,0.2}

\pscustom[linecolor=black, linewidth=0.04]
{
\newpath
\moveto(3.6,-1.2518533)
\lineto(5.6,-1.0518533)
\curveto(6.6,-0.95185333)(9.2,-0.95185333)(10.8,-1.0518533)
\curveto(12.4,-1.1518533)(14.9,-1.1518533)(15.8,-1.0518533)
\curveto(16.7,-0.95185333)(18.1,-0.95185333)(19.6,-1.2518533)
}
\pscustom[linecolor=colour0, linewidth=0.04]
{
\newpath
\moveto(3.6,-1.6918533)
}

\pscustom[linecolor=colour0, linewidth=0.04]
{
\newpath
\moveto(8.4,-1.6918533)
}
\pscustom[linecolor=colour0, linewidth=0.04]
{
\newpath
\moveto(3.6,-1.6918533)
}
\pscustom[linecolor=colour0, linewidth=0.04]
{
\newpath
\moveto(3.6,-1.5318533)
}
\pscustom[linecolor=black, linewidth=0.04]
{
\newpath
\moveto(3.64,-1.5718533)
\lineto(5.3353767,-1.4718534)
\curveto(6.1830654,-1.4218533)(7.6613507,-1.3618534)(8.291948,-1.3518534)
\curveto(8.922545,-1.3418534)(9.914961,-1.3418534)(10.276778,-1.3518534)
\curveto(10.638597,-1.3618534)(10.990078,-0.9618533)(10.97974,-0.55185336)
\curveto(10.969402,-0.14185333)(10.990078,0.5981467)(11.021091,0.92814666)
\curveto(11.052105,1.2581466)(11.093454,1.6881467)(11.103791,1.7881466)
\curveto(11.11413,1.8881466)(11.155481,2.0381467)(11.186494,2.0881467)
\curveto(11.217507,2.1381466)(11.279532,2.1981466)(11.310546,2.2081466)
\curveto(11.341558,2.2181466)(11.341558,2.1281466)(11.310546,2.0281467)
\curveto(11.279532,1.9281467)(11.217507,1.3381467)(11.186494,0.8481467)
\curveto(11.155481,0.35814667)(11.134805,-0.42185333)(11.145143,-0.7118533)
\curveto(11.155481,-1.0018533)(11.320883,-1.3018533)(11.475947,-1.3118533)
\curveto(11.631013,-1.3218533)(12.561402,-1.3418534)(13.336727,-1.3518534)
\curveto(14.112052,-1.3618534)(15.207844,-1.3618534)(15.528312,-1.3518534)
\curveto(15.84878,-1.3418534)(16.851534,-1.3418534)(17.53382,-1.3518534)
\curveto(18.216103,-1.3618534)(19.063791,-1.3818533)(19.56,-1.4118533)
}
\pscircle[linecolor=black, linewidth=0.02, linestyle=dashed, dash=0.17638889cm 0.10583334cm, dimen=outer](3.8,-1.3718534){0.76}
\pscircle[linecolor=black, linewidth=0.02, linestyle=dashed, dash=0.17638889cm 0.10583334cm, dimen=outer](5.4,-1.2918533){0.72}
\psline[linecolor=black, linewidth=0.04, linestyle=dotted, dotsep=0.10583334cm](6.24,-0.81185335)(8.68,-0.7718533)
\pscircle[linecolor=black, linewidth=0.02, linestyle=dashed, dash=0.17638889cm 0.10583334cm, dimen=outer](10.88,-1.2118534){0.6}
\pscircle[linecolor=black, linewidth=0.02, linestyle=dashed, dash=0.17638889cm 0.10583334cm, dimen=outer](9.44,-1.2118534){0.64}
\pscircle[linecolor=black, linewidth=0.02, linestyle=dashed, dash=0.17638889cm 0.10583334cm, dimen=outer](11.04,0.10814667){0.64}
\pscircle[linecolor=black, linewidth=0.02, linestyle=dashed, dash=0.17638889cm 0.10583334cm, dimen=outer](11.12,1.4681467){0.6}
\pscircle[linecolor=black, linewidth=0.02, linestyle=dashed, dash=0.17638889cm 0.10583334cm, dimen=outer](12.24,-1.2518533){0.6}
\pscircle[linecolor=black, linewidth=0.02, linestyle=dashed, dash=0.17638889cm 0.10583334cm, dimen=outer](19.0,-1.2918533){0.72}
\psline[linecolor=black, linewidth=0.02, arrowsize=0.05291666666666668cm 2.0,arrowlength=1.4,arrowinset=0.0]{<-}(11.36,2.1481466)(13.24,2.3481467)
\rput[bl](13.36,2.1081467){\Large $\phi_j(x)$}
\psline[linecolor=black, linewidth=0.02, arrowsize=0.05291666666666668cm 2.0,arrowlength=1.4,arrowinset=0.0]{<-}(16.36,-1.0118533)(16.8,0.46814668)
\rput[bl](16.96,0.5881467){\Large $T$}
\psline[linecolor=black, linewidth=0.02, arrowsize=0.05291666666666668cm 2.0,arrowlength=1.4,arrowinset=0.0]{<-}(4.24,-2.0118532)(9.76,-3.1718533)
\psline[linecolor=black, linewidth=0.02, arrowsize=0.05291666666666668cm 2.0,arrowlength=1.4,arrowinset=0.0]{<-}(5.92,-1.8518534)(9.8,-3.0518534)
\psline[linecolor=black, linewidth=0.02, arrowsize=0.05291666666666668cm 2.0,arrowlength=1.4,arrowinset=0.0]{<-}(9.6,-1.8118533)(9.84,-2.8918533)
\psline[linecolor=black, linewidth=0.02, linestyle=dotted, dotsep=0.10583334cm, arrowsize=0.05291666666666668cm 2.0,arrowlength=1.4,arrowinset=0.0]{<-}(10.84,-1.8518534)(10.04,-2.9318533)(10.04,-2.9718533)
\rput[bl](9.92,-3.5318534){\large $B_{r_i^j}(p_i^j)$}
\psline[linecolor=black, linewidth=0.02, arrowsize=0.05291666666666668cm 2.0,arrowlength=1.4,arrowinset=0.0]{<-}(18.48,-1.8518534)(11.44,-3.0918534)(11.44,-3.0918534)
\psline[linecolor=black, linewidth=0.04, linestyle=dotted, dotsep=0.10583334cm](13.04,-0.8518533)(18.04,-0.81185335)
\pscustom[linecolor=colour0, linewidth=0.04, linestyle=dotted, dotsep=0.10583334cm]
{
\newpath
\moveto(3.6,-1.2918533)
}
\pscustom[linecolor=colour0, linewidth=0.04, linestyle=dotted, dotsep=0.10583334cm]
{
\newpath
\moveto(3.6,-1.2518533)
}
\pscustom[linecolor=colour0, linewidth=0.04, linestyle=dotted, dotsep=0.10583334cm]
{
\newpath
\moveto(3.4,1.2281467)
}
\pscustom[linecolor=colour0, linewidth=0.04]
{
\newpath
\moveto(2.32,1.6681466)
}
\pscustom[linecolor=colour0, linewidth=0.04]
{
\newpath
\moveto(3.6,-1.2518533)
}
\pscustom[linecolor=colour0, linewidth=0.04]
{
\newpath
\moveto(3.6,-1.2518533)
}
\psarc[linecolor=colour0, linewidth=0.04, dimen=outer](13.321739,-4.648375){0.04}{-110.0}{-95.0}
\psarc[linecolor=colour0, linewidth=0.04, dimen=outer](3.7565217,-1.3788098){0.76521736}{131.69592}{132.105}
\psarc[linecolor=colour0, linewidth=0.04, dimen=outer](6.504348,-3.0831578){0.04}{-20.0}{-5.0}
\psline[linecolor=black, linewidth=0.02, arrowsize=0.05291666666666668cm 2.0,arrowlength=1.4,arrowinset=0.0]{<-}(6.6086955,-1.1701142)(5.6,0.95162493)
\rput[bl](5.495652,1.1603206){\Large $Q_j(x)$}
\end{pspicture}
}
\end{center}
\caption{\label{discretization without point mass}\small This figure illustrates the geometric objects using in Lemma \ref{case 1}.}
\end{figure}

By \cite[Corollary 1.14]{AF62}, for $j\gg 1$, $x\in I_0(m, l)_0$, there exists isoperimetric choices $Q_j(x)\in\bI_{n+1}(M^{n+1})$, such that
\begin{equation}\label{isoperimetric choice case 1}
\partial Q_j(x)=\phi_j(x)-T,\quad \M\big(Q_j(x)\big)=\F\big(\phi_j(x)-T\big)\leq \ep_j<1/j.
\end{equation}
For each $i=1,\cdots, v$, let $d_i(x)=dist(p_i, x)$ be the distance function to $p_i$ on $(M, g)$. Using the Slicing Theorem \cite[28.5]{Si83}, for each $i=1, \cdots, v$, we can find a sequence of positive numbers $\{r^j_i\}$, such that $r^j_i\searrow r_i$, such that for all $x\in I_0(m, l)_0$, the slices $\lan Q_j(x), d_i, r^j_i\ran\in\bI_n(M)$, and
\begin{equation}\label{slices case 1}
\lan Q_j(x), d_i, r^j_i\ran =\partial \big(Q_j(x)\lc B_{r^j_i}(p_i)\big)-\big(\phi_j(x)-T\big)\lc B_{r^j_i}(p_i).
\end{equation}
Also as $\lim_{j\rightarrow\infty}\M\big(Q_j(x)\big)=0$, by \cite[28.5(1)]{Si83}, we can choose $\{r^j_i\}$ so that for $j$ large enough,
\begin{equation}\label{mass of slices case 1}
\sum_{x\in I_0(m, l)_0}\sum_{i=1}^v \M\big(\lan Q_j(x), d_i, r^j_i\ran\big)\leq \frac{\de}{2(n_0+1)}.
\end{equation}
Using Fact \ref{fact 1 of case 1} and the lower semi-continuity of mass functional,  for $j$ large enough,
\begin{equation}\label{local small mass case 1}
\|\phi_j(x)\|\big(B_{r^j_i}(p_i)\big)<\de/3,\quad\quad \|T\|\big(B_{r^j_i}(p_i)\big)<\de/3;
\end{equation}
\begin{equation}\label{marginal small mass case 1}
\|\phi_j(x)\|\big(M\backslash\cup_{i=1}^v B_{r^j_i}(p_i)\big)<\de/3, \quad \|T\|\big(M\backslash\cup_{i=1}^v B_{r^j_i}(p_i)\big)<\de/3;
\end{equation}
\begin{equation}\label{small mass gap case 1}
\big(\|T\|-\|\phi^j_i(x)\|\big)\big(B_{r^j_i}(p_i)\big)\leq \frac{\de}{2(n_0+1)v},
\end{equation}
for all $i=1, \cdots, v$, and $x\in I_0(m, l)_0$.

Let $v+1=3^{N_1}$, $N_1\in\N$, then $N_1$ depends only on $l, m, T, \de, L$. Define $\psi_j: I(1, N_1)_0\times I_0(m, l)_0\rightarrow \Z_n(M^{n+1})$ by,
\begin{equation}\label{construction of psi_j case 1}
\begin{split}
\psi_j([\frac{i}{3^{N_1}}], x)= &\phi_j(x)-\sum_{a=1}^i\partial\big(Q_j(x)\lc B_{r^j_a}(p_a)\big), \textrm{ for } 0\leq i\leq 3^{N_1}-1,\\
                                             &\psi_j([1], x)=T.
\end{split}
\end{equation}
Similar arguments as in the proof of \cite[Lemma 13.4]{MN12} using (\ref{slices case 1})(\ref{mass of slices case 1})(\ref{local small mass case 1})(\ref{marginal small mass case 1})(\ref{small mass gap case 1}) in place of \cite[(67)(68)(69)(70)(71)]{MN12} show that $\psi_j([\frac{i}{3^{N_1}}], x)\in\B^{\F}_{\ep_j}(T)$ for all $1\leq i\leq 3^{N_1}$, $x\in I_0(m, l)_0$, and that $\{\psi_j\}$ satisfy properties (\rom{1})(\rom{2})(\rom{3}) in Lemma \ref{lemma 1 for technical 1}.

Noe let us check property (0) in Lemma \ref{lemma 1 for technical 1}. We assume that $T\neq 0$ (the case $T=0$ is easier). Denote $\phi_j(x)=\partial[[\Om_j(x)]]$, $\Om_j(x)\in\C(M)$, then by Lemma \ref{lemma on isoperimetric choice 2}, for $j$ large enough, the isoperimetric choices $Q_j(x)$ in (\ref{isoperimetric choice case 1}) satisfy that:
$$Q_j(x)=[[\Om_j(x)-\Om_T]],\quad \textrm{for all } x\in I_0(m, l)_0.$$
Hence by (\ref{construction of psi_j case 1}), for $0\leq i\leq 3^{N_1}-1$,
\begin{displaymath}
\begin{split}
\psi_j([\frac{i}{3^{N_1}}], x) & =\partial[[\Om_j(x)]]-\sum_{a=1}^i\partial\big([[\Om_j(x)-\Om_T]]\lc B_{r^j_a}(p_a)\big)\\
                                           & =\partial\Big\{[[\Om_j(x)\lc\big(M\backslash\cup_{a=1}^i B_{r^j_a}(p_a)\big)]]+[[\Om_T\lc \big(\cup_{a=1}^i B_{r^j_a}(p_a)\big)]]\Big\}.
\end{split}
\end{displaymath}
This proves Lemma \ref{lemma 1 for technical 1}(0) as $\Om_j(x)\lc\big(M\backslash\cup_{a=1}^i B_{r^j_a}(p_a)\big)+\Om_T\lc \big(\cup_{a=1}^i B_{r^j_a}(p_a)\big)\in\C(M)$.

Finally let us check property (\rom{4}) in Lemma \ref{lemma 1 for technical 1}. Assume that $m=1$, and $\ep_j<\nu_M$. Let us calculate $F_A(\psi_j|_{I(1, N_1)_0\times\{[0]\}})$ and $F_A(\psi_j|_{I(1, N_1)_0\times\{[1]\}})$. First we do $F_A(\psi_j|_{I(1, N_1)_0\times\{[0]\}})$. By the definition of Almgren's isomorphism (\ref{Almgren isomorphism1}),
$$F_A(\psi_j|_{I(1, N_1)_0\times\{[0]\}})=\sum_{i=1}^{v+1}Q_{j, i}(0),$$
where $Q_{j, i}(0)$ is the isoperimetric choice of $\psi_j([\frac{i}{3^{N_1}}], [0])-\psi_j([\frac{i-1}{3^{N_1}}], [0])$, $i=1, \cdots, v$, and $Q_{j, v+1}(0)$ is the isoperimetric choice of $T-\psi_j([\frac{v}{3^{N_1}}], [0])$. By (\ref{construction of psi_j case 1}),
$$\psi_j([\frac{i}{3^{N_1}}], [0])-\psi_j([\frac{i-1}{3^{N_1}}], [0])=-\partial\big(Q_j(x)\lc B_{r^j_i}(p_i)\big),$$
and hence by Lemma \ref{lemma on isoperimetric choice}, $Q_{j, i}(0)=-Q_j(x)\lc B_{r^j_i}(p_i)=[[\Om_T-\Om_j(0)]]\lc B_{r^j_i}(p_i)$. Similarly,
$$T-\psi_j([\frac{v}{3^{N_1}}], [0])=-\partial\big(Q_j(x)\lc [M\backslash\cup_{i=1}^vB_{r^j_i}(p_i)]\big),$$
and hence by Lemma \ref{lemma on isoperimetric choice}, $Q_{j, v+1}(0)=-Q_j(x)\lc  [M\backslash\cup_{i=1}^vB_{r^j_i}(p_i)]=[[\Om_T-\Om_j(0)]]\lc  [M\backslash\cup_{i=1}^vB_{r^j_i}(p_i)]$. Summing them together,
\begin{displaymath}
\begin{split}
F_A(\psi_j|_{I(1, N_1)_0\times\{[0]\}}) &=\sum_{i=1}^{v}[[\Om_T-\Om_j(0)]]\lc B_{r^j_i}(p_i)+ [[\Om_T-\Om_j(0)]]\lc  [M\backslash\cup_{i=1}^vB_{r^j_i}(p_i)]\\
                                                            &=[[\Om_T-\Om_j(0)]].
\end{split}
\end{displaymath}
Similar arguments show that $F_A(\psi_j|_{I(1, N_1)_0\times\{[1]\}})=[[\Om_T-\Om_j(1)]]$, and hence property (\rom{4}) (in Lemma \ref{lemma 1 for technical 1}) is proved.
\end{proof}

\begin{lemma}\label{case 2}
In {\bf Case 2}, there exist $N_2=N_2(l, m, \de, L)\in\N$, and a subsequence (still denoted by) $\{\phi_j\}$, and a sequence of positive numbers $\rho_j\rightarrow 0$, as $j\rightarrow \infty$, and
$$\psi_j: I(1, N_2)_0\times I_0(m, l)_0\rightarrow \B^{\F}_{\rho_j}(T),$$
satisfying:
\begin{itemize}
\setlength{\itemindent}{1em}
\addtolength{\itemsep}{-0.7em}
\item[$(0)$] $\psi_j(y, x)=\partial[[\Om^j_{y, x}]]$, $\Om^j_{y, x}\in\C(M)$, $(y, x)\in I(1, N_2)_0\times I_0(m, l)_0$;
\item[$(\rom{1})$] $\f(\psi_j)\leq \de$ if $m=1$, and $\f(\psi_j)\leq \f(\phi_j)+\de$ if $m>1$;
\item[$(\rom{2})$] $\psi_j([0], \cdot)=\phi_j$,
$$\lim_{j\rightarrow\infty}|\psi_j([1], x)|=V(x)\lc G_n\big(M\backslash S_{con}\big)\footnote{$G_{n}(U)$, $U\subset M$ denotes the $n$-Grassmannian bundle over $U$ \cite[\S 38]{Si83}.}\ \textrm{as varifolds for all $x\in I_0(m, l)_0$};$$
\item[$(\rom{3})$] $$\sup\{\M\big(\psi_j(y, x)\big), (y, x)\in I(1, N_2)_0\times I_0(m, l)_0\}\leq \sup_{x\in I_0(m, l)_0}\M\big(\phi_j(x)\big)+\frac{\de}{n_0+1};$$
\item[$(\rom{4})$] If $m=1$, $\de<\nu_M$, $\phi_j([0])=\partial[[\Om_{j, 0}]]$, $\phi_j([1])=\partial[[\Om_{j, 1}]]$, $\psi_j([1]\otimes [0])=\partial[[\Om^{\pr}_{j, 0}]]$, $\psi_j([1]\otimes [1])=\partial[[\Om^{\pr}_{j, 1}]]$\footnote{We introduce new notions $\Om^{\pr}_{j, 0}, \Om^{\pr}_{j, 1}$ to simplify presentation, and according to $(0)$, $\Om_{j, 0}^{\pr}=\Om_{[1], [0]}^j$ and $\Om_{j, 1}^{\pr}=\Om_{[1], [1]}^j$.}, then
$$F_A(\psi_j|_{I(1, N_2)_0\times\{[0]\}})=[[\Om^{\pr}_{j, 0}-\Om_{j, 0}]],\quad F_A(\psi_j|_{I(1, N_2)_0\times \{[1]\}})=[[\Om^{\pr}_{j, 1}-\Om_{j, 1}]]\footnote{Here we identify $I(1, N_2)_0\times \{[0]\}$ and $(I(1, N_2))_0\times \{[1]\}$ with $I(1, N_2)_0$.},$$
where $F_A$ is the Almgren's isomorphism (\ref{Almgren isomorphism1}).
\end{itemize}
\end{lemma}
\begin{remark}
This lemma is the key part towards Theorem \ref{discretization and identification}. As the proof is very subtle, we sketch the main ideas here. Let us focus on a simpler case when $S_{con}$ contains only one point $q$ (Part \Rom{1} in the proof), and the general case (Part \Rom{2}) follows from straightforward induction. For $j$ large, we will find points $p_j\rightarrow q$ and radii $r_j\rightarrow 0$, such that the mass of the slicing $\M\big[\partial\big(\phi_j(x)\lc B(p_j, r_j)\big)\big]\rightarrow 0$ (Fact \ref{smallness of boundary and radius}). To get rid of the mass concentration, we will connect $\phi_j(x)$ to local cones $0\ttimes \partial\big(\phi_j(x)\lc B(p_j, r_j)\big)$ inside $B(p_j, r_j)$ in finitely many steps simultaneously for all $x\in I_0(m, l)_0$. To keep the fineness small during this procedure, we will find finitely many concentric annuli inside $B(p_j, r_j)$ (Fact \ref{fact 2 of case 2}), and do the deformation step by step on each annulus (Step 1 to 3). The number of annuli can be chosen to depend only on $l, m, \de, L$ (Fact \ref{fact 2 of case 2}.4). All the properties (0)(\rom{1})(\rom{2})(\rom{3})(\rom{4}) are checked in Step 4 and 5.

As we are working on a manifold, so all the cone construction should be passed to the tangent plane using exponential map. We summarize the related formulae for local exponential maps in \S \ref{basic facts of exponential map}.
\end{remark}
\begin{proof}
For all basics facts about the local exponential map, we refer to \S \ref{basic facts of exponential map}.

$C(m, l)$ denotes the number of vertices in $I_0(m, l)_0$.

Denote $\al=\de/5$, then the set $S_{con}$ has at most $C(m, l)\frac{2L}{\al}$ points. Given $q\in S_{con}$, then $\|V(x)\|(q)>\al$, for some $x\in I_0(m, l)_0$. Choose a neighborhood $Z=Z_q$ of $q$ satisfying the requirement of \S \ref{basic facts of exponential map}, with respect to some fixed $\ep\leq n/2$. We can make sure that the sets $\{Z_q:\ q\in S_{con}\}$ are pairwise disjoint by possibly shrinking $Z_q$.

\vspace{1em}

\noindent{\bf\large Part \Rom{1}:}
First assume that $S_{con}$ has a single point, i.e. $S_{con}=\{q\}$, and write $Z=Z_q$. We will discuss the general cases using induction method later.

We need the following facts.
\begin{itemize}
\vspace{-5pt}
\setlength{\itemindent}{0.2em}
\addtolength{\itemsep}{-0.7em}
\item[(A)] By basic measure theory,
$$\lim_{r\rightarrow 0}\|V(x)\|\big(B(q, r)\backslash\{q\}\big)=0,\ \forall x\in I_0(m, l)_0.$$

\item[(B)] Given a set of integral currents $\{T(x)\in \Z_n(M^{n+1}):\ x\in I_0(m, l)_0\}$, by \cite[3.6]{P81}, the set
$$\{p\in Z:\ \|T(x)\|\big(\partial B(p, t)\big)=0, \forall t>0, B(p, t)\subset Z\}$$
has a full measure in $Z$;


\item[(C)] Fix $p\in Z$, and $s>0$, with $B(p, 2s)\subset Z$. Then by the slicing theorem \cite[28.5]{Si83} and \S \ref{basic facts of exponential map}(d), $\partial(T(x)\lc B(p, t))\in\Z_{n-1}(M)$ for ($L^1$ almost all) $t\in[s/2, 2s]$, and
\vspace{-5pt}
$$2\|T(x)\|\big(A(p, s/2, 2s)\big)\geq Lip (r_p)\|T(x)\|\big(A(p, s/2, 2s)\big)\geq \int_{s/2}^{2s}\M\big[\partial (T(x)\lc B(p, t))\big]dt.$$
Hence by the Pigeonhole Principle, there exists $r\in [s/2, 2s]$, such that for all $x\in I_0(m, l)_0$,
\begin{itemize}
\vspace{-5pt}
\addtolength{\itemsep}{-0.3em}
\item $\partial \big(T(x)\lc B(p, r)\big)=\lan T(x), r_p, r\ran\in \Z_{n-1}(M^{n+1})$\footnote{$\lan T, r_p, r\ran$ denotes the slicing of $T$ by the function $r_p$ (see \S \ref{basic facts of exponential map}) at $r$ \cite[28.4]{Si83}.};

\item $2 C(m, l)\|T(x)\|\big(A(p, s/2, 2s)\big)\geq \frac{3}{2} s\M\big[\partial(T(x)\lc B(p, r))\big]\geq \frac{3}{4}r \M\big[\partial(T(x)\lc B(p, r))\big].$
\end{itemize}
\end{itemize}

\vspace{5pt}
Now denote $T_j(x)=\phi_j(x)$, $x\in I_0(m, l)_0$,

\begin{claim}\label{claim 1 for case 2}
We can find (possibly up to a further subsequence of $\{\phi_j\}$),
\begin{itemize}
\vspace{-5pt}
\setlength{\itemindent}{1em}
\addtolength{\itemsep}{-0.7em}
\item a sequence of points $p_j\in Z$, $p_j\rightarrow q$ as $j\rightarrow \infty$;

\item sequences of numbers $s_j, r_j\in\R$, with $0<s_j/2<r_j<2s_j$, $\lim_{j\rightarrow\infty}s_j=0$;
\vspace{-5pt}
\end{itemize}

satisfying
\begin{itemize}
\vspace{-5pt}
\setlength{\itemindent}{1em}
\addtolength{\itemsep}{-0.7em}

\item[$(\rom{1})$] $B(q, s_j/8)\subset B(p_j, s_j/4)\subset B(p_j, 2s_j)\subset B(q, 4s_j)$;

\item[$(\rom{2})$] $\|T_j(x)\|\big(\partial B(p_j, t)\big)=0$, for all $x\in I_0(m, l)_0$, $0<t<2s_j$;

\item[$(\rom{3})$] $\lim_{j\rightarrow \infty}\max_{x\in I_0(m, l)_0}\|T_j(x)\|\big[A(p_j, s_j/2, 2 s_j)\big]=0$;

\item[$(\rom{4})$] $\partial\big(T_j(x)\lc B(p_j, r_j)\big)=\lan T_j(x), r_{p_j}, r_j \ran\in \Z_{n-1}(M)$;

\item[$(\rom{5})$] $r_j \M\big[\partial\big(T_j(x)\lc B(p_j, r_j)\big)\big]\leq 8/3 C(m, l)\|T_j(x)\|\big(A(p_j, s_j/2, 2 s_j)\big)$;

\item[$(\rom{6})$] $\lim_{j\rightarrow \infty}|T_j(x)|\lc G_n\big(B^c(p_j, r_j)\big)=V(x)\lc G_n(M\backslash\{q\})$ as varifolds\footnote{$B^c(p, r)$ denotes the complement of $B(p, r)$ in $M$.}.
\end{itemize}
\end{claim} 
Now let us check the claim. By fact (A), we can find $s_j>0$, $s_j\rightarrow 0$, as $j\rightarrow\infty$, such that
$$\lim_{j\rightarrow\infty}\max_{x\in I_0(m, l)_0}\|V(x)\|\big(B(q, 4s_j)\backslash\{q\}\big)=0.$$
As $|T_j(x)|=|\phi_j(x)|$ converge to $V(x)$ as varifolds, we can possibly take a subsequence of $\{\phi_j\}$, still denoted by $\{\phi_j\}$, such that
$$\lim_{j\rightarrow \infty}\max_{x\in I_0(m, l)_0} \|T_j(x)\|\big(A(q, s_j/8, 4s_j)\big)=0,\quad \textrm{and }$$
$$\lim_{j\rightarrow \infty}|T_j(x)|\lc G_n(B^c(q, s_j/8))=V(x)\lc G_n(M\backslash\{q\}), \textrm{ as varifolds, for all }x\in I_0(m, l)_0.$$
(In fact, for any $j$ one can find $j^{\pr}\geq j$, such that $\|T_{j^{\pr}}(x)\|\big(A(q, s_j/8, 4s_j)\big)\leq 2\|V(x)\|\big(B(q, 4s_j)\backslash\{q\}\big)$ and $\|T_{j^{\pr}}(x)\|B(q, s_j/8)\leq \|V(x)\|B(q, s_j/8)+\frac{1}{j}$, and $\{\phi_{j^{\pr}}(x)=T_{j^{\pr}}(x)\}$ satisfies the requirement).

By fact (B), we can find a sequence $p_j\in Z$, $p_j\rightarrow q$, such that $B(q, s_j/8)\subset B(p_j, s_j/4)\subset B(p_j, 2s_j)\subset B(q, 4s_j)$, and $\|T_j(x)\|\big(\partial B(p_j, s)\big)=0$, for all $x\in I_0(m, l)_0$ and $s>0$ with $B(p_j, s)\subset Z$. Hence (\rom{1})(\rom{2}) are true. (\rom{3}) is true as $A(p_j, s_j/2, 2s_j)\subset A(q, s_j/8, 4s_j)$.
Now for each $j$, by fact (C), we can find $r_j\in [s_j/2, 2s_j]$, such that (\rom{4})(\rom{5}) are true. (\rom{6}) is true as $B^c(p_j, r_j)\subset B^c(q, s_j/8)$ and $B(p_j, r_j)\backslash B(q, s_j/8)\subset A(q, s_j/8, 4s_j)$.

\vspace{5pt}
Then we have the following facts.
\begin{fact}\label{smallness of boundary and radius}
Given $\de_1>0$ (to be determined later), $\de_1<\de$, by Claim \ref{claim 1 for case 2}(\rom{3})(\rom{5}), there exists $J$ large enough, such that if $j\geq J$,
\begin{equation}\label{mass bound for slices}
\frac{2r_j}{n} \M\big[\partial\big(T_j(x)\lc B(p_j, r_j)\big)\big]\leq \de_1/5;
\end{equation}
\begin{equation}\label{volume bound for small balls}
vol\big(B(p_j, r_j)\big)\leq \de_1/5;
\end{equation}
\begin{equation}\label{mass bound for small balls}
vol\big(\partial B(p_j, r)\big)\leq \de_1/5, \ \textrm{for all $r\leq r_j$}.
\end{equation}
\end{fact}

\vspace{1em}
Now we are going to connect $T_j(x)\lc B(p_j, r_j)$ to the cones $E_{\#}\big[\de_0\ttimes E^{-1}_{\#}\partial\big(T_j(x)\lc B(p_j, r_j) \big)\big]$ using discrete sequences with controlled fineness simultaneously for all $x\in I_0(m, l)_0$.

We separate the whole procedure into several steps. For notions $E, \mu(\la), h(r)$, we refer to \S \ref{basic facts of exponential map}.

\vspace{0.5em}
\noindent \underline{\bf Step 0}: Now fix $j\geq J$, and forget the subscript $``j"$ now. So $T(x)$ and $B(p, r)$ satisfy (\ref{mass bound for slices})(\ref{volume bound for small balls})(\ref{mass bound for small balls}). Recall that $T(x)=\partial[[\Om(x)]]$, $\Om(x)\in \C(M)$. For simplicity, we will identify $\Om(x)$ with $[[\Om(x)]]$ in the following of the proof. By the Pigeonhole Principle and the Slicing Theorem \cite[28.5]{Si83}, we have that
\begin{fact}\label{fact 2 of case 2}
we can find finitely many numbers $r_i>0$\footnote{Note that $r_i$'s are different from the $r_j$'s in Claim \ref{claim 1 for case 2}, and we will forget the subscript $``j"$ of $r_j$ until Step 5.}, $i=1,\cdots, \nu$, for some $\nu\in\N$, with $r>r_1>r_2>\cdots>r_{\nu}>0$, such that for all $x\in I_0(m, l)_0$, $1\leq i\leq \nu-1$,
\begin{enumerate}
\vspace{-5pt}
\addtolength{\itemsep}{-0.7em}
\item $\|T(x)\|A(p, r_{i+1}, r_i)\leq \de/5$, $\|T(x)\|B(p, r_{\nu})\leq \de/5$;
\item $\partial \big(T(x)\lc B(p, r_i)\big)\in \Z_{n-1}(M^{n+1})$;
\item $\lan \Om(x), r_p, r_i\ran= \partial\big(\Om(x)\lc B(p, r_i)\big)-T(x)\lc B(p, r_i)\in \bI_n(M^{n+1})$;
\item $\nu$ can be any integer no less than $C(m, l)(\de/6)^{-1}\max_{x\in I_0(m, l)_0}\M(T(x)\lc Z)$, and hence depends only on $m, l, \de, L$.
\end{enumerate}
\end{fact}

\begin{figure}[t]
\begin{center}
\psscalebox{0.7 0.7} 
{
\begin{pspicture}(0,-5.6)(23.94,5.6)
\psdots[linecolor=black, dotsize=0.2](8.84,0.0)
\pscircle[linecolor=black, linewidth=0.04, dimen=outer](8.84,0.0){5.6}
\pscircle[linecolor=black, linewidth=0.04, dimen=outer](8.84,0.0){4.4}
\pscustom[linecolor=black, linewidth=0.04]
{
\newpath
\moveto(0.44,3.6)
\lineto(0.84,2.6)
\curveto(1.04,2.1)(1.44,1.3)(1.64,1.0)
\curveto(1.84,0.7)(2.24,0.1)(2.44,-0.2)
\curveto(2.64,-0.5)(3.14,-1.0)(3.44,-1.2)
\curveto(3.74,-1.4)(4.44,-1.8)(4.84,-2.0)
\curveto(5.24,-2.2)(6.24,-2.5)(8.04,-2.8)
}
\pscustom[linecolor=black, linewidth=0.04]
{
\newpath
\moveto(9.64,-2.8)
\lineto(11.44,-2.4)
\curveto(12.34,-2.2)(13.64,-1.8)(14.04,-1.6)
\curveto(14.44,-1.4)(15.14,-1.0)(15.44,-0.8)
\curveto(15.74,-0.6)(16.24,-0.2)(16.44,0.0)
\curveto(16.64,0.2)(17.04,0.6)(17.24,0.8)
\curveto(17.44,1.0)(17.84,1.4)(18.04,1.6)
\curveto(18.24,1.8)(18.44,2.2)(18.44,2.4)
}
\psline[linecolor=black, linewidth=0.04, linestyle=dashed, dash=0.17638889cm 0.10583334cm](3.4,-1.16)(8.84,0.0)(8.84,0.0)
\psline[linecolor=black, linewidth=0.04, linestyle=dashed, dash=0.17638889cm 0.10583334cm](8.84,0.0)(14.24,-1.5)(14.24,-1.4)
\psline[linecolor=black, linewidth=0.04](8.04,-2.8)(8.14,-2.7)(8.24,-3.0)(8.34,-2.6)(8.44,-2.9)(8.44,-3.1)(8.54,-3.0)(8.54,-2.8)(8.64,-2.6)(8.74,-3.0)(8.84,-2.8)(8.84,-3.1)(8.94,-2.6)(9.04,-2.8)(9.14,-3.1)(9.24,-2.8)(9.34,-2.6)(9.44,-2.8)(9.34,-3.1)(9.54,-3.2)(9.54,-3.0)(9.54,-2.8)(9.64,-2.8)(9.74,-2.8)
\psline[linecolor=red, linewidth=0.04](3.36,-1.16)(4.52,-0.92)(4.56,-0.92)
\psarc[linecolor=red, linewidth=0.04, dimen=outer](8.84,-0.04){4.36}{191.5}{208.0}
\psline[linecolor=red, linewidth=0.04](13.08,-1.16)(14.2,-1.48)(14.2,-1.48)
\psarc[linecolor=red, linewidth=0.04, dimen=outer](8.84,0.0){4.4}{332.0}{344.5}
\psline[linecolor=black, linewidth=0.02, arrowsize=0.05291666666666667cm 2.0,arrowlength=1.4,arrowinset=0.0]{<-}(1.12,1.8)(0.4,0.88)(0.4,0.88)
\rput[bl](0.0,0.64){\Large $T$}
\rput[bl](1.24,-2.76){\Large $S_1$}
\psline[linecolor=black, linewidth=0.02, arrowsize=0.05291666666666667cm 2.0,arrowlength=1.4,arrowinset=0.0]{<->}(8.84,0.04)(12.96,3.8)
\rput[bl](10.6,2.0){\Large $r$}
\psline[linecolor=black, linewidth=0.02, arrowsize=0.05291666666666667cm 2.0,arrowlength=1.4,arrowinset=0.0]{<->}(5.36,2.64)(8.84,0.04)
\rput[bl](6.92,1.8){\Large $r_1$}
\psline[linecolor=black, linewidth=0.02, arrowsize=0.05291666666666667cm 2.0,arrowlength=1.4,arrowinset=0.0]{<-}(4.04,-1.08)(1.72,-2.6)
\psline[linecolor=black, linewidth=0.02, arrowsize=0.05291666666666667cm 2.0,arrowlength=1.4,arrowinset=0.0]{<-}(13.48,-1.32)(15.16,-2.56)
\rput[bl](15.24,-2.72){\Large $S_1$}
\psline[linecolor=black, linewidth=0.02, arrowsize=0.05291666666666667cm 2.0,arrowlength=1.4,arrowinset=0.0]{<-}(4.72,-1.56)(2.6,-3.44)
\rput[bl](1.48,-3.96){\Large $(E\circ h(r_1)\circ E^{-1})_{\#}T\lc B(p, r)\backslash B(p, r_1)$}
\psline[linecolor=black, linewidth=0.02, arrowsize=0.05291666666666667cm 2.0,arrowlength=1.4,arrowinset=0.0]{<-}(12.92,-1.64)(14.56,-3.52)
\rput[bl](13.88,-3.92){\Large $(E\circ h(r_1)\circ E^{-1})_{\#}T\lc B(p, r)\backslash B(p, r_1)$}
\end{pspicture}
}
\end{center}

\caption{\label{figure: discretization with point mass_step 1}\small This figure illustrates Step 1 in the discretization process with point mass. We omit the variable $x\in I(m, l)_0$.}
\end{figure}
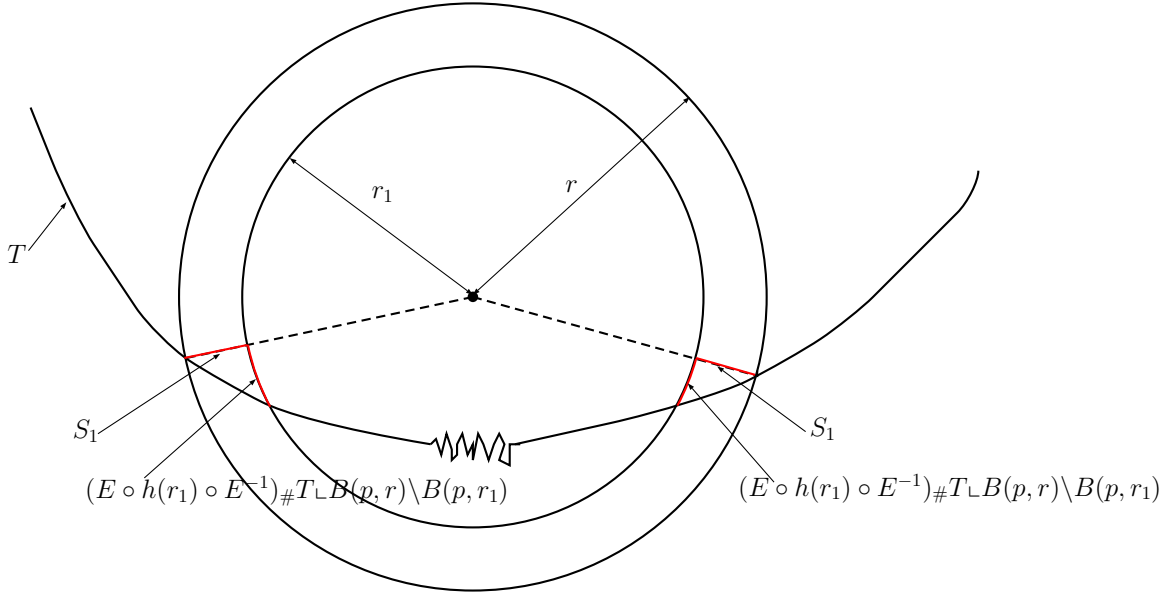

\vspace{0.5em}
\noindent \underline{\bf Step 1}: (See Figure \ref{figure: discretization with point mass_step 1}) For each $x\in I_0(m, l)_0$, let
$$S_1(x)=E_{\#}\Big\{\de_0\ttimes\big[E^{-1}_{\#}\partial (T(x)\lc B(p, r))-\mu(\frac{r_1}{r})_{\#}E^{-1}_{\#}\partial (T(x)\lc B(p, r))\big]\Big\};$$
then by (\ref{mass bound for slices}) and \S \ref{basic facts of exponential map}(k), $spt(S_1(x))\subset A(p, r_1, r)$, and
\begin{equation}\label{mass bound for cone 1}
\M(S_1(x))\leq 2r n^{-1}(1-(\frac{r_1}{r})^n)\M\big(\partial(T(x)\lc B(p, r))\big)\leq 2r n^{-1}\M\big(\partial (T(x)\lc B(p, r))\big)\leq \de_1/5.
\end{equation}
For each $x\in I_0(m, l)_0$, define
\begin{equation}\label{R1}
R_1(x)=\left. \Bigg\{ \begin{array}{ll}
S_1(x), \quad\quad \textrm{ in $A(p, r_1, r)$}\\
(E\circ h(r_1)\circ E^{-1})_{\#}T(x)\lc B(p, r), \quad \textrm{ in $B(p, r_1)$}\\
T(x), \quad\quad \textrm{outside $B(p, r)$}
\end{array}.\right. 
\end{equation}

\begin{claim}\label{claim 2 for case 2}
For each $x\in I_0(m, l)_0$, $R_1(x)=\partial \Om_1(x)$ for some $\Om_1(x)\in\C(M)$.
\end{claim}
\begin{proof}
For each $x\in I_0(m, l)_0$, by the definition of slices \cite[28.4]{Si83}, the slices $\lan\Om(x), r_p, r_i\ran$ is represented by the set $\Om(x)\cap \partial B(p, r_i)$, which has finite perimeter as $\partial \lan\Om(x), r_p, r_i\ran=-\partial \big(T(x)\lc B(p, r_i)\big)$. Denote $O_i(x)=\Om(x)\cap \partial B(p, r_i)=\lan\Om(x), r_p, r_i\ran$, $O(x)=\Om(x)\cap \partial B(p, r)=\lan\Om(x), r_p, r\ran$. Define a subset of $M$ as\footnote{$0\ttimes S$ denotes the cone in $\R^{n+1}$ over $S\subset \R^{n+1}$.}
\begin{equation}\label{Om1}
\Om_1(x)=\left. \Big\{ \begin{array}{ll}
E\big\{0\ttimes \big[E^{-1}O(x)-\frac{r_1}{r}E^{-1} O(x)\big]\big\}, \quad \textrm{ in $A(p, r_1, r)$}\\
\Om(x), \quad\quad \textrm{in $B_0(p, r_1)$ and outside $B(p, r)$}
\end{array}.\right. 
\end{equation}
Clearly $\Om_1(x)$ is a set of finite perimeter, i.e. $\Om(x)\in\C(M)$, as each part supported in $B(p, r_1)$, $B^c(p, r)$, $A(p, r_1, r)$ is. We will show that $R_1(x)=\partial\Om_1(x)$. By \cite[28.5(2)]{Si83},
\begin{displaymath}
\begin{split}
\partial \Om_1(x) & =\partial \big[\Om(x)\lc B^c(p, r)\big]+\partial E\big\{0\ttimes \big[E^{-1}O(x)-\frac{r_1}{r}E^{-1} O(x)\big]\big\}+ \partial \big[\Om(x)\lc B_0(p, r_1)\big]\\
                         & =T(x)\lc B^c(p, r)-\lan\Om(x), r_p, r\ran+ O(x)-(E\circ\mu(\frac{r_1}{r})\circ E^{-1})_{\#}O(x)\\
                         & \ \ -E_{\#}\big\{\de_0\ttimes \big[E^{-1}_{\#}\partial O(x)-\mu(\frac{r_1}{r})_{\#}E^{-1}_{\#}\partial O(x)\big]\}+T(x)\lc B(p, r_1)+\lan\Om(x), r_p, r_1\ran\\
                         & =T(x)\lc B^c(p, r)+T(x)\lc B(p, r_1)-(E\circ\mu(\frac{r_1}{r})\circ E^{-1})_{\#}O(x)+O_1(x)\\
                         & \ \ +E_{\#}\big\{\de_0\ttimes \big[E^{-1}_{\#}\partial (T(x)\lc B(p, r))-\mu(\frac{r_1}{r})_{\#}E^{-1}_{\#}\partial (T(x)\lc B(p, r))\big]\}.
\end{split}
\end{displaymath}
So together with Claim \ref{claim 1 for case 2}(\rom{2}),
\begin{displaymath}
\begin{split}
R_1(x)-\partial\Om_1(x) & =(E\circ h(r_1)\circ E^{-1})_{\#}\big[T(x)\lc A(p, r_1, r)\big]+(E\circ\mu(\frac{r_1}{r})\circ E^{-1})_{\#}O(x)-O_1(x)\\
                                   & =(E\circ h(r_1)\circ E^{-1})_{\#}\big(T(x)\lc A(p, r_1, r)+O(x)-O_1(x)\big)\\
                                   & =(E\circ h(r_1)\circ E^{-1})_{\#}\partial \big(\Om(x)\lc A(p, r_1, r)\big)\\
                                   & =\partial (E\circ h(r_1)\circ E^{-1})_{\#}\big(\Om(x)\lc A(p, r_1, r)\big)\\
                                   & =0,
\end{split}
\end{displaymath}
where we used the fact that $h(r_1)=\mu(\frac{r_1}{r})$ on $\partial B(p, r)$ in the second $``="$, and the fact that any integral $(n+1)$-current on an $n$-dimensional manifold $\partial B(p, r_1)$ is zero in the last $``="$. Hence we finish the proof of the claim.
\end{proof}

As $R_1(x)=\partial\Om_1(x)$, using (\ref{mass bound for small balls}) it is easily seen that
\begin{equation}\label{mass bound for lateral part 1}
\M\big(R_1(x)\lc\partial B(p, r_1)\big)\leq vol\big(\partial B(p, r_1)\big)\leq\de_1/5.
\end{equation}

The set $\{R_1(x): x\in I_0(m, l)_0\}$ satisfies the following properties. First using Claim \ref{claim 1 for case 2}(\rom{2}), Fact \ref{fact 2 of case 2}.1, (\ref{mass bound for cone 1})(\ref{mass bound for lateral part 1}), we have the continuity estimate,
\begin{equation}\label{case 2 continuity estimate 1}
\begin{split}
\M\big(R_1(x)-T(x)\big) &\leq \M\big(T(x)\lc A(p, r_1, r)\big)+\M\big(R_1(x)\lc \partial B(p, r_1)\big)+\M\big(S_1(x)\big)\\
                                 &\leq \de/5+2\de_1/5.
\end{split}
\end{equation}
Using Claim \ref{claim 1 for case 2}(\rom{2}), (\ref{mass bound for cone 1})(\ref{mass bound for lateral part 1}), we have the mass estimate,
\begin{equation}\label{case 2 mass estimate 1}
\begin{split}
\M\big(R_1(x)\big) &\leq \M\big(T(x)\lc B^c(p ,r)\big)+\M\big(S_1(x)\big)+\M\big(R_1(x)\lc \partial B(p, r_1)\big)+\M\big(T(x)\lc B(p, r_1)\big)\\
                           &\leq \M\big(T(x)\big)+2\de_1/5.
\end{split}
\end{equation}
If $m>1$, given $x, y\in I_0(m, l)_0$, such that $d(x, y)=1$, then
\begin{displaymath}
\begin{split}
R_1(x)-R_1(y) & =\big(S_1(x)-S_1(y)\big)+\big(R_1(x)-R_1(y)\big)\lc \partial B(p, r_1)\\
                     & +\big(T(x)-T(y)\big)\lc B(p, r_1)\cup B^c(p, r);
\end{split}
\end{displaymath}
hence using (\ref{mass bound for cone 1})(\ref{mass bound for lateral part 1}), we have the fineness estimate,
\begin{equation}\label{case 2 fineness estimate 1}
\begin{split}
\M\big(R_1(x)-R_1(y)\big) &\leq \M\big(R_1(x)\lc \partial B(p, r_1)\big)+\M\big(R_1(y)\lc \partial B(p, r_1)\big)\\
                                     &\ \ +\M\big(S_1(x)\big)+\M\big(S_1(y)\big)+\M\big(T(x)-T(y)\big)\\
                                     &\leq 4\de_1/5+\f(\phi),
\end{split}
\end{equation}
where $\f(\phi)$ is the fineness (\ref{fineness}) of $\phi$. 

\vspace{0.5em}
\noindent \underline{\bf Step 2}: Now for $2\leq i\leq \nu$, $x\in I_0(m, l)_0$, we can similarly define
$$S_i(x)=E_{\#}\Big\{\de_0\ttimes\big[E^{-1}_{\#}\partial (T(x)\lc B(p, r))-\mu(\frac{r_i}{r})_{\#}E^{-1}_{\#}\partial (T(x)\lc B(p, r))\big]\Big\};$$
then by (\ref{mass bound for slices}) and \S \ref{basic facts of exponential map}(k), $spt(S_i(x))\subset A(p, r_i, r)$, and
\begin{equation}\label{mass bound for cone i}
\M(S_i(x))\leq 2r n^{-1}(1-(\frac{r_i}{r})^n)\M(\partial (T(x)\lc B(p, r)))\leq 2r n^{-1}\M(\partial (T(x)\lc B(p, r)))\leq \de_1/5.
\end{equation}
Similarly define
\begin{equation}\label{Ri}
R_i(x)=\left. \Bigg\{ \begin{array}{ll}
S_i(x), \quad\quad \textrm{ in $A(p, r_i, r)$}\\
(E\circ h(r_i)\circ E^{-1})_{\#}T(x)\lc B(p, r), \quad \textrm{ in $B(p, r_i)$}\\
T(x), \quad\quad \textrm{outside $B(p, r)$}
\end{array}.\right. 
\end{equation}
The same argument as in Claim \ref{claim 2 for case 2} with $r_1$ changed to $r_i$ shows that $R_i(x)=\partial\Om_i(x)$, $\Om_i(x)\in\C(M)$ for all $2\leq i\leq \nu$, $x\in I_0(m, l)_0$, with
\begin{equation}\label{Omi}
\Om_i(x)=\left. \Big\{ \begin{array}{ll}
E\big\{0\ttimes \big[E^{-1}O(x)-\frac{r_i}{r}E^{-1} O(x)\big]\big\}, \quad \textrm{ in $A(p, r_i, r)$}\\
\Om(x), \quad\quad \textrm{in $B_0(p, r_i)$ and outside $B(p, r)$}
\end{array},\right. 
\end{equation}
and hence by (\ref{mass bound for small balls}),
\begin{equation}\label{mass bound for lateral part i}
\M\big(R_i(x)\lc\partial B(p, r_i)\big)\leq vol\big(\partial B(p, r_i)\big)\leq\de_1/5.
\end{equation}
Using (\ref{mass bound for cone i})(\ref{mass bound for lateral part i}) in place of (\ref{mass bound for cone 1})(\ref{mass bound for lateral part 1}) and similar estimates as in Step 1, the currents $\{R_i(x): 2\leq i\leq \nu, x\in I_0(m, l)_0\}$ satisfy the following properties:
\begin{equation}\label{case 2 continuity estimate i}
\begin{split}
\M\big(R_i(x)-R_{i-1}(x)\big) &\leq \M\big(T(x)\lc A(p, r_i, r_{i-1})\big)+\M\big(R_i(x)\lc \partial B(p, r_i)\big)\\
                                        &\ \ +\M\big(R_{i-1}(x)\lc \partial B(p, r_{i-1})\big)+\M\big(S_i(x)-S_{i-1}(x)\big)\\
                                        &\leq \de/5+3\de_1/5.
\end{split}
\end{equation}
\begin{equation}\label{case 2 mass estimate i}
\begin{split}
\M\big(R_i(x)\big) &\leq \M\big(T(x)\lc B^c(p ,r)\big)+\M\big(S_i(x)\big)+\M\big(R_i(x)\lc \partial B(p, r_i)\big)+\M\big(T(x)\lc B(p, r_i)\big)\\
                          &\leq \M\big(T(x)\big)+2\de_1/5.
\end{split}
\end{equation}
If $m>1$, given $x, y\in I_0(m, l)_0$, such that $d(x, y)=1$, then
\begin{equation}\label{case 2 fineness estimate i}
\begin{split}
\M\big(R_i(x)-R_i(y)\big) &\leq \M\big(R_i(x)\lc \partial B(p, r_i)\big)+M\big(R_i(y)\lc \partial B(p, r_i)\big)\\
                                     &\ \ +\M\big(S_i(x)\big)+\M\big(S_i(y)\big)+\M\big(T(x)-T(y)\big)\\
                                     &\leq 4\de_1/5+\f(\phi).
\end{split}
\end{equation}

\vspace{0.5em}
\noindent \underline{\bf Step 3}: Define the cones
$$S_{\nu+1}(x)=E_{\#}\Big\{\de_0\ttimes E^{-1}_{\#}\partial (T(x)\lc B(p, r))\Big\};$$
then by (\ref{mass bound for slices}) and \S \ref{basic facts of exponential map}(k), $spt(S_{\nu+1}(x))\subset B(p, r)$, and
\begin{equation}\label{mass bound for final cone}
\M(S_{\nu+1}(x))\leq 2r n^{-1}\M(\partial (T(x)\lc B(p, r)))\leq \de_1/5.
\end{equation}
Define
\begin{equation}\label{R final}
R_{\nu+1}(x)=\left. \Big\{ \begin{array}{ll}
S_{\nu+1}(x), \quad\quad \textrm{ in $B(p, r)$}\\
T(x), \quad\quad \textrm{outside $B(p, r)$}
\end{array}.\right. 
\end{equation}
Similar argument as in Claim \ref{claim 2 for case 2} with $r_1$ changed to $0$ shows that $R_{\nu+1}(x)=\partial\Om_{\nu+1}(x)$, $\Om_{\nu+1}(x)\in\C(M)$ for all $x\in I_0(m, l)_0$, with
\begin{equation}\label{Om final}
\Om_{\nu+1}(x)=\left. \Big\{ \begin{array}{ll}
E\big\{0\ttimes\big[E^{-1}O(x)\big]\big\}, \quad \textrm{ in $B(p, r)$}\\
\Om(x), \quad\quad \textrm{outside $B(p, r)$}
\end{array}.\right. 
\end{equation}
Using Claim \ref{claim 1 for case 2}(\rom{2}), Fact \ref{fact 2 of case 2}.1, (\ref{mass bound for lateral part i})(\ref{mass bound for final cone}), we have that
\begin{equation}\label{case 2 continuity estimate final}
\begin{split}
\M\big(R_{\nu+1}(x)-R_{\nu}(x)\big) &\leq \M\big(T(x)\lc B(p, r_{\nu})\big)+\M\big(R_{\nu}(x)\lc \partial B(p, r_{\nu})\big)\\
                                        & \ \ +\M\big(S_{\nu+1}(x)-S_{\nu}(x)\big)\\
                                        &\leq \de/5+2\de_1/5.
\end{split}
\end{equation}
\begin{equation}\label{case 2 mass estimate final}
\begin{split}
\M\big(R_{\nu+1}(x)\big) &\leq \M\big(T(x)\lc B^c(p ,r)\big)+\M\big(S_{\nu+1}(x)\big)\\
                                   &\leq \M\big(T(x)\big)+\de_1/5.
\end{split}
\end{equation}
If $m>1$, given $x, y\in I_0(m, l)_0$, such that $d(x, y)=1$, then
\begin{equation}\label{case 2 fineness estimate final}
\begin{split}
\M\big(R_{\nu+1}(x)-R_{\nu+1}(y)\big) &\leq \M\big(S_i(x)\big)+\M\big(S_i(y)\big)+\M\big(T(x)-T(y)\big)\\
                                                      &\leq 2\de_1/5+\f(\phi).
\end{split}
\end{equation}

\vspace{0.5em}
\noindent \underline{\bf Step 4}:
Take $\nu+1=3^{\ti{N}}$ for $\ti{N}\in \N$, then $\ti{N}$ depends only on $l, m, \de, L$ by Fact \ref{fact 2 of case 2}.4. We can define a map
$$\psi: I(1, \ti{N})_0\times I_0(m, l)_0\rightarrow \Z_n(M^{n+1}),$$
by $\psi(0, x)=T(x)=\phi(x)$, $\psi([\frac{i}{3^{\ti{N}}}], x)=R_i(x)$ for $1\leq i\leq \nu+1$. Now we check that $\psi$ satisfy Lemma \ref{case 2}(0)(\rom{1})(\rom{3})(\rom{4}). By combining (\ref{case 2 continuity estimate 1})(\ref{case 2 mass estimate 1})(\ref{case 2 fineness estimate 1})(\ref{case 2 continuity estimate i})(\ref{case 2 mass estimate i})(\ref{case 2 fineness estimate i})(\ref{case 2 continuity estimate final})(\ref{case 2 mass estimate final})(\ref{case 2 fineness estimate final}) and our construction, we have
\begin{itemize}
\setlength{\itemindent}{1em}
\addtolength{\itemsep}{-0.7em}
\item[$(0)$] $\psi([\frac{i}{3^{\ti{N}}}], x)=\partial[[\Om_i(x)]]$, $\Om_i(x)\in\C(M)$;

\item[$(\rom{1})$] $\f(\psi)\leq \de/5+3\de_1/5$ if $m=1$, and $\f(\psi)\leq \max\{\de/5+3\de_1/5, \f(\phi)+4\de_1/5\}$ if $m>1$;

\item[$(\rom{3})$] $\max\{\M(\psi([\frac{i}{3^{\ti{N}}}], x))\}\leq \max\{\M(\phi(x))\}+2\de_1/5$.
\end{itemize}

If $m=1$, $\de<\nu_M$, let us calculate $F_A(\psi|_{I(1, \ti{N})_0\times\{[0]\}})$ and $F_A(\psi|_{I(1, \ti{N})_0\times\{[1]\}})$. First focus on $F_A(\psi|_{I(1, \ti{N})_0\times\{[0]\}})$. We will use notions as above. By the definition of Almgren's isomorphism (\ref{Almgren isomorphism1}),
$$F_A(\psi|_{I(1, \ti{N})_0\times\{[0]\}})=\sum_{i=1}^{\nu+1}Q_i(0),$$
where $Q_1(0)$ is the isoperimetric choice for $R_1(0)-T(0)$, and $Q_i(0)$ is the isoperimetric choice of $R_i(0)-R_{i-1}(0)$, $2\leq i\leq \nu+1$, with $R_i(0)$ given by (\ref{R1})(\ref{Ri})(\ref{R final}). Recall that $T(0)=\partial\Om(0)$, $R_i(0)=\partial \Om_i(0)$, with $\Om(0), \Om_i(0)\in\C(M)$, and that $\Om_i(0)-\Om_{i-1}(0)$ are all supported in $B(p, r)$ by the construction (\ref{Om1})(\ref{Omi})(\ref{Om final}), so $\M\big(\Om_i(0)-\Om_{i-1}(0)\big)\leq vol(B(p, r))\leq 1/2vol(M)$, as $r$ is very small. By Lemma \ref{lemma on isoperimetric choice}, $Q_1(0)=\Om_1(0)-\Om(0)$, $Q_i(0)=\Om_i(0)-\Om_{i-1}(0)$ for $2\leq i\leq \nu+1$, hence
$$F_A(\psi|_{I(1, \ti{N})_0\times\{[0]\}})=\Om_1(0)-\Om(0)+\sum_{i=2}^{\nu+1}(\Om_i(0)-\Om_{i-1}(0))=\Om_{\nu+1}(0)-\Om(0).$$
Similarly we can prove that $F_A(\psi|_{I(1, \ti{N})_0\times\{[1]\}})=\Om_{\nu+1}(1)-\Om(1)$. By changing the notions, we 
showed that
\begin{itemize}
\vspace{-5pt}
\setlength{\itemindent}{1em}
\addtolength{\itemsep}{-0.7em}
\item[$(\rom{4})$] If $m=1$, $\de<\nu_M$, $\phi([0])=\partial[[\Om_0]]$, $\phi([1])=\partial[[\Om_1]]$, $\psi([1]\otimes [0])=\partial[[\Om^{\pr}_0]]$, $\psi([1]\otimes [1])=\partial[[\Om^{\pr}_1]]$, then
$$F_A(\psi|_{I(1, \ti{N})_0\times\{[0]\}})=[[\Om^{\pr}_0-\Om_0]],\quad F_A(\psi|_{I(1, \ti{N})_0\times \{[1]\}})=[[\Om^{\pr}_1-\Om_1]].$$
\end{itemize}

\vspace{0.5em}
\noindent \underline{\bf Step 5}:
We now pick up the subscript $``j"$. For each $\phi_j$, $j\geq J$, we can construct $\psi_j: I(1, \ti{N})_0\times I_0(m, l)_0\rightarrow \Z_n(M^{n+1})$ as above. Denote $\phi_j(x)=\partial [[\Om_j(x)]]$, and $\psi_j(y, x)=R_{j, i}(x)=\partial[[\Om_{j, i}(x)]]$ for $y=[\frac{i}{3^{\ti{N}}}]$, with $\Om_j(x), \Om_{j, i}(x)\in\C(M)$. By the construction (\ref{Om1})(\ref{Omi})(\ref{Om final}), $\Om_{j, i}(x)-\Om_j(x)$ are all supported in $B(p_j, r_j)$. Recall that $r_j\rightarrow 0$ by Claim \ref{claim 1 for case 2}, so
$$\F\big(\psi_j(y, x), \phi_j(x)\big)\leq \M\big(\Om_{j, i}(x)-\Om_j(x)\big)\leq vol\big(B(p_j, r_j)\big)\rightarrow 0,$$
uniformly for all $(y, x)\in I(1, \ti{N})_0\times I_0(m, l)_0$ as $j\rightarrow\infty$.

Define
$$\rho_j=\ep_j+\max\{\F\big(\psi_j(y, x), \phi_j(x)\big): (y, x)\in I(1, \ti{N})_0\times I_0(m, l)_0\},$$
where $\ep_j$ is given in Lemma \ref{lemma 1 for technical 1}; then $\rho_j\rightarrow 0$, as $j\rightarrow \infty$, and $\F(\psi_j(y, x), T)\leq \F(\psi_j(y, x), \phi_j(x))+\F(\phi_j(x), T)\leq \rho_j$, so
$$\psi_j: I(1, \ti{N})_0\times I_0(m, l)_0\rightarrow \B^{\F}_{\rho_j}(T).$$

Finally, we claim that
\begin{itemize}
\vspace{-5pt}
\setlength{\itemindent}{1em}
\addtolength{\itemsep}{-0.7em}
\item[$(\rom{2})$]
\begin{equation}\label{varifold convergence away from the singular point}
\lim_{j\rightarrow\infty}|\psi_j([1], x)|=V(x)\lc G_n\big(M\backslash \{q\}\big),\ \textrm{as varifolds}.
\end{equation}
\end{itemize}
In fact, by (\ref{R final}), $\psi_j([1], x)=\phi_j(x)$ outside $B(p_j, r_j)$, and inside $B(p_j, r_j)$, by (\ref{mass bound for final cone}) and Claim \ref{claim 1 for case 2}(\rom{3})(\rom{5}),
$$\M\big(\psi_j([1], x)\lc B(p_j, r_j)\big)\leq 2r_j n^{-1}\M\big(\partial (T_j(x)\lc B(p_j, r_j))\big)\rightarrow 0, \textrm{ as }j\rightarrow\infty.$$
Therefore (\ref{varifold convergence away from the singular point}) is a directly corollary of Claim \ref{claim 1 for case 2}(\rom{6}).

\vspace{0.5em}
All the above properties show that $\{\psi_j\}$ satisfy Lemma \ref{case 2} when $S_{con}=\{q\}$.

\vspace{1em}
\noindent{\bf\large Part \Rom{2}:}
If $S_{con}$ contains more than one point, we can construct $\psi_j$ successively on the pairwise disjoint neighborhoods $\{Z_q: q\in S_{con}\}$ as above, as the construction is purely local. The only things to be taken care of are the increase of mass and fineness. 

Write $S_{con}=\{q_a\}_{a=1}^{\ka}$,  $Z_a=Z_{q_a}$, $\ka\in\N$. As mentioned above, $\ka\leq C(m, l)\frac{2L}{\al}$ depends only on $m, l, \de, L$. We start by following the above process inside $Z_1$ to extend $\phi_j$ (possibly up to a subsequence) to $\psi^1_j: I(1, \ti{N})_0\times I_0(m, l)_0\rightarrow \B^{\F}_{\rho^1_j}(T)$,  where $\{\rho^1_j\}$ is a sequence of positive numbers converging to zero. Denote $\phi^1_j(\cdot)=\psi^1_j([1], \cdot)$. Then $\{\psi^1_j\}$ satisfy (by Step 4 and Step 5 in Part 1): for all $x\in I_0(m, l)_0$
\begin{itemize}
\vspace{-5pt}
\setlength{\itemindent}{1em}
\addtolength{\itemsep}{-0.7em}
\item $\psi^1_j([\frac{i}{3^{\ti{N}}}])=\partial[[\Om^1_{j, i}(x)]]$, $\Om^1_{j, i}(x)\in \C(M)$;

\item $\f(\psi^1_j)\leq \de/5+3\de_1/5$ if $m=1$, and $\f(\psi^1_j)\leq \max\{\de/5+3\de_1/5, \f(\phi_j)+4\de_1/5\}$ if $m>1$;

\item $\psi_j^1([0], x)=\phi_j(x)$, $\lim_{j\rightarrow\infty}|\psi^1_j([1], x)|=V(x)\lc G_n(M\backslash\{q_1\})$ as varifolds;

\item $\max\{\M\big(\psi^1_j([\frac{i}{3^{\ti{N}}}], x)\big)\}\leq \max\{\M\big(\phi_j(x)\big)\}+2\de_1/5$;

\item If $m=1$, and denote $\phi_j([0])=\partial[[\Om_{j, 0}]]$, $\phi_j([1])=\partial[[\Om_{j, 1}]]$, $\phi^1_j([0])=\partial[[\Om^1_{j, 0}]]$, $\phi^1_j([1])=\partial[[\Om^1_{j, 1}]]$; then
$$F_A(\psi^1_j|_{I(1, \ti{N})_0\times\{[0]\}})=[[\Om^1_{j, 0}-\Om_{j, 0}]],\quad F_A(\psi^1_j|_{I(1, \ti{N})_0\times \{[1]\}})=[[\Om^1_{j, 1}-\Om_{j, 1}]].$$
\end{itemize}

Also $\{\phi^1_j\}$ satisfy: for all $x\in I_0(m, l)_0$,
\begin{itemize}
\vspace{-5pt}
\setlength{\itemindent}{1em}
\addtolength{\itemsep}{-0.7em}
\item  $\phi^1_j(x)=\phi_j(x)$ outside $Z_1$, by (\ref{R final});

\item $\phi^1_j(x)=\partial[[\Om^1_j(x)]]$, $\Om^1_j(x)\in\C(M)$;

\item $\lim_{j\rightarrow\infty}|\phi^1_j(x)|=V(x)\lc G_n(M\backslash\{q_1\})$, as varifolds;

\item $\M\big(\phi^1_j(x)\big)\leq \M\big(\phi_j(x)\big)+\de_1/5$, by (\ref{case 2 mass estimate final});

\item If $m>1$, $\f(\phi^1_j)\leq \f(\phi_j)+2\de_1/5$, by (\ref{case 2 fineness estimate final}).
\end{itemize}





As $\phi^1_j(x)=\phi_j(x)$ outside $Z_1$, for all $x\in I_0(m, l)_0$, we can repeat the construction in Part \Rom{1} inductively on $Z_2, \cdots, Z_\ka$, to get (possibly up to subsequences) $\{\psi^a_j\}$ and $\{\phi^a_j\}$, $2\leq a\leq \ka$, such that $\psi^a_j: I(1, \ti{N})_0\times I_0(m, l)_0\rightarrow \B^{\F}_{\rho^a_j}(T)$, $\phi^a_j: I_0(m, l)_0\rightarrow \B^{\F}_{\rho^a_j}(T)$, with $\{\rho^a_j\}$ a sequence of positive numbers converging to zero as $j\rightarrow \infty$ for each $2\leq a\leq \ka$, and $\phi^a_j(x)=\psi^a_j([1], x)$, and the following statements are true. For each $2\leq a\leq \ka$,

$\{\psi^a_j\}$ satisfy that:  for all $x\in I_0(m, l)_0$,
\begin{enumerate}
\vspace{-5pt}
\setlength{\itemindent}{1em}
\addtolength{\itemsep}{-0.7em}
\item $\psi^a_j([\frac{i}{3^{\ti{N}}}])=\partial[[\Om^a_{j, i}(x)]]$, $\Om^a_{j, i}(x)\in \C(M)$;

\item $\f(\psi^a_j)\leq \de/5+3\de_1/5$ if $m=1$, and if $m>1$, $\f(\psi^a_j)\leq \max\{\de/5+3\de_1/5, \f(\phi^{a-1}_j)+4\de_1/5\}$, so by property 5 of $\phi^a_j$ (see below),
$$\f(\psi^a_j)\leq \max\{\de/5+3\de_1/5, \f(\phi_j)+2(a+1)\de_1/5\};$$

\item $\psi^a_j([0], x)=\phi^{a-1}_j(x)$, $\lim_{j\rightarrow\infty}|\psi^a_j([1], x)|=V(x)\lc G_n(M\backslash\{q_1,\cdots, q_a\})$ as varifolds;

\item $\max\{\M\big(\psi^a_j([\frac{i}{3^{\ti{N}}}], x)\big)\}\leq \max\{\M\big(\phi^{a-1}_j(x)\big)\}+2\de_1/5$, hence by property 4 of $\phi^a_j$ (see below),
$$\max\{\M\big(\psi^a_j([\frac{i}{3^{\ti{N}}}], x)\big)\}\leq \max\{\M\big(\phi_j(x)\big)\}+(a+1)\de_1/5;$$

\item If $m=1$, and denote $\phi^{a-1}_j([0])=\partial[[\Om^{a-1}_{j, 0}]]$, $\phi^{a-1}_j([1])=\partial[[\Om^{a-1}_{j, 1}]]$, $\phi^a_j([0])=\partial[[\Om^a_{j, 0}]]$, $\phi^a_j([1])=\partial[[\Om^a_{j, 1}]]$; then
$$F_A(\psi^a_j|_{I(1, \ti{N})_0\times\{[0]\}})=[[\Om^a_{j, 0}-\Om^{a-1}_{j, 0}]],\quad F_A(\psi^a_j|_{I(1, \ti{N})_0\times \{[1]\}})=[[\Om^a_{j, 1}-\Om^{a-1}_{j, 1}]].$$
\end{enumerate}

$\{\phi^a_j\}$ satisfy: for all $x\in I_0(m, l)_0$,
\begin{enumerate}
\vspace{-5pt}
\setlength{\itemindent}{1em}
\addtolength{\itemsep}{-0.7em}
\item  $\phi^a_j(x)=\phi_j(x)$ outside $Z_1\cup\cdots\cup Z_a$, by (\ref{R final});

\item $\phi^a_j(x)=\partial[[\Om^a_j(x)]]$, $\Om^a_j(x)\in\C(M)$;

\item $\lim_{j\rightarrow\infty}|\phi^a_j(x)|=V(x)\lc G_n(M\backslash\{q_1,\cdots, q_a\})$, as varifolds;

\item $\M\big(\phi^a_j(x)\big)\leq \M\big(\phi^{a-1}_j(x)\big)+\de_1/5$ by (\ref{case 2 mass estimate final}), so
$$\M\big(\phi^a_j(x)\big)\leq \M\big(\phi_j(x)\big)+a\de_1/5;$$

\item If $m>1$, $\f(\phi^a_j)\leq \f(\phi^{a-1}_j)+2\de_1/5$ by (\ref{case 2 fineness estimate final}), so
$$\f(\phi^a_j)\leq \f(\phi_j)+2a\de_1/5.$$
\end{enumerate}

Finally, let $\ka(\nu+1)=3^{N_2}$, for some $N_2\in\N$, with $\nu$ given in Fact \ref{fact 2 of case 2}; then $N_2$ depends only on $m, l, \de, L$. Recall that $\nu+1=3^{\ti{N}}$ (see Step 4 in Part \Rom{1}); then we can define $\psi_j: I(1, N_2)_0\times I_0(m, l)_0\rightarrow\B^{\F}_{\rho^{\ka}_j}(T)$ as:
\begin{equation}
\psi_j([\frac{i}{3^{N_2}}], x)=\psi^a_j([\frac{i-(a-1)(\nu+1)}{3^{\ti{N}}}], x), \quad \textrm{ if } (a-1)(\nu+1)\leq i\leq a(\nu+1).
\end{equation}
Choose $\de_1<\de$, such that 
$$2(\ka+1)\de_1/5\leq \de, (\ka+1)\de_1/5\leq \frac{\de}{n_0+1},$$ 
and let $\rho_j=\rho^{\ka}_j$; then $\psi_j$ satisfy (0)(\rom{1})(\rom{2})(\rom{3}) in Lemma \ref{case 2}. To check Lemma \ref{case 2}(\rom{4}), if $m=1$, by the definition of Almgren's isomorphism (\ref{Almgren isomorphism1}),
$$F_A(\psi_j|_{I(1, N_2)_0\times\{[0]\}})=\sum_{a=1}^{\ka}F_A(\psi^a_j|_{I(1, \ti{N})_0\times\{[0]\}})=\sum_{a=1}^{\ka}[[\Om^a_{j, 0}-\Om^{a-1}_{j, 0}]]=[[\Om^{\ka}_{j, 0}-\Om_{j, 0}]].$$
Similarly, $F_A(\psi_j|_{I(1, N_2)_0\times\{[1]\}})=[[\Om^{\ka}_{j, 1}-\Om_{j, 1}]]$. So Lemma \ref{case 2}(\rom{4}) is true by noticing that $\psi_j([1], [0])=\partial[[\Om^{\ka}_{j, 0}]]$ and $\psi_j([1], [1])=\partial[[\Om^{\ka}_{j, 1}]]$. The proof of Lemma \ref{case 2} is now finished.

\end{proof}

Now let us go back to the proof of Lemma \ref{lemma 1 for technical 1}. If $S_{con}=\emptyset$, then $\psi_j$ can be constructed by Lemma \ref{case 1} with $\rho_j=\ep_j$, $N=N_1$. If $S_{con}\neq \emptyset$, let $\de^{\pr}=\de/2$, and construct (possibly up to a subsequence) $\psi^2_j: I(1, N_2)_0\times I_0(m, l)_0\rightarrow \B^{\F}_{\rho_j}(T)$ by Lemma \ref{case 2} for the set of numbers $l, m, \de^{\pr}, L$. Then denote $\phi_j^{\pr}(\cdot)=\psi^2_j([1], \cdot): I_0(m, l)_0\rightarrow \B^{\F}_{\rho_j}(T)$. By Lemma \ref{case 2}(\rom{2}), $\{\phi^{\pr}_j\}$ satisfy the requirement of Lemma \ref{case 1} for the set of numbers $l, m, \de^{\pr}, L+\frac{\de^{\pr}}{n_0+1}$. Now we can apply Lemma \ref{case 1} to $\{\phi^{\pr}_j\}$, and construct (possibly up to a subsequence) $\psi^1_j: I(1, N_1)_0\times I_0(m, l)_0\rightarrow \B^{\F}_{\rho_j}(T)$.

Assume $3^{N}=3^{N_1}+3^{N_2}$, $N\in\N$; then $N$ depends only on $(l, m, T, \de, L)$, as $N_2$ depends only on $(l, m, \de/2, L)$ and $N_1$ depends only on $(l, m, T, \de/2, L+\frac{\de}{2(n_0+1)})$. Define $\psi_j: I(1, N)_0\times I_0(m, l)_0\rightarrow \B^{\F}_{\rho_j}(T)$ by
\begin{displaymath}
\begin{split}
 & \psi_j([\frac{i}{3^N}], x)=\psi^2_j([\frac{i}{3^{N_2}}], x), \quad \textrm{ if } 0\leq i\leq 3^{N_2};\\
 & \psi_j([\frac{i}{3^N}], x)=\psi^1_j([\frac{i-3^{N_2}}{3^{N_1}}], x), \quad \textrm{ if } 3^{N_2}\leq i\leq 3^{N}.
\end{split}
\end{displaymath}
Then $\{\psi_j\}$ satisfy Lemma \ref{lemma 1 for technical 1}(0)(\rom{2})(\rom{4}) by combining Lemma \ref{case 1}(0)(\rom{2})(\rom{4}) with Lemma \ref{case 2} (0)(\rom{2})(\rom{4}). For Lemma \ref{lemma 1 for technical 1}(\rom{1}), if $m=1$,
$$\f(\psi_j)\leq \max\{\f(\psi^1_j), \f(\psi^2_j)\}\leq \de/2;$$
if $m>1$, then by Lemma \ref{case 1}(\rom{1}) and Lemma \ref{case 2}(\rom{1}),
$$\f(\psi_j)\leq \max\{\f(\psi^1_j), \f(\psi^2_j)\}\leq \f(\psi^2_j)+\de/2\leq \f(\phi_j)+\de.$$
For Lemma \ref{lemma 1 for technical 1}(\rom{3}), by Lemma \ref{case 1}(\rom{3}) and Lemma \ref{case 2}(\rom{3}),
\begin{displaymath}
\begin{split}
\max\big\{\M(\psi_j(\cdot, \cdot))\big\} & \leq \max\big\{\max\{\M(\psi^1_j(\cdot, \cdot))\}, \max\{\M(\psi^2_j(\cdot, \cdot))\}\big\}\\
                                               & \leq \max\{\M(\psi^2_j(\cdot, \cdot))\}+\frac{\de}{2(n_0+1)}\leq \max\{\M(\phi_j(\cdot))\}+\frac{\de}{n_0+1}.
\end{split}
\end{displaymath}
So we finished checking that $\{\psi_j\}$ satisfy Lemma \ref{lemma 1 for technical 1}(0)(\rom{1})(\rom{2})(\rom{3})(\rom{4}). 
\end{proof}

Now let us go back to the proof of Proposition \ref{technical 1}. This part is similar to the final part of \cite[13.3]{MN12}. We will use notions in Lemma \ref{lemma 1 for technical 1}. We are going to construct the extensions $\ti{\phi}_j$ of $\phi_j$ from $I(m, k_0)_0$ to $\B^{\F}_{\rho_j}(T)$ for every $j$ large enough, therefore get a contradiction.

First let us discuss the case when $m>1$. Let
$$\hat{\phi}_j: I_0(m, N)_0\times I(1, N)_0\rightarrow \B^{\F}_{\rho_j}(T),$$
be defined by $\hat{\phi}_j(x, y)=\psi_j(y, \n(N, l)(x))$, where $\psi_j$ are constructed in Lemma \ref{lemma 1 for technical 1}. Recall that $S(m+1, N)_0=I_0(m, N)_0 \times I(1, N)_0$. We can extend $\hat{\phi}_j$ to
$$S(m+1, N)_0\cup T(m+1, N)_0,$$
by assigning it to $T$ on $T(m+1, N)_0$.

Now recall the map $\mr(N): I(m, N+q)_0\rightarrow S(m+1, N)_0\cup T(m+1, N)_0$ defined in \cite[Appendix C]{MN12}, which satisfies: $q$ depends on $m$ but not on $N$; if $x, y \in I(m, N+q)_0$, $\md(x, y)=1$, then $\md\big(\mr(N)(x), \mr(N)(y)\big)\leq m$; if $x\in I_0(m, N+q)_0$, then $\mr(N)(x)\in [0]\times I_0(m, N)_0$ and $\mr(N)(x)=\n(N+q, N)(x)$.

With out loss of generality, we can assume $k_0>N+q$; then the extension $\ti{\phi}_j: I(m, k_0)_0\rightarrow \B^{\F}_{\rho_j}(T)$ is defined by
$$\ti{\phi}_j=\hat{\phi}_j\circ \mr_m(N)\circ \n(k_0, N+q),$$
for which Proposition \ref{technical 1}(\rom{1})(\rom{2})(\rom{3}) are easily seen true by Lemma \ref{lemma 1 for technical 1}(\rom{1})(\rom{2})(\rom{3}).

\vspace{0.5em}

Finally when $m=1$, define $\hat{\phi}_j: I(1, N+1)_0\rightarrow \B^{\F}_{\rho_j}(T)$ by:
\begin{displaymath}
\begin{split}
& \hat{\phi}_j([\frac{i}{3^{N+1}}])=\psi_j([\frac{i}{3^{N}}], [0]), \quad \textrm{ if } 0\leq i\leq 3^N;\\
& \hat{\phi}_j([\frac{i}{3^{N+1}}])=T, \quad\quad \textrm{ if } 3^N+1\leq i\leq 2\cdot 3^N;\\
& \hat{\phi}_j([\frac{i}{3^{N+1}}])=\psi_j([\frac{3^{N+1}-i}{3^{N}}], [1]), \quad \textrm{ if } 2\cdot 3^N+1\leq i\leq 3^{N+1},
\end{split}
\end{displaymath}
for which Proposition \ref{technical 1}(\rom{1})(\rom{2})(\rom{3}) are automatically true by Lemma \ref{lemma 1 for technical 1}(\rom{1})(\rom{2})(\rom{3}). To check Proposition \ref{technical 1}(\rom{4}), by the definition of Almgren's isomorphism (\ref{Almgren isomorphism1}) and Lemma \ref{lemma 1 for technical 1}(\rom{4}),
\begin{displaymath}
\begin{split}
F_A(\hat{\phi}_j) & =F_A\big(\psi_j|_{I(1, N)_0\times\{[0]\}}\big)-F_A\big(\psi_j|_{I(1, N)_0\times\{[1]\}}\big)\\
                           & =[[\Om_T-\Om_{j, 0}]]-[[\Om_T-\Om_{j, 1}]]=[[\Om_{j, 1}-\Om_{j, 0}]].
\end{split}
\end{displaymath}
For $k_0>N+1$, the extension $\ti{\phi}_j: I(1, k_0)_0\rightarrow \B^{\F}_{\rho_j}(T)$ is given by $\hat{\phi}_j\circ\n(k_0, N+1)$.
\end{proof}

\vspace{1em}
The next result removes the dependence of $\ep$ and $k$ on the parameters $l, m$ in Proposition \ref{technical 1}, which is analogous to \cite[13.5]{MN12}. The idea is to apply Proposition \ref{technical 1} inductively along the $p$-skeletons of $I(m, l)$, $1\leq p\leq m$. In the induction process, compared to \cite[13.5]{MN12} where they need to pay attention to the increase of the parameter $``\mathbf{m}(\phi, r)"$\footnote{This parameter measures the local mass density. See \cite[4.2]{MN12}.}, we need to take care of the increase of the size of the neighborhoods around $T$.

Fix $n_0\in\N$. $b(n_0)$ is a constant depending only on $n_0$.

\begin{proposition}\label{technical 2}
Given $\de, L>0$, and
$$T\in\Z_n(M)\cap \{S: \M(S)\leq 2L-\de\}, \textrm{with}\ T=\partial[[\Om_T]],$$
$\Om_T\in\C(M)$, 
then there exist $0<\ep=\ep(T, \de, L)<\de$, and $k=k(T, \de, L)\in\N$, and a function $\rho=\rho_{(T, \de, L)}: \R^1_+\rightarrow \R^1_+$, with $\rho(s)\rightarrow 0$, as $s\rightarrow 0$, such that: given $l, m\in\N$, $m\leq n_0+1$, $0<s<\ep$, and
\begin{equation}\label{phi to be extended 2}
\phi: I_0(m, l)_0\rightarrow \B^{\F}_s(T)\cap\{S: \M(S)\leq 2L-\de\},\ \textrm{with } \phi(x)=\partial[[\Om_x]],
\end{equation}
$\Om_x\in\C(M)$, 
$x\in I_0(m, l)_0$, there exists
$$\ti{\phi}: I(m, l+k)_0\rightarrow \B^{\F}_{\rho(s)}(T),\ \textrm{with }\ti{\phi}(y)=\partial[[\Om_y]],$$
$\Om_y\in\C(M)$, 
$y\in I(m, l+k)_0$, and satisfying
\begin{itemize}
\setlength{\itemindent}{1em}
\addtolength{\itemsep}{-0.7em}
\item[$(\rom{1})$] $\f(\ti{\phi})\leq \de$ if $m=1$, and $\f(\ti{\phi})\leq b(n_0)(\f(\phi)+\de)$ if $m>1$;
\item[$(\rom{2})$] $\ti{\phi}=\phi\circ \n(l+k, l)$ on $I_0(m, l+k)_0$;
\item[$(\rom{3})$] $$\sup_{x\in I(m, l+k)_0}\M\big(\ti{\phi}(x)\big)\leq \sup_{x\in I_0(m, l)_0}\M\big(\phi(x)\big)+\de;$$
\item[$(\rom{4})$] If $m=1$, $\de<\nu_M$, $\phi([0])=\partial[[\Om_0]]$, $\phi([1])=\partial[[\Om_1]]$, then
$$F_A(\ti{\phi})=[[\Om_1-\Om_0]]$$
where $F_A$ is the Almgren's isomorphism (\ref{Almgren isomorphism1}).
\end{itemize}
\end{proposition}
\begin{proof}
The case $m=1$ follows directly from Proposition \ref{technical 1}. In fact, take $\ep=\ep(0, 1, T, \de, L)$, $k=k(0, 1, T, \de, L)$ and $\rho(s)=\rho_{(0, 1, T, \de, L)}(s)$ by Proposition \ref{technical 1}, and denote the extension by $\ti{\phi}_1: I(1, k)_0\rightarrow \B^{\F}_{\rho(s)}(T)$. Then $\ti{\phi}: I(1, l+k)_0\rightarrow \B^{\F}_{\rho(s)}(T)$ is given by $\ti{\phi}=\ti{\phi}_1\circ \n(l+k, k)$. The fact that $\ti{\phi}$ satisfies properties (\rom{1})(\rom{2})(\rom{3})(\rom{4}) follows from the fact that $\ti{\phi}_1$ satisfies Proposition \ref{technical 1}(\rom{1})(\rom{2})(\rom{3})(\rom{4}).

Now let us assume that $m>1$. Using notations in Proposition \ref{technical 1}, we can inductively define integers,
$$k_0=0, k_1=k(0, 1, T, \de, L), \cdots, k_i=k(k_{i-1}, i, T, \de, L), \cdots, k_m=k(k_{m-1}, m, T, \de, L);$$
and positive numbers,
$$\ep_1=\ep(0, 1, T, \de, L), \cdots, \ep_i=\ep(k_{i-1}, i, T, \de, L), \cdots, \ep_m=\ep(k_{m-1}, m, T, \de, L);$$
and functions from $\R^1_+$ to $\R^1_+$,
$$\rho_1=\rho_{(0, 1, T, \de, L)}, \cdots, \rho_i=\rho_{(k_{i-1}, i, T, \de, L)}\circ \rho_{i-1}, \cdots, \rho_m=\rho_{(k_{m-1}, m, T, \de, L)}\circ \rho_{m-1}.$$

As $\lim_{s\rightarrow 0}\rho_{(k_{i-1}, i, T, \de, L)}(s)=0$, for $1\leq i\leq m$, we know that $\lim_{s\rightarrow 0}\rho_i(s)=0$, for all $1\leq i\leq m$. Hence we can choose $\ep>0$, such that $\ep\leq \min\{\ep_1, \cdots, \ep_m\}$, and
$$\ti{\ep}_i :=\max_{0\leq s\leq \ep}\rho_i(s)\leq \ep_{i+1},\quad \textrm{for all }1\leq i\leq m-1.$$
Let $k=k_m$, and $\rho=\rho_m$; then $\ep, k, \rho$ depend only on $T, \de, L$. In the following, we will show that $\ep, k, \rho$ satisfy the requirement.

Fix a map $\phi: I_0(m, l)_0\rightarrow \B^{\F}_s(T)\cap\{S: \M(S)\leq 2L-\de\}$, with $\phi(x)=\partial[[\Om_x]]$, $\Om_x\in\C(M)$, for all $x\in I_0(m, l)_0$. Assume that $s\leq \ep$. Given $p\leq m$, let $V_p$ be the set of vertices of $I(m, l+k_p)$ that belong to the $p$-skeleton of $I(m, l)$, i.e. $V_p=\cup_{\al\in I(m, l)_p}\al(k_p)_0$. Clearly $V_m=I(m, l+k)_0$. Say a map $\phi_p: V_p\rightarrow \B^{\F}_{\rho_p(s)}(T)\cap\{S: \M(S)\leq 2L\}$ is a $p$-extension of $\phi$, if:
\begin{enumerate}
\vspace{-5pt}
\setlength{\itemindent}{1em}
\addtolength{\itemsep}{-0.7em}
\item $\phi_p(y)=\partial[[\Om_y]]$, $\Om_y\in\C(M)$, for all $y\in V_p$;

\item $\phi_p=\phi\circ \n(l+k_p, l)$ on $V_p \cap I_0(m, l+k_p)_0$;

\item If $p=1$, then $\f(\phi_p)\leq \f(\phi)+\de$; if $p>1$, there exists a $(p-1)$-extension $\phi_{p-1}$ of $\phi$, such that
$$\f(\phi_p)\leq p\big(\f(\phi_{p-1})+\de\big);$$

\item $\sup_{y\in V_p}\M\big(\phi_p(y)\big)\leq \sup_{x\in I_0(m, l)_0}\M\big(\phi(x)\big)+\frac{p\de}{n_0+1}$.
\end{enumerate}

We start with the construction of $1$-extension $\phi_1$ of $\phi$. First construct a trivial extension of $\phi$ to $I(m, l)_0$, i.e. $\phi_0: I(m, l)_0\rightarrow \B^{\F}_{s}(T)\cap\{S: \M(S)\leq 2L-\de\}$ by
\begin{displaymath}
\begin{split}
& \phi_0(x)=\phi(x), \quad x\in I_0(m, l)_0;\\
& \phi_0(x)=T,\quad\quad x\notin I_0(m, l)_0.
\end{split}
\end{displaymath}
Then we can construct $\ti{\phi}_0: V_1\rightarrow \B^{\F}_{\rho_1(s)}(T)$ as follows: given $\al\in I(m, l)_1$, $\ti{\phi}_0|_{\al(k_1)_0}$ is gotten by extending $\phi_0|_{\al_0}$ on $\al_0$ to $\al(k_1)_0$ using Proposition \ref{technical 1} for $l=0, m=1, T, \de, L$ as $s\leq \ep\leq \ep_1$. Finally, we can define $\phi_1: V_1\rightarrow \B^{\F}_{\rho_1(s)}(T)$ by
\begin{displaymath}
\begin{split}
& \phi_1=\phi_0\circ \n(l+k_1, l), \quad \textrm{on $\al(k_1)_0$, if $\al$ is a 1-cell of $I_0(m, l)$};\\
& \phi_1=\ti{\phi}_0,\quad\quad\quad \textrm{on $\al(k_1)_0$, if $\al$ is not a 1-cell of $I_0(m, l)$}.
\end{split}
\end{displaymath}
It is easy to check that $\phi_1$ is a $1$-extension of $\phi$.

To get $p$-extension  inductively, we need the following lemma.
\begin{lemma}\label{p-extension lemma}
Given a $p$-extension $\phi_p$ of $\phi$, $p\leq m-1$, there exists a $(p+1)$-extension $\phi_{p+1}$ of $\phi$.
\end{lemma}
\begin{proof}
By assumption $\phi_p$ maps $V_p$ into $\B^{\F}_{\rho_p(s)}(T)\cap \{S: \M(S)\leq 2L\}$, so the image of $\phi_p$ also lie in $\B^{\F}_{\ep_{p+1}}(T)\cap \{S: \M(S)\leq 2L\}$ as $\rho_p(s)\leq \ti{\ep}_p\leq \ep_{p+1}$. Using the fact that $\phi_p(x)=\partial[[\Om_x]]$, $\Om_x\in\C(M)$ for all $x\in V_p$, we can apply Proposition \ref{technical 1} for each $(p+1)$-cell $\al\in I(m, l)_{p+1}$ to extend $\phi_p|_{\al_0(k_p)_0}$ to $\ti{\phi}_{p, \al}: \al(k_{p+1})_0\rightarrow \B^{\F}_{\rho_{p+1}(s)}(T)$ for $l=k_p, m=p+1, T, \de, L$. Given any two adjacent $(p+1)$-cells $\al, \bar{\al}\in I(m, l)_{p+1}$, by Proposition \ref{technical 1}(\rom{2}), $\ti{\phi}_{p, \al}=\ti{\phi}_{p, \bar{\al}}=\phi_p \circ\n(k_{p+1}, k_p)$ on $\al(k_{p+1})_0\cap \bar{\al}(k_{p+1})_0$, so we can construct a map
$$\ti{\phi}_p: V_{p+1}\rightarrow \B^{\F}_{\rho_{p+1}(s)}(T),$$
by letting $\ti{\phi}_p=\ti{\phi}_{p, \al}$ on each $\al(k_{p+1})$, $\al\in I(m, l)_{p+1}$.
By Proposition \ref{technical 1}(\rom{1})(\rom{3}) and the inductive hypothesis 4, 
\begin{displaymath}
\begin{split}
& \quad\quad\f(\ti{\phi}_p)\leq (p+1)\big(\f(\phi_p)+\de\big);\\
& \sup_{x\in V_{p+1}}\M\big(\ti{\phi}_p(x)\big)\leq \sup_{x\in V_p}\M\big(\phi_p(x)\big)+\frac{\de}{n_0+1}\leq \sup_{x\in I_0(m, l)_0}\M\big(\phi(x)\big)+\frac{(p+1)\de}{n_0+1}.
\end{split}
\end{displaymath}
Finally we define $\phi_{p+1}: V_{p+1}\rightarrow \B^{\F}_{\rho_{p+1}(s)}(T)$ by
\begin{displaymath}
\begin{split}
& \phi_{p+1}=\phi_p\circ \n(l+k_{p+1}, l+k_p), \quad \textrm{on $\al(k_{p+1})_0$, if $\al$ is a $(p+1)$-cell of $I_0(m, l)$};\\
& \phi_{p+1}=\ti{\phi}_p,\quad\quad\quad \textrm{on $\al(k_{p+1})_0$, if $\al$ is not a $(p+1)$-cell of $I_0(m, l)$}.
\end{split}
\end{displaymath}
Now we check that $\phi_{p+1}$ satisfies all the requirements for a $(p+1)$-extension of $\phi$. First, by construction $\phi_{p+1}(x)=\partial[[\Om_x]]$, $\Om_x\in\C(M)$, for all $x\in V_{p+1}$; second, given a $(p+1)$-cell $\al$ in $I_0(m, l)$, by inductive hypothesis 2, $\phi_{p+1}=\phi_p\circ \n(l+k_{p+1}, l+k_p)=\phi\circ \n(l+k_p, l)\circ \n(l+k_{p+1}, l+k_p)=\phi\circ\n(l+k_{p+1}, l)$ on $\al(k_{p+1})_0$; lastly, as $\phi_{p+1}$ is gotten by replacing $\ti{\phi}_p$ by $\phi_p\circ \n(l+k_{p+1}, l+k_p)$ on $V_{p+1}\cap I_0(m, l+k_{p+1})_0$, hence $\f(\phi_{p+1})\leq \f(\ti{\phi}_p)\leq (p+1)\big(\f(\phi_p)+\de\big)$, and $\sup_{x\in V_{p+1}}\M\big(\phi_{p+1}(x)\big)\leq \sup_{x\in V_{p+1}}\M\big(\ti{\phi}_p(x)\big)\leq \sup_{x\in I_0(m, l)_0}\M\big(\phi(x)\big)+\frac{(p+1)\de}{n_0+1}$.
\end{proof}

We can then inductively construct an $m$-extension $\phi_m: I(m, l+k_m)_0\rightarrow \B^{\F}_{\rho_m(s)}(T)$. Let $\ti{\phi}=\phi_m$; then it is easy to see that $\ti{\phi}, \ep, k=k_m, \rho=\rho_m$ satisfy all the requirements of Proposition \ref{technical 2}.
\end{proof}

\subsection{Proof of Theorem \ref{discretization and identification}.}
The idea is briefly as follows. Denote
\begin{displaymath}
L(\Phi)=\max_{x\in[0, 1]}\M\big(\Phi(x)\big).
\end{displaymath}
Given a $\de>0$, we can cover the set $\Z_n(M^{n+1})\cap\{S: \M(S)\leq 2L(\Phi)\}\cap\{S: S=\partial[[\Om]]: \Om\in\C(M)\}$ by finitely many balls $\{\B_{\ep_i}^{\F}(T_i)\}_{i=1}^N$, such that Proposition \ref{technical 2} can be applied on each ball for $n_0=1, T_i, \de, L=L(\Phi)$\footnote{Note that $n_0=1$ is the dimension of parameter space.}. Take $j$ large enough, such that for each $1$-cell $\al\in I(1, j)_1$, the image $\Phi(\al)$ lie in some $\B^{\F}_{\ep_i}(T_i)$; then we can apply Proposition \ref{technical 2} to each $\Phi|_{\al_0}$, and construct a discrete map $\phi_\de$ which has fineness controlled by $\de$, and total mass bounded by $L(\Phi)+\de$. Finally, taking a sequence $\de_i\rightarrow 0$, $i\rightarrow \infty$, we can construct the desired $(1, \M)$-homotopy sequence $\{\phi_i\}_{i\in\N}$ by letting $\phi_i=\phi_{\de_i}$. Detailed argument is given as below.

\begin{proof}
(of Theorem \ref{discretization and identification}) In this part, we will repeatedly use notations and conclusions in Proposition \ref{technical 2} for $n_0=1$\footnote{Again, $n_0=1$ is the dimension of parameter space.}.

\vspace{0.5em}
\noindent{\bf Step \Rom{1}:} Fix $\de>0$, such that $L=L(\Phi)<2L-2\de$. By the weak compactness of the set $\Z_n(M^{n+1})\cap\{S: \M(S)\leq 2L\}\cap\{S: S=\partial[[\Om]]: \Om\in\C(M)\}$ (see \cite[\S 37.2]{Si83}\cite[\S 1.20]{Gi}), we can find a finite covering by balls $\big\{\B^{\F}_{\ep_i}(T_i): i=1, \cdots, N\big\}$, such that $T_i=\partial[[\Om_i]]$, $\Om_i\in\C(M)$, $\M(T_i)\leq 2L$, and
\begin{equation}\label{smallness of radius of the covering}
3\ep_i+\sup_{0\leq s\leq 3\ep_i}\rho_i(s)<\ep(T_i, \de, L).
\end{equation}
where $\ep(T_i, \de, L)$, $k_i=k(T_i, \de, L)$ and $\rho_i(s)=\rho_{(T_i, \de, L)}(s)$ are given by Proposition \ref{technical 2}. Assume that $\ep_1\leq \ep_2\leq \cdots \leq \ep_N\leq \de$, and denote $k=\max\{k_i: 1\leq i\leq N\}$.

By the continuity of $\Phi$ under the flat topology, we can take $j\in\N$ large enough, such that for any $\al\in I(1, j)_1$,
\begin{equation}\label{small fineness wrt flat norm}
\sup_{x, y\in \al}\F\big(\Phi(x)-\Phi(y)\big)<\ep_1<\de.
\end{equation}
Define $c: I(1, j)_0\rightarrow \{1, \cdots, N\}$ by $c(x)=\sup\{i: \Phi(x)\in \B^{\F}_{\ep_i}(T_i)\}$. Then define
\begin{displaymath}
c: I(1, j)_1\rightarrow\{1, \cdots, N\},
\end{displaymath}
by $c(\al)=\sup\{c(x): x\in \al_0\}$.
\begin{claim}\label{image of Phi}
$\Phi(\al)\subset \B^{\F}_{2\ep_{c(\al)}}(T_{c(\al)})$.
\end{claim}
\begin{proof}
By definition, there exists $x\in\al_0$, such that $c(\al)=c(x)$, then $\Phi(x)\in \B^{\F}_{\ep_{c(\al)}}(T_{c(\al)})$. By (\ref{small fineness wrt flat norm}), for any $y\in\al$, $\Phi(y)\in \B^{\F}_{\ep_1}\big(\Phi(x)\big)\subset\B^{\F}_{2\ep_{c(\al)}}(T_{c(\al)})$, as $\ep_1\leq \ep_{c(\al)}$.
\end{proof}

\vspace{5pt}
Let $\phi_0: I(1, j)_0\rightarrow\Z_n(M)$ be the restriction of $\Phi$ to $I(1, j)_0$, then $\phi_0(\al_0)\subset \B^{\F}_{2\ep_{c(\al)}}(T_{c(\al)})$ for all $\al\in I(1, j)_1$. By (\ref{smallness of radius of the covering}) and Theorem \ref{discretization and identification}(a), we can apply Proposition \ref{technical 2} to each $\phi_0|_{\al_0}$, $\al\in I(1, j)_1$, and get
$$\ti{\phi}_{0, \al}: \al(k_{c(\al)})_0\rightarrow \B^{\F}_{\rho_{c(\al)}(2\ep_{c(\al)})}(T_{c(\al)}).$$
Define $\phi_{\de}: I(1, j+k)_0\rightarrow \Z_n(M)$ by
$$\phi_\de=\ti{\phi}_{0, \al}\circ \n(j+k, j+k_{c(\al)}), \quad \textrm{ on } \al(k)_0.$$
Now we collect a few properties of $\phi_{\de}$.
\begin{enumerate}
\vspace{-5pt}
\setlength{\itemindent}{1em}
\addtolength{\itemsep}{-0.7em}

\item $\phi_{\de}=\Phi$ on $I(1, j)_0$;

\item $\phi_{\de}(x)=\partial [[\Om_x]]$, $\Om_x\in\C(M)$, for all $x\in I(1, j+k)_0$;

\item $\f(\phi_{\de})\leq \de$;

\item For any $\al\in I(1, j)_1$,
$$\sup_{x\in \al(k)_0}\M\big(\phi_{\de}(x)\big)\leq \sup_{x\in \al_0}\M\big(\Phi(x)\big)+\de<2L-\de;$$

\item $\sup\big\{\F\big(\phi_{\de}(x)-\Phi(x)\big): x\in I(1, j+k)_0\big\}\leq \de$;

\item If $\de<\nu_M$, then $F_A(\phi_{\de})=[[\Om_1-\Om_0]]$, where $\Phi(0)=\partial[[\Om_0]]$, $\Phi(1)=\partial[[\Om_1]]$.
\end{enumerate}
1 is by construction. 2,3,4,6 directly come from Proposition \ref{technical 2}. 5 comes from (\ref{smallness of radius of the covering}), and the fact that $\phi_{\de}(\al(k)_0)\subset \B^{\F}_{\rho_{c(\al)}(2\ep_{c(\al)})}(T_{c(\al)})$, $\Phi(\al)\subset \B^{\F}_{2\ep_{c(\al)}}(T_{c(\al)})$.

\vspace{0.5em}
\noindent{\bf Step \Rom{2}:}
We say $\bar{\phi}: I(1, \bar{k})_0\rightarrow \Z_n(M)$ is a $(\de, \bar{k})$-extension of $\Phi$, $\bar{k}\geq j+k$, if
\begin{enumerate}
\vspace{-5pt}
\setlength{\itemindent}{1em}
\addtolength{\itemsep}{-0.7em}

\item $\bar{\phi}=\Phi$ on $I(1, j)_0$;

\item $\bar{\phi}(x)=\partial [[\Om_x]]$, $\Om_x\in\C(M)$, for all $x\in I(1, \bar{k})_0$;

\item $\f(\bar{\phi})\leq \de$;

\item For any $\al\in I(1, j)_1$, 
$$\sup_{x\in \al(\bar{k}-j)_0}\M\big(\bar{\phi}(x)\big)\leq  \sup_{x\in \al}\M\big(\Phi(x)\big)+\de<2L-\de;$$

\item $\sup\big\{\F\big(\phi_{\de}(x)-\Phi(x)\big): x\in I(1, \bar{k})_0\big\}\leq \ep_1$.
\end{enumerate}

The following lemma says that a $(\de, \bar{k})$-extension $\bar{\phi}$ is $1$-homotopic to $\phi_{\de}$ with fineness $\de$.
\begin{lemma}\label{1-homotopic lemma}
Given a $(\de, \bar{k})$-extension $\bar{\phi}$ of $\Phi$, with $\bar{k}\geq j+k$, then there exists
$$\psi: I(1, \hat{k})_0\times I(1, \hat{k})_0\rightarrow \Z_n(M),$$
with $\hat{k}=\bar{k}+k$, such that
\begin{itemize}
\vspace{-5pt}
\setlength{\itemindent}{1em}
\addtolength{\itemsep}{-0.7em}

\item[$(a)$] $\psi(y, x)=\partial[[\Om_{y, x}]]$, $\Om_{y, x}\in \C(M)$, for any $(y, x)\in I(1, \hat{k})_0\times I(1, \hat{k})_0$;

\item[$(b)$] $\psi([0], \cdot)=\phi_{\de}\circ \n(\hat{k}, j+k)$, and $\psi([1], \cdot)=\bar{\phi}\circ \n(\hat{k}, \bar{k})$;

\item[$(c)$] $\f(\psi)\leq c_0 \de$, for a fixed constant $c_0$;

\item[$(d)$] $\M\big(\psi(y, x)\big)\leq \sup\big\{\M\big(\Phi(x^{\pr})\big), x, x^{\pr} \textrm{ lie in some common 1-cell }\al\in I(1, j)_1\big\}+2\de$, for any $(y, x)\in I(1, \hat{k})_0\times I(1, \hat{k})_0$.
\end{itemize}
\end{lemma}
\begin{proof}
Given $\al\in I(1, j)_1$, using property 5 for $(\de, \bar{k})$-extension and the fact that $\Phi(\al)\subset \B^{\F}_{2\ep_{c(\al)}}(T_{c(\al)})$, we have $\bar{\phi}\big(\al\cap I(1, \bar{k})_0\big)\subset \B^{\F}_{3\ep_{c(\al)}}(T_{c(\al)})$.

We will first construct $\psi$ on $[0, \frac{1}{3^j}](\hat{k}-j)_0\times I(1, \hat{k})_0$\footnote{Notice that $[0, \frac{1}{3^j}](\hat{k}-j)_0=[0, \frac{1}{3^j}]\cap I(1, \hat{k})_0$.}, such that $\psi$ satisfies:
\begin{displaymath}
\psi([0], \cdot)=\phi_{\de}\circ \n(\hat{k}, j+k);\quad \psi([\frac{1}{3^j}], \cdot)=\bar{\phi}\circ \n(\hat{k}, \bar{k}),
\end{displaymath}
and Lemma \ref{1-homotopic lemma}(a)(c)(d), where in (d) $(y, x)\in [0, \frac{1}{3^j}](\hat{k}-j)_0\times I(1, \hat{k})_0$. Then we can extend $\psi$ to 
$\big([\frac{1}{3^j}, 1]\cap I(1, \hat{k})_0\big)\times I(1, \hat{k})_0$ trivially by letting $\psi(y, x)=\bar{\phi}\circ \n(\hat{k}, \bar{k})(x)$ for $(y, x)\in \big([\frac{1}{3^j}, 1]\cap I(1, \hat{k})_0\big)\times I(1, \hat{k})_0$.

Let $W_1$ be the set of vertices of $[0, \frac{1}{3^j}](\bar{k}-j)_0\times I(1, \bar{k})_0$ which belong to the $1$-skeleton of $[0, \frac{1}{3^j}]\times I(1, j)$ (think $[0, \frac{1}{3^j}]\cong I(1, 0)$), and define $\psi_0: W_1\rightarrow \Z_n(M)$ by:
\begin{displaymath}
\begin{split}
& \psi_0([0], \cdot)=\phi_\de\circ \n(\bar{k}, j+k);\quad  \psi_0([\frac{1}{3^j}], \cdot)=\bar{\phi};\\
& \psi_0(\cdot, x)\equiv \Phi(x),\quad \textrm{ for all } x\in I(1, j)_0.
\end{split}
\end{displaymath}
Then $\psi_0$ satisfies:
\begin{enumerate}
\vspace{-5pt}
\setlength{\itemindent}{0.5em}
\addtolength{\itemsep}{-0.7em}

\item $\f(\psi_0)\leq \max\{\f(\phi_{\de}), \f(\bar{\phi})\}\leq \de$, as $\phi_{\de}|_{I(1, j)_0}=\bar{\phi}|_{I(1, j)_0}=\Phi$;

\item Given any $2$-cell $\be$ in $[0, \frac{1}{3^j}]\times I(1, j)$, with $\be=[0, \frac{1}{3^j}]\otimes \al$, for some $\al \in I(1, j)_1$, then $\psi_0$ maps $\be_0(\bar{k}-j)_0$\footnote{Here $\be_0(\bar{k}-j)_0=\be_0\cap I(1, \bar{k})_0\times I(1, \bar{k})_0$.} into $\B^{\F}_{\rho_{c(\al)}(2\ep_{c(\al)})+3\ep_{c(\al)}}(T_{c(\al)})$, so $\psi_0\big(\be_0(\bar{k}-j)_0\big)\subset\B^{\F}_{\ep(T_{c(\al)}, \de, L)}(T_{c(\al)})$ by (\ref{smallness of radius of the covering});

$--$This is because $\Phi(\al)\subset \B^{\F}_{2\ep_{c(\al)}}(T_{c(\al)})$, $\phi_{\de}\big(\al(k)_0\big)\subset \B^{\F}_{\rho_{c(\al)}(2\ep_{c(\al)})}(T_{c(\al)})$, and $\bar{\phi}\big(\al(\bar{k}-j)_0\big)\subset \B^{\F}_{3\ep_{c(\al)}}(T_{c(\al)})$\footnote{Here $\al(\bar{k}-j)_0=\al\cap I(1, \bar{k})_0$.}.

\item $\psi_0(y, x)=\partial [[\Om_{y, x}]]$, $\Om_{y, x}\in\C(M)$, for all $(y, x)\in W_1$;

$--$This comes from property 2 of $\phi_{\de}$ and $\bar{\phi}$.

\item $\sup_{(y, x)\in W_1}\M\big(\psi_0(y, x)\big)\leq \max\big\{\sup_{I(1, j+k)_0}\M(\phi_{\de}), \sup_{I(1, \bar{k})_0}\M(\bar{\phi})\big\}\leq 2L-\de$;

$--$The last $``\leq"$ comes from property 4 of $\phi_{\de}$ and $\bar{\phi}$.
\end{enumerate}


Therefore we can apply Proposition \ref{technical 2} for each $2$-cell $\be=[0, \frac{1}{3^j}]\otimes \al$ in $[0, \frac{1}{3^j}]\times I(1, j)$ to extend $\psi_0|_{\be_0(\bar{k}-j)_0}$ to
\begin{displaymath}
\ti{\psi}_{0, \be}: \be(\bar{k}-j+k_{c(\al)})_0\rightarrow \Z_n(M),
\end{displaymath}
which satisfies:
\begin{itemize}
\vspace{-5pt}
\setlength{\itemindent}{1em}
\addtolength{\itemsep}{-0.7em}

\item[$(a)$] $\f(\ti{\psi}_{0, \be})\leq b(1)\big(\f(\psi_0)+\de\big)\leq 2b(1)\de$;

\item[$(b)$] $\ti{\psi}_{0, \be}([0], \cdot)=\psi_0([0], \cdot)\circ \n(\bar{k}-j+k_{c(\al)}, \bar{k}-j)=\phi_{\de}\circ\n(\bar{k}-j+k_{c(\al)}, k)$ on $\al(\bar{k}-j+k_{c(\al)})_0$,

and $\ti{\psi}_{0, \be}([\frac{1}{3^j}], \cdot)=\psi_0([\frac{1}{3^j}], \cdot)\circ \n(\bar{k}-j+k_{c(\al)}, \bar{k}-j)=\bar{\phi}\circ\n(\bar{k}-j+k_{c(\al)}, \bar{k}-j)$ on $\al(\bar{k}-j+k_{c(\al)})_0$;

\item[$(c)$]
\begin{displaymath}
\begin{split}
\sup_{\be(\bar{k}-j+k_{c(\al)})_0}\M(\ti{\psi}_{0, \be}) &\leq \sup_{\be_0(\bar{k}-j)_0}\M(\psi_0)+\de\leq \max\big\{\sup_{I(1, j+k)_0}\M(\phi_{\de}), \sup_{I(1, \bar{k})_0}\M(\bar{\phi})\big\}+\de\\
                                                                                   &\leq \sup_{x\in\al}\M\big(\Phi(x)\big)+2\de.
\end{split}
\end{displaymath}
\end{itemize}
Also given any two adjacent $2$-cells $\be=[0, \frac{1}{3^j}]\otimes \al$ and $\bar{\be}=[0, \frac{1}{3^j}]\otimes\bar{\al}$ in $[0, \frac{1}{3^j}]\times I(1, j)$, by Proposition \ref{technical 2}(\rom{2}), we know that $\ti{\psi}_{0, \be}\circ \n(\hat{k}-j, \bar{k}-j+k_{c(\al)})=\ti{\psi}_{0, \bar{\be}}\circ \n(\hat{k}-j, \bar{k}-j+k_{c(\bar{\al})})=\psi_0\circ\n(\hat{k}-j, \bar{k}-j)$ on $\be(\hat{k}-j)_0\cap \bar{\be}(\hat{k}-j)_0$, so we can put all $\{\ti{\psi}_{0, \be}\}$ together and construct the desired map
$$\psi: [0, \frac{1}{3^j}](\hat{k}-j)_0\times I(1, \hat{k})_0\rightarrow \Z_n(M)$$
by letting $\psi=\ti{\psi}_{0, \be}\circ \n(\hat{k}-j, \bar{k}-j+k_{c(\al)})$ on $\be(\hat{k}-j)_0$ for each $2$-cell $\be=[0, \frac{1}{3^j}]\otimes \al$. It is straightforward to check that $\psi$ satisfies the requirement.
\end{proof}

Now let us go back to finish the proof of Theorem \ref{discretization and identification}. Take a sequence of positive numbers $\{\de_i\}$, $\de_i\rightarrow 0$, as $i\rightarrow\infty$; then by Step \Rom{1}, we can construct a sequence of mappings $\{\phi_i\}$, with $\phi_i=\phi_{\de_i/c_0}: I(1, j_i+k_i)_0\rightarrow\Z_n(M)$\footnote{$c_0$ is given in Lemma \ref{1-homotopic lemma}(c).}. After extracting a subsequence, we can assume that $\phi_{i+1}$ is a $(\de_i, j_{i+1}+k_{i+1})$-extension of $\Phi$. Then we can apply Lemma \ref{1-homotopic lemma} to $\phi_i$ and $\phi_{i+1}$, so as to construct $\psi_i$ satisfying Theorem \ref{discretization and identification}(\rom{2}). The fact that $\phi_i$ satisfy Theorem \ref{discretization and identification}(\rom{1})(\rom{3})(\rom{4}) come from properties 4,5,6 of $\phi_{\de}$ in Part \Rom{1}.
\end{proof}


\section{Proof of the main theorem}\label{proof of the main theorem}

The main idea for proving Theorem \ref{main theorem1} is to apply the Almgren-Pitts min-max theory to the good families constructed in \S \ref{min-max family}, so that we can obtain an optimal minimal hypersurface satisfying the requirement. The idea is similar to the proof of \cite[Theorem 1.1]{Z12}, while we need a more delicate comparison argument when checking the min-max hypersurface has index one (c.f. Claim \ref{index bound claim}).

Given $\Si\in\mS$ (\ref{space of minimal surfaces}), we can define a mapping into $\big(\Z_n(M^{n+1}), \{0\}\big)$
\begin{equation}\label{Phi_Simga}
\Phi^{\Si}: [0, 1]\rightarrow \big(\Z_n(M^{n+1}), \{0\}\big)
\end{equation}
as follows:
\begin{itemize}
\vspace{-5pt}
\setlength{\itemindent}{1em}
\addtolength{\itemsep}{-0.7em}

\item[(\rom{1})] When $\Si\in\mS_+$ (\ref{S+}), let $\Phi^{\Si}(x)=\partial[[\Om_x]]$, where $\Om_x=\{p\in M: d^{\Si}_{\pm}(p)\leq (2x-1)d(M)\}$. Here $d^{\Si}_{\pm}$ is the signed distance function (\ref{signed distance}), and $d(M)$ is the diameter of $M$.

\item[(\rom{2})] When $\Si\in\mS_-$ (\ref{S-}), let $\Phi^{\Si}(x)=\partial[[\Om_x]]$, where $\Om_x=\{p\in M: d^{\Si}(p)\leq xd(M)\}$. Here $d^{\Si}$ is the distance function to $\overline{\Si}$.
\end{itemize}
By Proposition \ref{property of level surface 1} and Proposition \ref{property of level surface 2}, $\Phi^{\Si}$ satisfies:
\begin{proposition}\label{GMT properties of min-max family}
$\Phi^{\Si}: [0, 1]\rightarrow \big(\Z_n(M^{n+1}), \{0\}\big)$ is continuous under the flat topology, and
\begin{itemize}
\vspace{-5pt}
\setlength{\itemindent}{1em}
\addtolength{\itemsep}{-0.7em}

\item[$(a)$] $\Phi^{\Si}(x)=\partial[[\Om_x]]$, $\Om_x\in\C(M)$ for all $x\in [0, 1]$, and $\Om_0=\emptyset$, $\Om_1=M$;

\item[$(b)$] $\sup_{x\in[0, 1]}\M\big(\Phi^{\Si}(x)\big)=\mH^n(\Si)$, if $\Si\in\mS_+$;

\item[$(c)$] $\sup_{x\in[0, 1]}\M\big(\Phi^{\Si}(x)\big)=2\mH^n(\Si)$, if $\Si\in\mS_-$.
\end{itemize}
\end{proposition}
\begin{remark}
Notice that $\Phi^{\Si}$ satisfies the requirement to apply Theorem \ref{discretization and identification}.
\end{remark}


We need one more elementary fact about min-max hypersurface, which is well-known to experts. A proof is included for completeness.
\begin{lemma}\label{tangent cone multiplicity one}
Let $\{\Si_i\}_{i=1}^l$ be the singular minimal hypersurfaces given in Theorem \ref{AP min-max theorem}, then the associated integral varifolds $[\Si_i]$ have tangent cones with multiplicity one everywhere. Therefore $\Si_i\in \mS$.
\end{lemma}
\begin{proof}
Given any point $p\in\overline{\Si}_i$, then $\overline{\Si}_i$ is stable (c.f. \cite[2.3]{P81}\cite[(5)(6)]{I96}) in any small annuli neighborhood of $p$ by \cite[3.3]{P81}. A standard cutoff  argument implies that $\overline{\Si}_i$ is stable near $p$, and \cite[Theorem B]{I96} implies that every tangent cone of $[\Si_i]$ has multiplicity one. 
\end{proof}

\vspace{0.5em}
\begin{proof}
(of Theorem \ref{main theorem1}) Given $\Si\in\mS$ and $\Phi^{\Si}$ (\ref{Phi_Simga}), we can apply Theorem \ref{discretization and identification} to $\Phi^{\Si}$ and get a $(1, \M)$-homotopy sequence $S^{\Si}=\{\phi^{\Si}_i\}_{i\in\N}$ into $\big(\Z_n(M^{n+1}, \M), \{0\}\big)$. By (\ref{max mass control of discrete family}) and Proposition \ref{GMT properties of min-max family},
\begin{equation}\label{min-max bound}
\left. \bL(\{\phi^{\Si}_{i}\}_{i\in\N})\leq \Big\{ \begin{array}{ll}
\mH^n(\Si), \quad \textrm{ if $\Si\in\mS_+$};\\
2\mH^n(\Si), \quad \textrm{ if $\Si\in\mS_-$}.
\end{array}\right. 
\end{equation}
Also by Theorem \ref{discretization and identification}(\rom{4}), $S^{\Si}\in F_A^{-1}\big([[M]]\big)\in\pi^{\#}_1\big(\Z_n(M^{n+1}, \M), \{0\}\big)$. Denote $F_A^{-1}\big([[M]]\big)$ by $\Pi_M$. By Theorem \ref{AP min-max theorem}, $\bL(\Pi_M)>0$. Using (\ref{min-max bound}), we have that
$$\bL(\Pi_M)\leq A_M,$$
where $A_M$ is defined in (\ref{definition of A_M}).

The Min-max Theorem \ref{AP min-max theorem} applied to $\Pi_M$ gives a stationary varifold $V=\sum_{i=1}^l m_i[\Si_i]$, with $m_i\in\N$ and $\{\Si_i\}$ a disjoint collection of minimal hypersurfaces in $\mS$, such that $\bL(\Pi_M)=\|V\|(M)=\sum_{i=1}^lm_i\mH^n(\Si_i)$. Notice that there is only one connected component, denoted by $\Si_A$, by Theorem \ref{generalized Frankel theorem} as $M$ has positive Ricci curvature, i.e. $V=m[\Si_A]$ for some $m\in\N$, $m\neq 0$. Therefore
\begin{equation}
\left. m\mH^n(\Si_A)=\bL(\Pi_M)\leq A_M\leq \Big\{ \begin{array}{ll}
\mH^n(\Si_A), \quad \textrm{ if $\Si_A\in\mS_+$};\\
2\mH^n(\Si_A), \quad \textrm{ if $\Si_A\in\mS_-$},
\end{array}\right. 
\end{equation}
where the last $``\leq"$ follows from the definition (\ref{definition of A_M}) of $A_M$. Thus we have the following two cases:
\begin{itemize}
\vspace{-5pt}
\setlength{\itemindent}{2em}
\addtolength{\itemsep}{-0.7em}

\item[{\bf Case 1}:] If $\Si_A\in\mS_+$, orientable, then $m\leq 1$, so $m=1$, and $\mH^n(\Si_A)=A_M$;

\item[{\bf Case 2}:] If $\Si_A\in\mS_-$, non-orientable, then $m\leq 2$, so $m=1$ or $m=2$.
\end{itemize}

In Case 1 when $\Si_A\in\mS_+$, to prove Theorem \ref{main theorem1}(\rom{1}), we only need to show
\begin{claim}\label{index bound claim}
In this case, $\Si_A$ has Morse index one.
\end{claim}
Assume that the claim is false, i.e. the index of $\Si_A$ is no less than $2$. By Definition \ref{Morse index for hypersurfaces with singularities}, there exists an open set $\Om\subset \Si_A$ with smooth boundary, such that $\Ind(\Om)\geq 2$. Then we can find two nonzero $L^2$-orthonormal eigenfunctions $\{v_1, v_2\}\subset C^{\infty}_0(\Om)$ of the Jacobi operator $L_{\Si_A}$ with negative eigenvalues. A linear combination will give a $v_3\in C^{\infty}_0(\Om)$, such that
\begin{equation}\label{orthogonality condition}
\int_{\Om}v_3\cdot L_{\Si_A}1d\mu=\int_{\Om}1\cdot L_{\Si_A}v_3=0, \quad v_3\neq 0.
\end{equation}
We can assume that $\Om=U\cap \Si_A$ for some open set $U\subset M\backslash sing(\Si_A)$. Let $\ti{X}=v_3\nu$ with $\nu$ the unit normal of $\Si_A$, and extend it to a tubular neighborhood of $\Si_A$, such that $\ti{X}$ has compact support in $\bar{U}$. Let $\{\ti{F}_s\}_{s\in[-\ep, \ep]}$ be the flow of $\ti{X}$, and denote $\Si_{t, s}=\ti{F}_s(\Si_t)$, where $\{\Si_t\}$ is the family associated to $\Si_A$ as in Proposition \ref{property of level surface 1}. Notice that $\Si_{t, s}=\Si_t$ outside $\overline{U}$, and $\{\Si_{t, s}\lc U\}_{(s, t)\in[-\ep, \ep]\times[-\ep, \ep]}$ is a smooth family for small $\ep$ by Proposition \ref{property of level surface 1}(c). Denote $\ti{f}(t, s)=\mH^n(\Si_{t, s}\cap \overline{U})$. Then $\nabla\ti{f}(0, 0)=0$ (by minimality of $\Si_A$), $\frac{\partial^2}{\partial t\partial s}\ti{f}(0, 0)=-\int_{\Om}v_3L_{\Si_A}1d\mu=0$ (by (\ref{orthogonality condition})), $\frac{\partial^2}{\partial t^2}\ti{f}(0, 0)=-\int_{\Om}1\cdot L_{\Si_A}1 d\mu<0$ (by $Ric_g>0$), and $\frac{\partial^2}{\partial s^2}\ti{f}(0, 0)=-\int_{\Om}v_3L_{\Si_A}v_3 d\mu<0$ (as $v_3$ is a linear combination of eigenfunctions of $L_{\Si_0}$ with negative eigenvalues). 

Now consider $\mH^n(\Si_{t, s})=\mH^n(\Si_{t, s}\cap\overline{U})+\mH^n(\Si_{t, s}\backslash \overline{U})=\ti{f}(t, s)+\mH^n(\Si_t\backslash\overline{U})$. For $(t, s)\in[-\ep, \ep]\times[-\ep, \ep]$, $s\neq 0$, with $\ep$ small enough, by Taylor expansion,
\begin{displaymath}
\begin{split}
\mH^n(\Si_{t, s}) &=\ti{f}(t, 0)+\frac{\partial}{\partial s}\ti{f}(t, 0)s+\frac{\partial^2}{\partial s^2}\ti{f}(t, 0)s^2+o(s^2)+\mH^n(\Si_t\backslash\overline{U})\\
                           &=\ti{f}(t, 0)+\big\{\frac{\partial}{\partial s}\ti{f}(0, 0)+\frac{\partial^2}{\partial t\partial s}\ti{f}(0, 0)t+o(t)\big\}s+\frac{\partial^2}{\partial s^2}\ti{f}(t, 0)s^2+o(s^2)+\mH^n(\Si_t\backslash\overline{U})\\
                           &=\ti{f}(t, 0)+\frac{\partial^2}{\partial s^2}\ti{f}(t, 0)s^2+o(ts+s^2)+\mH^n(\Si_t\backslash\overline{U})\\
                           &<\ti{f}(t, 0)+\mH^n(\Si_t\backslash\overline{U})\\
                           &=\mH^n(\Si_t)\leq\mH^n(\Si_A),
\end{split}
\end{displaymath}
where the fourth $``<"$ follows from the fact that $\frac{\partial^2}{\partial s^2}\ti{f}(t, 0)<0$ for $t$ small enough (as $\frac{\partial^2}{\partial s^2}\ti{f}(0, 0)<0$). For $|t|\geq \ep$, as $\mH^n(\Si_t)<\mH^n(\Si_A)$, we can find $\de>0$, $\de\leq \ep$ small enough, such that $\mH^n(\Si_{t, \de})<\mH^n(\Si_A)$. In summary,
$$\max\{\mH^n(\Si_{t, \de}): -d(M)\leq t\leq d(M)\}<\mH^n(\Si_A).$$
As $\{\Si_{t, \de}\}$ are deformed from $\{\Si_t\}$ by the ambient isotopy $\ti{F}_{\de}: M\rightarrow M$, we can associate it with a mapping $\Phi_{\de}: [0, 1]\rightarrow \big(\Z_{n}(M^{n+1}, \F), \{0\}\big)$ as in (\ref{Phi_Simga})(\rom{1}), such that
\begin{itemize}
\vspace{-5pt}
\setlength{\itemindent}{2em}
\addtolength{\itemsep}{-0.7em}

\item $\max_{x\in[0, 1]}\M\big(\Phi_{\de}(x)\big)=\max_{t}\mH^n(\Si_{t, \de})<\mH^n(\Si_A)=\bL(\Pi_M)$;

\item $\Phi_{\de}(x)=\partial[[\ti{\Om}_x]]$, $\ti{\Om}_x=\ti{F}_{\de}(\Om_x)\in\C(M)$, for all $x\in[0, 1]$.
\end{itemize}
Applying Theorem \ref{discretization and identification} to $\Phi_{\de}$ gives a $(1, \M)$-homotopy sequence $S_{\de}=\{\phi^{\de}_i\}_{i\in\N}$, such that $S_{\de}\in\Pi_M$, and
$$\bL(S_{\de})\leq \max_{x\in[0, 1]}\M\big(\Phi_{\de}(x)\big)<\bL(\Pi_M),$$
which is a contradiction to the definition of $\bL(\Pi_M)$ (\ref{width}). So we finish the prove of Claim \ref{index bound claim} and hence Theorem \ref{main theorem1}(\rom{1}).

In Case 2 when $\Si_A\in\mS_-$. By Proposition \ref{orientation and multiplicity result}, $m$ must be an even number. Hence $m=2$, and $2\mH^n(\Si_A)=A_M$. To prove Theorem \ref{main theorem1}(\rom{2}), we only need to show
\begin{claim}\label{index bound claim2}
In this case, $\Si_A$ is stable, i.e. $\Ind(\Si_A)=0$.
\end{claim}
The proof is similar to Claim \ref{index bound claim}. If the claim is false, then there exists an open set $\Om\subset \Si_A$ with smooth boundary, such that $\Ind_D(\Om)\geq 1$. Denote $\ti{\Si}_A$ by the orientable double cover of $\Si_A$, and $\ti{\Om}$ the lift-up of $\Om$; then there exists an anti-symmetric eigenfunction $\ti{\phi}\in C^{\infty}_0(\ti{\Om})$ of the Jacobi operator $L_{\ti{\Si}_A}$ of $\ti{\Si}_A$ with negative eigenvalue (c.f. \S \ref{Orientation, second variation and Morse index}). The anti-symmetric condition directly implies that:
\begin{equation}\label{orthogonality condition 2}
\int_{\ti{\Om}}\ti{\phi}\cdot L_{\ti{\Si}_A}1d\mu=\int_{\ti{\Om}}1\cdot L_{\ti{\Si}_A}\ti{\phi} d\mu=0.
\end{equation}
Let $\ti{\nu}$ be the unit normal of $\ti{\Si}_A$, and $\pi:\ti{\Si}_A\rightarrow \Si_A$ the covering map. The anti-symmetric condition of $\ti{\phi}$ implies that $\ti{\phi}\ti{\nu}$ is symmetric on $\ti{\Si}_A$ (c.f. \S \ref{Orientation, second variation and Morse index}). Hence denote $\ti{X}=\pi_{*}(\ti{\phi}\ti{\nu})$ by the push-forward of $\ti{\phi}\ti{\nu}$ to $\Si_A$ under $\pi$. 
Similarly as above, extend $\ti{X}$ to a neighborhood of $\Si_A$, and denote $\{\ti{F}_s\}_{s\in[-\ep, \ep]}$ by the flow associated to $\ti{X}$. Let $\{\Si_t\}$ be the family associated to $\Si_A$ by Proposition \ref{property of level surface 2}, where we assume that $\Si_0$ is a double cover of $\Si_A$; then $\{\Si_t\}_{t\in[0, \ep]}$ is a smooth family away from $sing(\Si_A)$ for small $\ep$ by Proposition \ref{property of level surface 2}(c). Let $\Si_{t, s}=\ti{F}_s(\Si_t)$; then $\Si_{t, s}$ are deformations of $\Si_t$ away from $sing(\Si_A)$ by ambient isotopies. By similar argument as in Claim \ref{index bound claim} using (\ref{orthogonality condition 2}) instead of (\ref{orthogonality condition}), we can find $\de>0$ small enough, such that
$$\max\{\mH^n(\Si_{t, \de}): 0\leq t\leq d(M)\}<2\mH^n(\Si_A).$$
Then we can get a contradiction by discretizing the family $\{\Si_{t, \de}\}$ in the same way. Now we finish the proof.
\end{proof}


\section{Appendix}

\subsection{Reverse statement of Proposition \ref{technical 1}.}\label{reverse of technical 1}
Now we list the detailed argument to get the reverse statement of Proposition \ref{technical 1} used in the proof. In fact, Proposition \ref{technical 1} has another equivalent formulation as follows:

\begin{proposition}\label{technical 1 version 2}
Given $\de, L, l, m, T$ as in Proposition \ref{technical 1}, there exists $k=k(l, m, T, \de, L)\in\N$, such that for any $\rho>0$, there exists a $\ep=\ep(\rho, l, m, T, \de, L)>0$, such that for any $0<s<\ep$, and $\phi$ as in (\ref{phi to be extended}), there exists
$$\ti{\phi}: I(m, k)_0\rightarrow \B^{\F}_{\rho}(T),\ \textrm{with }\ti{\phi}(y)=\partial[[\Om_y]],$$
$\Om_y\in\C(M)$, 
$y\in I(m, l)_0$, and satisfying $(\rom{1})(\rom{2})(\rom{3})(\rom{4})$ in Proposition \ref{technical 1}.
\end{proposition}
Now we show that this formulation implies Proposition \ref{technical 1}. In fact, under the assumption in the above proposition, we can fix an $\rho_0=1>0$, and take $\ep=\ep(\rho_0, l, m, T, \de, L)$. Given $0<s<\ep$, and $\phi$ as in (\ref{phi to be extended}), we can define
\begin{equation}\label{rho_phi}
\begin{split}
\rho_{\phi, s}=\inf\{ \rho:\ & \textrm{$\exists\ \rho>0$, and } \ti{\phi}: I(m, k)_0\rightarrow \B^{\F}_{\rho}(T),\ \textrm{with }\ti{\phi}(y)=\partial[[\Om_y]], \Om_y\in\C(M),\\
                              &\textrm{satisfying }(\rom{1})(\rom{2})(\rom{3})(\rom{4})\textrm{ in Proposition \ref{technical 1}}\}.
\end{split}
\end{equation}
$\rho_{\phi, s}$ is well-defined since $\rho_0$ belongs to the above set, and $0\leq \rho_{\phi, s}\leq \rho_0$. Now define the function $\rho: [0, \ep)\rightarrow\R^1_+$,
$$\rho(s)=2\sup\{\rho_{\phi, s}:\ \phi \textrm{ is any map as in } (\ref{phi to be extended})\}.$$
$\rho(s)$ is well-defined, as $\rho(s)\leq 2\rho_0$. Also from the definition, the function $\rho$ depends only on $l, m, T, \de, L$. 
\begin{claim}
$\rho(s)\rightarrow 0$, as $s\rightarrow 0$.
\end{claim}
\begin{proof}
For any $\si>0$ small enough, by Proposition \ref{technical 1 version 2} we can find $\ep_{\si}=\ep(\si, l, m, T, \de, L)>0$, so that if $0<s<\ep_{\si}$, then every $\phi$ as in (\ref{phi to be extended}) can be extended to $\ti{\phi}: I(m, k)_0\rightarrow \B^{\F}_{\si}(T)$ satisfying the requirement as in (\ref{rho_phi}); hence $\rho_{\phi, s}\leq \si$, and $\rho(s)\leq 2\si$ by definition.
\end{proof}
\noindent By taking $k, \ep, \rho(s)$ as above, Proposition \ref{technical 1 version 2} implies Proposition \ref{technical 1}. The reverse is trivial.

\vspace{0.5em}
To get the reverse statement of Proposition \ref{technical 1}, we can use the reverse statement of Proposition \ref{technical 1 version 2}.

\subsection{Some basic facts of exponential map.}\label{basic facts of exponential map}
Here we collect a few basic facts about exponential maps summarized in \cite[\S 3.4]{P81} that we need to use for the discretization procedure in Lemma \ref{case 2}. We will use the following notions:
\begin{itemize}
\setlength{\itemindent}{0.5em}
\addtolength{\itemsep}{-0.6em}
\item $r_p(\cdot)$ denotes the distance function of $M^{n+1}$ to $p\in M$, and $B(p, r)$ denotes the closed ball centered at $p$ of radius $r$ in $M$;

\item Given $\la\geq 0$, $\mu(\la): \R^{n+1}\rightarrow \R^{n+1}$ denotes the scaling map by: $\mu(\la): x\rightarrow \la x$;

\item Given a map $f: (W, g_1)\rightarrow (Z, g_2)$, $Lip(f)$ denotes the Lipschitz constant with respect the metrics $g_1, g_2$.
\end{itemize}

Given $p\in M$, let $exp_p: T_p M\cong\R^{n+1}\rightarrow M$ be the exponential map. First, let us list several basic facts in \cite[\S 3.4(4)]{P81}.  Given $q\in M$, and $0<\ep<1$, there exists a neighborhood $Z\subset M$ of $q$, such that, if $p\in Z$, $W=exp_p^{-1}(Z)\subset T_p M\cong\R^{n+1}$, and $E=exp_p|_W$, then the following properties hold:
\begin{itemize}
\setlength{\itemindent}{0.5em}
\addtolength{\itemsep}{-0.7em}
\item[$(a)$] $E$ is a $C^2$ diffeomorphism onto $Z$;

\item[$(b)$] $Z$ is strictly geodesic convex;

\item[$(c)$] $(Lip E)^n(Lip E^{-1})^n\leq 2$;

\item[$(d)$] $Lip(r_p|_Z)\leq 2$;

\item[$(e)$] If $x\in Z$ and $0\leq \la\leq 1$, then $E\circ \mu(\la)\circ E^{-1}(x)\in Z$;

\item[$(f)$] if $x\in Z$, $0\leq \la\leq 1$, and $v\in \La_n T_x M$ ($n$-th wedge product of $T_xM$ \cite[\S 25]{Si83}), then
$$\|D\big( E\circ \mu(\la)\circ E^{-1}\big)_* v\|\leq \la^n\big(1+\ep(1-\la)\big)\|v\|.$$
Also $\la^n\big(1+\ep(1-\la)\big)\leq 1$ for all $0\leq \la\leq 1$, $\ep<n/2$.
\end{itemize}

Now we list a few facts about scaling of currents in Euclidean spaces as in \cite[\S 3.4(5)(6)(7)]{P81}. Given $r>0$, $0\leq \la\leq 1$, denote $B(0, r)$ by the closed ball of radius $r$ in $\R^{n+1}$, and $T\in \Z_{n-1}(\partial B(0, r))$, then we can define the cone of $T$ over the annulus $A(0, \la r, r)=B(0, r)\backslash B(0, \la r)$ as \cite[26.26]{Si83}
$$S=\de_0\ttimes(T-\mu(\la)_{\#}T)\in\Z_n(\R^{n+1}),$$
then
\begin{itemize}
\setlength{\itemindent}{0.5em}
\addtolength{\itemsep}{-0.7em}
\item[$(g)$] $\partial S=T-\mu(\la)_{\#}T$;

\item[$(h)$] $\M(S)=r n^{-1}(1-\la^n)\M(T)$;

\item[$(i)$] $spt(S)\subset A(0, \la r, r)$, where $spt(S)$ is the support of $S$ \cite[26.11]{Si83}.
\end{itemize}
Given $\la\geq 0$, and $T\in \bI_n(\R^{n+1})$, then it is easily seen that
$$\M(\mu(\la)_{\#}T)=\la^n\M(T).$$
Using notions as above, 
\begin{itemize}
\setlength{\itemindent}{0.5em}
\addtolength{\itemsep}{-0.7em}
\item[$(j)$] Given $r>0$, $0\leq \la\leq 1$, $B(p, r)\subset Z$, and $T\in \Z_n\big(B(p, r), \partial B(p, r)\big)$, then by (f),
\begin{equation}\label{mass change under scaling}
\M\big((E\circ\mu(\la)\circ E^{-1})_{\#}T\big)\leq \la^n(1+\ep(1-\la))\M(T)\leq \M(T);
\end{equation}

\item[$(k)$] Denote $S_{\la}=E_{\#}\big(\de_0\ttimes\big[E^{-1}_{\#}(\partial T)-(\mu(\la)\circ E^{-1})_{\#}(\partial T)\big]\big)$, then by (g)(h)(i),
$$\partial S_{\la}=\partial T-\partial\big[(E\circ \mu(\la)\circ E^{-1})_{\#}T\big],\ spt(S_{\la})\subset A(p, \la r, r)=\overline{B(p, r)\backslash B(p, \la r)},$$
\begin{equation}\label{mass of cone}
\M(S_{\la})\leq (Lip E)^n(Lip E)^{-n}r n^{-1}(1-\la^n)\M(\partial T)\leq 2 r n^{-1}(1-\la^n)\M(\partial T).
\end{equation}
\end{itemize}

Finally let us recall the contraction map in \cite[\S 3.4(8)]{P81}. For $r>0$, define
$$h(r): \R^{n+1}\rightarrow \R^{n+1},$$
by $h(r)(x)=x$ if $|x|\leq r$, and $h(r)(x)=r|x|^{-1}x$ if $|x|>r$. If $V\in\V_n(\R^{n+1})$, then
\begin{itemize}
\setlength{\itemindent}{0.5em}
\addtolength{\itemsep}{-0.7em}
\item[$(l)$] $spt(h(r)_{\#}V)\subset B(0, r)$;

\item[$(m)$] $(h(r)_{\#}V)\lc G_n\big(B_0(0, r)\big)=V\lc G_n\big(B_0(0, r)\big)$\footnote{$B_0(0, r)$ denotes the open ball of radius $r$ in $\R^{n+1}$.};

\item[$(n)$] $\M(h(r)_{\#}V)\leq \M(V)$.
\end{itemize}

\subsection{Isoperimetric choice.}\label{lemma on isoperimetric choice}
We refer the notions to \S \ref{Almgren isomorphism}.
\begin{lemma}\label{lemma on isoperimetric choice}
Given $T_1, T_2=\Z_n(M^{n+1})$, with $\F(T_1, T_2)\leq \nu_M$, assume that $T_1=\partial[[\Om_1]]$, $T_2=\partial[[\Om_2]]$, $\Om_1, \Om_2\in\C(M)$, and $\M\big([[\Om_2]]-[[\Om_1]]\big)<vol(M)/2$, then the isoperimetric choice of $T_2-T_1$ is $[[\Om_2-\Om_1]]$.
\end{lemma}
\begin{proof}
Let $Q\in \bI_{n+1}(M)$ be the isoperimetric choice of $T_2-T_1$, then $\M(Q)=\F(T_1, T_2)\leq \M\big([[\Om_2-\Om_1]]\big)$, and $\partial Q=T_2-T_1$. As $T_2-T_1=\partial[[\Om_2-\Om_1]]$, $\partial\big(Q-[[\Om_2-\Om_1]]\big)=0$ in $\bI_{n+1}(M^{n+1})$. The Constancy Theorem \cite[26.27]{Si83} implies that $Q-[[\Om_2-\Om_1]]=n[[M]]$ for some $n\in\mathbb{Z}$. But $\M\big(Q-[[\Om_2-\Om_1]]\big)\leq \M(Q)+\M([[\Om_2-\Om_1]])\leq 2\M([[\Om_2-\Om_1]])<vol(M)$, hence $n=0$, and $Q=[[\Om_2-\Om_1]]$.
\end{proof}

We will also need a more subtle technical lemma concerning the isoperimetric choice.
\begin{lemma}\label{lemma on isoperimetric choice 2}
Given $T_1, T_2$ as above, with $T_1\neq 0$, there exists $\de>0$ (depending on $T_1$), such that if $\F(T_1, T_2)\leq \de$, then the isoperimetric choice of $T_2-T_1$ is $[[\Om_2-\Om_1]]$.
\end{lemma}
\begin{proof}
We use the same notions as in the proof of the above Lemma.

$T_1\neq 0$ implies that $\Om_1\neq \emptyset$ and $\Om_1\neq M$. Take
$$\de=\frac{1}{2}\min\{\mH^{n+1}(\Om_1), \mH^{n+1}(M\backslash\Om_1)\}.$$
Then $0<\de<vol(M)/2$. As we always assume that $\Om_1$, $\Om_2$ have the same orientation as $M$, hence $\M\big([[\Om_2]]-[[\Om_2]]\big)=\mH^{n+1}(\Om_1\lap\Om_2)$, where $\Om_1\lap\Om_2$ is the symmetric difference, i.e. $\Om_1\lap\Om_2=(\Om_1\backslash\Om_2)\cup(\Om_2\backslash\Om_1)$. Let $Q$ be the isoperimetric choice of $T_2-T_1$, by the above proof $Q-[[\Om_2-\Om_1]]=n[[M]]$. If $n=0$, the proof is done. If $n\neq 0$, then $|n|vol(M)=\M(Q-[[\Om_2-\Om_1]])\leq \M(Q)+\M([[\Om_2-\Om_1]])\leq \F(T_1, T_2)+\mH^{n+1}(\Om_1\lap\Om_2)\leq \de+vol(M)<2vol(M)$, hence $n=\pm1$. If $n=1$, then $Q=[[M]]+[[\Om_2-\Om_1]]=[[M-\Om_1]]+[[\Om_2]]$; hence $\M(Q)\geq \mH^{n+1}(M\backslash\Om_1)>\de$ (as $M-\Om_1$ has the same orientation as $\Om_2$), a contradiction. If $n=-1$, then $-Q=[[M]]-[[\Om_2-\Om_1]]=[[M-\Om_2]]+[[\Om_1]]$; hence $\M(Q)\geq \mH^{n+1}(\Om_1)>\de$ (as $M-\Om_2$ has the same orientation as $\Om_1$), a contradiction. 
\end{proof}




{\em Current address}:

Department of Mathematics, South Hall 6501, University of California Santa Barbara, Santa Barbara, CA 93106, USA; zhou@math.ucsb.edu\\

{\em Old address}:

Department of Mathematics, Massachusetts Institute of Technology, 77 Massachusetts Avenue, Cambridge, MA 02139, USA; xinzhou@math.mit.edu

\end{document}